\documentclass[12pt]{amsart} 

\calclayout 

\usepackage[utf8]{inputenc} 
\usepackage{mathrsfs}
\usepackage{amsmath} 
\usepackage{amsthm} 
\usepackage{amssymb}
\usepackage{enumerate} 
\usepackage{tikz-cd} 
\usepackage[colorlinks,allcolors=black]{hyperref} 
\usepackage[nameinlink]{cleveref} 
\usepackage[url=false,eprint=false]{biblatex} 
\numberwithin{equation}{subsection} 
\newtheorem{theorem}[equation]{Theorem} 
\newtheorem{lemma}[equation]{Lemma}
\newtheorem{proposition}[equation]{Proposition}
\newtheorem{corollary}[equation]{Corollary}

\theoremstyle{definition} 

\newtheorem{remark}[equation]{Remark}
\newtheorem{notation}[equation]{Notation}
\newtheorem{example}[equation]{Example}

\newtheorem{construction}[equation]{Construction}

\newcommand{\CC}{\mathbb{C}}

\newcommand{\ZZ}{\mathbb{Z}}

\newcommand{\tildeboxtimes}{\operatorname{\widetilde{\boxtimes}}}
\newcommand{\tildetimes}{\operatorname{\tilde{\times}}}
\newcommand{\tripslash}{/ \!\! / \!\! /}
\newcommand{\doubslash}{/ \! /}

\newcommand{\Sym}{\operatorname{Sym}}

\newcommand{\Gr}{\mathrm{Gr}}

\newcommand{\reg}{\mathrm{reg}}
\newcommand{\gen}{\mathrm{gen}}
\newcommand{\F}{\mathcal{F}}

\title{Levi-Equivariant Restriction of Spherical Perverse Sheaves}
\author{Mark Macerato}
\date{}

\begin{document} 

\begin{abstract}
    We study the equivariant cohomology of spherical perverse sheaves on the affine Grassmannian of a connected reductive group $G$ with support in the affine Grassmannian of any Levi subgroup $L$ of $G$. In doing so, we extend the work of Ginzburg and Riche on the $T$-equivariant cofibers of spherical perverse sheaves. We obtain a description of this cohomology in terms of the Langlands dual group $\check{G}$. More precisely, we identify the cohomology of the regular sheaf on $\Gr_G$ with support along $\Gr_L$ with the algebra of functions on a hyperspherical Hamiltonian $\check{G}$-variety $T^*(\check{G}/(\check{U}, \psi_L))$, where the \textit{Whittaker datum} $\psi_L$ is an additive character (determined by $L$) of the maximal unipotent subgroup $\check{U}$.
\end{abstract}

\maketitle 
\tableofcontents
\section{Introduction}

\subsection{Geometric Satake}\label{Introduction}
Let $G$ denote a connected reductive group over $\CC$. Following Lusztig, we associate to $G$ its \textit{affine Grassmannian} $\Gr_G$ \cite[§11]{lusztig1983singularities}, an ind-projective ind-variety\footnote{We will not distinguish here between $\Gr_G$ and its reduction $\Gr_G^{\mathrm{red}}$, as our considerations are purely topological.} over $\CC$ (we refer the reader to \cite[§1.2]{acharriche2023}, \cite[§9.1]{achar2020}, \cite[§1]{zhu2016introduction}, and \cite[§3]{baumann2018notes} for its basic properties). The topology of $\Gr_G$ gives rise to the Langlands dual group $\check{G}$ through the \textit{geometric Satake equivalence} of Ginzburg \cite{Gi95} and Mirkovi\'c-Vilonen \cite{MV04}. More precisely, let $G(\mathcal{O}) = G(\CC[[t]])$ denote the arc-group\footnote{We abuse notation here and throughout the paper by writing $G(\mathcal{O})$ in place of the arc-group scheme $L^+G$ \cite[Definition 1]{Faltings2003}, whose group of $\CC$-points is $G(\mathcal{O})$. } of $G$. The pro-algebraic group $G(\mathcal{O})$ acts naturally on $\Gr_G$, and we may consider the $G(\mathcal{O})$-equivariant constructible derived category\footnote{See §\ref{sheaf conventions} for our notation and conventions concerning categories of sheaves.} $D_{G(\mathcal{O})}(\Gr_G)$ of sheaves of complex vector spaces\footnote{Of course, one of the most profound aspects of the approach of \cite{MV04} is that the geometric Satake equivalence holds with essentially arbitrary (in particular, integral) coefficients. We will not need this generality here; our reliance on the equivariant localization theorem in the body of the paper precludes our arguments from extending to the integral setting.} on $\Gr_G$. The triangulated category $D_{G(\mathcal{O})}(\Gr_G)$ carries a natural monoidal structure $\star$ given by the \textit{convolution product} of sheaves \cite[§4]{MV04}. Lusztig's ``miraculous theorem'' \cite[§5.3.6]{beilinsondrinfeld} (reproven conceptually by Gaitsgory \cite[Proposition 6]{gaitsgory2000construction}) asserts that the convolution product $- \star -$ is $t$-exact, and therefore restricts to an exact bifunctor on the heart of the perverse $t$-structure $\mathcal{P}_{G(\mathcal{O})}(\Gr_G) \subseteq D_{G(\mathcal{O})}(\Gr_G)$, the abelian category of $G(\mathcal{O})$-equivariant perverse sheaves on $\Gr_G$. 

The total cohomology functor 
$$
H^*(\Gr_G,-): \mathcal{P}_{G(\mathcal{O})}(\Gr_G) \rightarrow \mathrm{Vect}_\CC
$$
admits a natural monoidal structure. Ginzburg's approach to this fact in \cite{Gi95} is to observe that the equivariant cohomology functor $H_{G(\mathcal{O})}^*(\Gr_G, -)$ has an obvious monoidal structure (see, for example, \cite[Proposition 5.2.3]{zhu2016introduction}), which induces a monoidal structure on $H^*$ through the canonical isomorphism $H^*(\Gr_G,-) \simeq H_{G(\mathcal{O})}^*(\Gr_G, -) \otimes_{H_{G(\mathcal{O})}^*(\mathrm{pt}, \CC)} H^*(\mathrm{pt}, \CC)$ (which arises from the equivariant formality of all objects of $\mathcal{P}_{G(\mathcal{O})}(\Gr_G)$; compare to Proposition~\ref{equivariantformality} below). See \cite[§3]{MV04} for the approach of Mirkovi\'{c}-Vilonen, which works for general coefficients. One moreover shows that $H^*(\Gr_G,-)$ is exact and conservative \cite[Corollary 3.7]{MV04}. 

Through a global reinterpretation of the convolution product \cite[§5-6]{MV04}, Mirkovi\'{c}-Vilonen (following the ideas of Beilinson-Drinfeld \cite{beilinsondrinfeld}) equip the monoidal abelian category $\mathcal{P}_{G(\mathcal{O})}(\Gr_G)$ with a symmetric braiding, the \textit{commutativity constraint}. They show \cite[Proposition §6.3]{MV04} that with respect to this symmetric monoidal structure on $\mathcal{P}_{G(\mathcal{O})}(\Gr_G)$, the cohomology functor $H^*(\Gr_G, -)$ is symmetric monoidal.   

\begin{theorem}[{Ginzburg \cite{Gi95}, Mirkovi\'c-Vilonen \cite{MV04}}] \label{The Geometric Satake Equivalence}
    There exists a canonically defined connected reductive group $\check{G}$ over $\CC$ equipped with an equivalence of symmetric monoidal abelian categories
    $$
    \mathbb{S}_G: \mathcal{P}_{G(\mathcal{O})}(\Gr_G) \simeq \operatorname{Rep}(\check{G})
    $$
    and a natural isomorphism of symmetric monoidal functors 
    $$
    \mathrm{For} \circ \mathbb{S}_G \simeq H^*(\Gr_G,-): \mathcal{P}_{G(\mathcal{O})}(\Gr_G) \rightarrow \mathrm{Vect}_\CC,
    $$
    where $\mathrm{For}: \operatorname{Rep}(\check{G}) \rightarrow \mathrm{Vect}_\CC$ is the forgetul functor from the category of finite dimensional $\check{G}$-modules to the category of vector spaces. Moreover, $\check{G}$ is equipped with a canonically defined Borel subgroup $\check{B} \subseteq \check{G}$ and maximal torus $\check{T} \subseteq \check{B}$. The root datum $(X^*(\check{T}), X_*(\check{T}), \Phi, \check{\Phi})$ of the triple $\check{T} \subseteq \check{B} \subseteq \check{G}$ is dual to the abstract root datum\footnote{See the discussion \cite[§1.1]{delignelusztig} of Deligne-Lusztig for how one defines the root datum of a connected reductive group $G$ \textit{without} choosing $T \subseteq B \subseteq G$.} of $G$ (in the sense of Langlands). 
\end{theorem}

The discussion above can be summarized by saying that the symmetric monoidal abelian category $\mathcal{P}_{G(\mathcal{O})}(\Gr_G)$ together with the fiber functor $H^*(\Gr_G, -)$ defines a \textit{neutralized Tannakian category}\footnote{Actually, there are a few technical conditions to check, such as rigidity; these are easily verified for $\mathcal{P}_{G(\mathcal{O})}(\Gr_G).$}. The Tannakian reconstruction theorem \cite[Theorem 2.7]{baumann2018notes} constructs the affine group scheme $\check{G}$ as the automorphism group scheme of the fiber functor $H^*(\Gr_G,-)$. That $\check{G}$ is connected reductive is deduced from a Tannakian characterization \cite[Proposition 2.11]{baumann2018notes} of this property. The most important ingredient is the semisimplicity of the abelian category $\mathcal{P}_{G(\mathcal{O})}(\Gr_G)$ \cite[Theorem 4.2]{baumann2018notes}, which follows from the parity vanishing property \cite[Lemma 4.5]{baumann2018notes} of the stalks of the intersection complexes of the $G(\mathcal{O})$-orbit closures on $\Gr_G$. This parity vanishing property is a geometric reinterpretation of deep computations of Lusztig \cite[§11]{lusztig1983singularities} in the affine Hecke algebra and can be deduced from the BBD Decomposition Theorem. 

Observe that the fiber functor $H^*(\Gr_G, -)$ is $\ZZ$-graded by the cohomological index. By Tannakian formalism, this $\ZZ$-grading on $H^*(\Gr_G, -)$ defines a homomorphism $2\check{\rho}_G: \mathbb{G}_m \rightarrow \check{G}$. It is well known that the maximal torus $\check{T} \subseteq \check{G}$ (the Borel subgroup $\check{B} \subseteq \check{G}$) arises as the centralizer (respectively, the attracting scheme) of the cocharacter $2\check{\rho}_G$ in $\check{G}$. 

\subsection{Levi subgroups}

We can now approach the problem considered in this paper. We fix a parabolic subgroup $P \subseteq G$ with unipotent radical $V \subseteq P$ and Levi quotient $L = P/V$. Note that $L$ is a connected reductive group in its own right, so we also have the geometric Satake equivalence
$$
\mathbb{S}_{L}: \mathcal{P}_{L(\mathcal{O})}(\Gr_L) \simeq \mathrm{Rep}(\check{L}). 
$$
The categories $\mathcal{P}_{G(\mathcal{O})}(\Gr_G)$ and $\mathcal{P}_{L(\mathcal{O})}(\Gr_L)$ are directly related by a functor, which we shall call the \textit{parabolic restriction}, introduced by Beilinson-Drinfeld in \cite[§5.3.28]{beilinsondrinfeld} (generalizing the \textit{weight functors} of Mirkovi\'{c}-Vilonen \cite[§3]{MV04} in the case that $L = T$ is a torus). We will have much more to say about it in §§\ref{parabolic restriction}, and also refer the reader to \cite[§15.1]{baumann2018notes} for more details. For now, we note that it is an exact tensor functor 
$$
\mathrm{Res}^{G,\natural}_L: \mathcal{P}_{G(\mathcal{O})}(\Gr_G) \rightarrow \mathcal{P}_{L(\mathcal{O})}(\Gr_L).
$$
Namely, we have a diagram of ind-schemes 
$$
\begin{tikzcd}
    & \Gr_P \ar[rd,"r"] \ar[ld,"q",swap] \\
    \Gr_L & & \Gr_G. 
\end{tikzcd}
$$
The functor $q_*r^!: D_{G(\mathcal{O})}(\Gr_G) \rightarrow D_{L(\mathcal{O})}(\Gr_L)$ is not $t$-exact. However, let $\mathcal{F} \in \mathcal{P}_{G(\mathcal{O})}(\Gr_G)$. For each connected component $\Gr_L^\chi \subseteq \Gr_L$ of $\Gr_L$, the complex $q_*r^!\mathcal{F} \vert_{\Gr_L^\chi}$ is concentrated in a unique perverse degree. Thus, one defines the exact functor $\mathrm{Res}^{G,\natural}_L$ by shifting the functor $r_*q^!$ by an appropriate integer on each connected component $\Gr_L^\chi$ (see Remark~\ref{normalized parabolic restriction}).

Tannakian formalism now defines a closed embedding $\check{L} \hookrightarrow \check{G}$ and a canonical identification of the tensor functor $\mathbb{S}_L \circ \mathrm{Res}^{G,\natural}_L \circ \mathbb{S}_G^{-1}$ with the restriction functor $\mathrm{Rep}(\check{G}) \rightarrow \mathrm{Rep}(\check{L})$ \cite[Proposition 15.3]{baumann2018notes}. Moreover, the embedding $\check{L} \hookrightarrow \check{G}$ realizes $\check{L}$ as a Levi subgroup of $\check{G}$ containing the maximal torus $\check{T}$. The dual parabolic subgroup $\check{P} \subseteq \check{G}$ can now be constructed as the attracting scheme in $\check{G}$ for the coweight $2\check{\rho}_G - 2\check{\rho}_L$ of $\check{T}$. 

\subsection{Equivariant corestriction} We will now fix a Levi factor $L \subseteq P$. We obtain a closed immersion 
$$
i: \Gr_L \hookrightarrow \Gr_G
$$
identifying $\Gr_L$ with a closed subscheme of $\Gr_G$. In fact, $\Gr_L \subseteq \Gr_G$ is exactly the fixed point subscheme $\Gr_G^{Z^\circ(L)} \subseteq \Gr_G$ for the action of the neutral component $Z(L)^\circ \subseteq Z(L)$ of the center of $L$. We are interested in the (co)restriction functor 
$$
i^!: D_{G(\mathcal{O})}(\Gr_G) \rightarrow D_{L(\mathcal{O})}(\Gr_L)
$$
given by the composition of the forgetful functor $\mathrm{For}^{G(\mathcal{O})}_{L(\mathcal{O})}$ (forgetting $G(\mathcal{O})$-equivariance down to $L(\mathcal{O})$-equivariance) followed by the functor $i^!: D_{L(\mathcal{O})}(\Gr_G) \rightarrow D_{L(\mathcal{O})}(\Gr_L)$ of $L(\mathcal{O})$-equivariant $!$-restriction.\footnote{We consider $!$-restriction over $*$-restriction only because the spectral ``answer'' to our question will be cleaner; of course, the two functors differ only by Verdier duality.} Unlike $\mathrm{Res}^{G,\natural}_L$, the functor $i^!$ is not $t$-exact, and there is no saving exactness with a mere grading shift. Nonetheless, it will be convenient\footnote{Note that this grading shift will not be used in §\ref{automorphic}} to work with the following ``regraded'' or ``sheared'' version of $i^!$: 
$$
i^{!,\natural} := \bigoplus_{\chi \in \pi_0(\Gr_L)} i^! \vert_{\Gr_L^{\chi}} [-\langle 2\check{\rho}_G - 2\check{\rho}_L, \chi \rangle]. 
$$
Here, $\Gr_L^\chi \subseteq \Gr_L$ is the connected component of $\Gr_L$ labelled by the element $\chi \in \pi_0(\Gr_L) \simeq \Lambda/Q_L$, where $\Lambda$ is the weight lattice of $\check{T}$ and $Q_L \subseteq \Lambda$ is the root lattice of $\check{L}$. 

Unlike $\mathrm{Res}^{G,\natural}_L$, $i^{!,\natural}$ does not admit a monoidal structure. However, in §§\ref{monoidal structure section}, we will equip $i^{!,\natural}$ with a \textit{lax} monoidal structure. This is enough structure to ensure that $i^{!,\natural}$ carries ring objects of $D_{G(\mathcal{O})}(\Gr_G)$ to ring objects of $D_{L(\mathcal{O})}(\Gr_L)$. In fact, we will carry out this construction for any connected reductive subgroup $L \subseteq G$. We will give two constructions of this lax monoidal structure; the second (given in §§\ref{Fusion}) uses Gaitsgory's description \cite[Proposition 6]{gaitsgory2000construction} of the convolution product $- \star -$ in terms of a global nearby cycles construction. Moreover, in §§\ref{parabolic restriction} (see Construction~\ref{Xiconstruction} and Remark~\ref{normalized corestriction}), we will show that the natural adjunction morphism 
\begin{equation}\label{introductionXi}
\Xi^{\natural}: i^{!,\natural} \rightarrow \mathrm{Res}^{G,\natural}_L 
\end{equation}
is compatible with the lax monoidal structures on these functors. 

\subsection{The case of a torus and the work of Ginzburg-Riche} We now specialize to the case in which $L = T$ is a maximal torus of $G$. The affine Grassmannian $\Gr_T$ is the coweight lattice of $T$, which we identify with the weight lattice $\Lambda$ of $\check{T}$. The closed embedding $i: \Lambda = \Gr_T \hookrightarrow \Gr_G$ is the familiar embedding of the coweight lattice into $\Gr_G$ taking the coweight $\lambda \in \Lambda$ to the point $t^\lambda \in \Gr_G$ (the image of the uniformizer $t \in \CC((t))^*$ under the map $\lambda(\CC((t))): \CC((t))^* \rightarrow G(\CC((t))) \twoheadrightarrow \Gr_G$). For each $\mu \in \Lambda$, let $i_\mu: \{t^\mu\} \hookrightarrow \Gr_G$ denote the inclusion of the corresponding point of $\Gr_T$. Let $\lambda \in \Lambda^+$ denote a dominant coweight, and let $\mathrm{IC}_\lambda \in \mathcal{P}_{G(\mathcal{O})}(\Gr_G)$ denote the intersection complex of $\Gr_G^{\leq \lambda} = \overline{\Gr^\lambda_G}$. Note that $L(\lambda) := \mathbb{S}(\mathrm{IC}_\lambda) \in \mathrm{Rep}(\check{G})$ is the simple $\check{G}$-module of highest weight $\lambda \in \Lambda^+$. Since the costalk $i_\mu^! \mathrm{IC}_\lambda$ is only non-trivial when $\mu \leq \lambda$, we may assume that this is the case. 

Lusztig's description of the (non-equivariant) (co)stalks $H^*(i_\mu^!\mathrm{IC}_\lambda)$ predates the geometric Satake equivalence itself. 

\begin{theorem}[{Lusztig \cite[§11]{lusztig1983singularities}}]
    Let $\widetilde{W}_{\mathrm{aff}} = \Lambda \rtimes W$ denote the extended affine Weyl group and $\ell: \widetilde{W}_{\mathrm{aff}} \rightarrow \ZZ$ its length function; let $p_\mu, p_\lambda \in \widetilde{W}_{\mathrm{aff}}$ denote the coweights $\mu, \lambda \in \Lambda$ viewed as elements of $\widetilde{W}_{\mathrm{aff}}$. Let $n_\mu \in \widetilde{W}_{\mathrm{aff}}$ (resp. $n_\lambda \in \widetilde{W}_{\mathrm{aff}}$) denote the element of maximal length in the $W$-double coset $Wp_\mu W \subseteq \widetilde{W}_{\mathrm{aff}}$ (resp. $Wp_\lambda W \subseteq \widetilde{W}_{\mathrm{aff}}$). Then, we have the vanishing $H^i(i_\mu^!\mathrm{IC}_\lambda) \simeq 0$ for $i$ odd, as well as the equality
    $$
    \sum_{i \geq 0} \dim{H^{2i}(i_\mu^! \mathrm{IC}_\lambda)}q^{i} = P_{n_\mu, n_\lambda}(q) = d_\mu(L(\lambda); q) 
    $$
    in the ring $\ZZ[q]$. Here, for any elements $y, w \in \widetilde{W}_{\mathrm{aff}}$ of the extended affine Weyl group, $P_{y,w} \in \ZZ[q]$ denotes the corresponding Kazhdan-Lusztig polynomial (see \cite[§4]{lusztig1983singularities}). Moreover, $d_\mu(L(\lambda); q)$ denotes Lusztig's $q$-analog of the weight multiplicity (see \cite[§6]{lusztig1983singularities}).  
\end{theorem}

The parity vanishing in Lusztig's theorem implies that $H_T^*(i_\mu^! \mathrm{IC}_\lambda)$ is a free module over the ring $R_T := H^*_T(\mathrm{pt}, \CC)$. Hence, it implies the (non-canonical) isomorphism of $R_T$-modules $H_T^*(i_\mu^! \mathrm{IC}_\lambda) \simeq H^*(i_\mu^! \mathrm{IC}_\lambda) \otimes R_T$ and gives the graded dimension of $H^*(i_\mu^! \mathrm{IC}_\lambda)$. 

However, a more intrinsic description of the $R_T$-module $H_T^*(i_\mu^! \mathrm{IC}_\lambda)$ has been obtained by Ginzburg and Riche \cite{GR13}. It is their work which we aim to extend to arbitrary Levi subgroups of $G$. We point out that Ginzburg and Riche go further and describe the equivariant costalks $H_{T \rtimes \mathbb{G}_m}^*(i_\mu^!\mathrm{IC}_\lambda)$, where $\mathbb{G}_m$ denotes the loop rotation torus. In this paper, we will not consider the loop rotation equivariance (because it is somewhat orthogonal to the applications that we have in mind), although we anticipate that our arguments will extend to this setting.

We can evaluate the natural transformation \eqref{introductionXi} on the object $\mathrm{IC}_\lambda \in D_{G(\mathcal{O})}(\Gr_G)$, restrict to the component $t^\mu \in \Gr_T$ of $\Gr_T$, and take cohomology to obtain an $R_T$-module map
\begin{equation}\label{ginzburg riche automorphic map}
H_T^*( i^{!,\natural}_\mu \mathrm{IC}_\lambda ) \rightarrow H_T^*(\{t^\mu\}, \mathrm{Res}^{G,\natural}_T(\mathrm{IC}_\lambda) ) =  H^*(\{t^\mu\}, \mathrm{Res}^{G,\natural}_T(\mathrm{IC}_\lambda) ) \otimes R_T. 
\end{equation}
To formulate their description of $H_T^*( i^{!,\natural}_\mu \mathrm{IC}_\lambda)$, we bring in the Lie algebra $\check{\mathfrak{g}} = \mathrm{Lie}(\check{G})$, the Cartan subalgebra $\check{\mathfrak{t}} = \mathrm{Lie}(\check{T})$, the maximal nilpotent subalgebra $\check{\mathfrak{u}} = \mathrm{Lie}(\check{{U}})$, and the opposite maximal nilpotent subalgebra $\check{\mathfrak{u}}^- = \mathrm{Lie}(\check{U}^-)$. We have the triangular decomposition 
$$
\check{\mathfrak{g}} = \check{\mathfrak{u}} \oplus \check{\mathfrak{t}} \oplus \check{\mathfrak{u}}^-,
$$
which induces a projection $\check{\mathfrak{g}}/\check{\mathfrak{u}} \twoheadrightarrow \check{\mathfrak{t}}$. Let $\CC(-\mu)$ denote the one dimensional $\check{B}$-module corresponding to the weight $-\mu \in \Lambda$ of $\check{T}$. We obtain a canonical map 
\begin{align}\label{introduction ginzburg riche spectral map}
\left( L(\lambda) \otimes \operatorname{Sym}(\check{\mathfrak{g}}/\check{\mathfrak{u}}) \otimes \CC(-\mu) \right)^{\check{B}} &\hookrightarrow \left( L(\lambda) \otimes \operatorname{Sym}(\check{\mathfrak{g}}/\check{\mathfrak{u}}) \otimes \CC(-\mu) \right)^{\check{T}}\\
&\twoheadrightarrow \left( L(\lambda) \otimes \operatorname{Sym}(\check{\mathfrak{t}}) \otimes \CC(-\mu) \right)^{\check{T}} \nonumber \\
&=\,\, L(\lambda)_\mu \otimes  \operatorname{Sym}(\check{\mathfrak{t}}).\nonumber
\end{align}
Here, $L(\lambda)_\mu \subseteq L(\lambda)$ denotes the $\mu$ weight space. Recall, moreover, that there is a canonical isomorphism $R_T \simeq \operatorname{Sym}(\mathfrak{t}^*) \simeq \operatorname{Sym}(\check{\mathfrak{t}})$ which places $\check{\mathfrak{t}}$ in graded degree $2$ (see, for example, \cite[Theorem 6.6.8]{achar2020}).

\begin{theorem}[{Ginzburg-Riche, \cite[Theorem 2.3.1]{GR13}}]\label{GR13 Main Theorem}
    There exists a unique graded $R_T\simeq \operatorname{Sym}(\check{\mathfrak{t}})$-module isomorphism 
    \begin{equation}\label{introduction ginzburg riche isomorphism}
    H_T^*( i^{!,\natural}_\mu \mathrm{IC}_\lambda) \simeq \left( L(\lambda) \otimes \operatorname{Sym}(\check{\mathfrak{g}}/\check{\mathfrak{u}}) \otimes \CC(-\mu) \right)^{\check{B}}
    \end{equation}
    such that the diagram 
    $$
    \begin{tikzcd}
        H_T^*( i^{!,\natural}_\mu \mathrm{IC}_\lambda) \ar[r, "\eqref{ginzburg riche automorphic map}"] \ar[d,"\eqref{introduction ginzburg riche isomorphism}"] &  H^*(\{t^\mu\}, \mathrm{Res}^{G,\natural}_T(\mathrm{IC}_\lambda) ) \otimes R_T \ar[d,"\sim"] \\
        \left( L(\lambda) \otimes \operatorname{Sym}(\check{\mathfrak{g}}/\check{\mathfrak{u}}) \otimes \CC(-\mu) \right)^{\check{B}} \ar[r,"\eqref{introduction ginzburg riche spectral map}"] & L(\lambda)_\mu \otimes  \operatorname{Sym}(\check{\mathfrak{t}})
    \end{tikzcd}
    $$
    commutes. Here, $\check{\mathfrak{g}}/\check{\mathfrak{u}}$ is placed in graded degree $2$ (so that \eqref{introduction ginzburg riche spectral map} is a graded map). 
\end{theorem}

We can formulate their theorem as follows. Let $\mathcal{F}_\reg \in \mathcal{P}_{G(\mathcal{O})}(\Gr_G)$ denote the \textit{regular sheaf}\footnote{Strictly speaking, $\mathcal{F}_\reg$ does not belong to $D_{G(\mathcal{O})}(\Gr_G)$, but rather to a suitable ind-completion (just as the regular representation $\mathcal{O}(\check{G})$ does not belong to $\mathrm{Rep}(\check{G})$). See §§\ref{sheaf conventions}. We ignore this point for simplicity.}, the $\check{G}$-equivariant ring object of $\mathcal{P}_{G(\mathcal{O})}(\Gr_G)$ corresponding to the left regular representation $\mathcal{O}(\check{G}) \in \mathrm{Rep}(\check{G})$ under the geometric Satake equivalence. It admits the Peter-Weyl decomposition
$$
\mathcal{F}_\reg = \bigoplus_{\lambda \in \Lambda^+} \mathrm{IC}_\lambda \boxtimes L(\lambda)^*. 
$$
We identify $\mathrm{Sym}(\check{\mathfrak{g}}/\check{\mathfrak{u}})$ with the algebra of functions $\mathcal{O}\left((\check{\mathfrak{g}}/\check{\mathfrak{u}})^* \right)$. By tensoring $\eqref{introduction ginzburg riche isomorphism}$ with $L(\lambda)^*$ and passing to the direct sum over $\lambda \in \Lambda^+$ and $\mu \in \Lambda$, we obtain a canonical isomorphism of $\ZZ$-graded $\Lambda$-graded $\check{G}$-equivariant $\mathrm{Sym}(\check{\mathfrak{t}}) \simeq \mathcal{O}(\check{\mathfrak{t}}^*)$-algebras
\begin{align*}
H_T^*(i^{!,\natural} \mathcal{F}_\reg) &\simeq \left( \left( \bigoplus_{\lambda \in \Lambda^+} L(\lambda) \boxtimes L(\lambda)^* \right) \otimes \mathcal{O}\left((\check{\mathfrak{g}}/\check{\mathfrak{u}})^*\right) \right)^{\check{U}}\\
&\simeq \left(\mathcal{O}(\check{G}) \otimes \mathcal{O}((\check{\mathfrak{g}}/\check{\mathfrak{u}})^*) \right)^{\check{U}}\\
&\simeq \mathcal{O}\left(\check{G} \times (\check{\mathfrak{g}}/\check{\mathfrak{u}})^*\right)^{\check{U}} \\
&\simeq \mathcal{O}(T^*(\check{G}/\check{U})). 
\end{align*}
We can interpret the $\Lambda$-grading as arising from the $\check{T}$-action on the variety $T^*(\check{G}/\check{U})$ (through the identification $\check{T} \simeq \check{B}/\check{U}$ and the natural $\check{B}$-action by right translation). The $\mathcal{O}(\check{\mathfrak{t}}^*)$-algebra structure arises from the $\check{T}$-equivariant \textit{moment map}
$$
\mu: T^*(\check{G}/\check{U}) \rightarrow \check{\mathfrak{t}}^*.
$$
The $\check{G}$-module structure arises from the natural action of $\check{G}$ on the variety $T^*(\check{G}/\check{U})$ by left translation. Finally, the $\ZZ$-grading on $\mathcal{O}(T^*(\check{G}/\check{U}))$ is induced by the $\mathbb{G}_m$-action on the fibers of the vector bundle $T^*(\check{G}/\check{U})$ through which $t \in \mathbb{G}_m$ acts by $t^{-2}$. 

The isomorphism $H^*_T(i^{!,\natural}\mathcal{F}_\reg) \simeq \mathcal{O}(T^*(\check{G}/\check{U}))$ conversely yields Theorem~\ref{GR13 Main Theorem} by passage to the $\check{G}$-isotypic and $\check{T}$-isotypic components. 

\subsection{Our results} The geometric Satake equivalence induces a canonical pinning of the dual group $\check{G}$. We have already explained that $\check{G}$ is equipped with a canonical Borel subgroup $\check{B} \subseteq \check{G}$ (concretely, the subgroup of automorphisms of the fiber functor $H^*$ preserving the Mirkovi\'{c}-Vilonen filtration) and a canonical maximal torus $\check{T} \subseteq \check{B}$; the point is that one can moreover find distinguished non-zero root vectors $X_\alpha \in \mathfrak{\mathfrak{g}}_\alpha$. We refer the reader to \cite[§6.5]{GR13} for a discussion of the precise pinning that we will use. Let $\Delta$ denote the set of simple roots of $\check{G}$ and $I \subseteq \Delta$ the subset of simple roots of $\check{L}$. 

We can use this pinning to define a distinguished additive character $\psi_I$ of $\check{U}$ as the composition 
$$
\begin{tikzcd}
    \check{U} \ar[r,twoheadrightarrow] & \check{U}/[\check{U},\check{U}] \ar[r,"\sim"] & \bigoplus_{\alpha \in \Delta} \mathbb{G}_a \ar[r,"\mathrm{pr}_I",twoheadrightarrow] & \bigoplus_{\alpha \in I} \mathbb{G}_a \ar[r, "+"] & \mathbb{G}_a.
    \end{tikzcd}
$$
Here, $\mathrm{pr}_I$ is given by projection onto the factors indexed by $I$, and the last map is summation. We also denote by $\psi_I$ the linear form $(d\psi_I)_1 \in \check{\mathfrak{u}}^*$, where $1 \in \check{U}$ is the identity. We write $\check{\mathfrak{u}}^\perp \subseteq \check{\mathfrak{g}}^*$ for the linear complement to $\check{\mathfrak{u}} \subseteq \check{\mathfrak{g}}$. We may regard $\psi_I$ as an element of $\check{\mathfrak{g}}^*$ by extending it trivially over $\check{\mathfrak{b}}^-$. Let $\psi := \psi_\Delta$ (the \textit{non-degenerate} additive character).

The algebraic variety $T^*\check{G}$ is equipped with a pair of commuting actions of $\check{G}$ by left and right translation. We have the $\check{G}$-equivariant (for the \textit{left} translation $\check{G}$-action) moment map (for the \textit{right} translation $\check{G}$-action) 
$$
\mu: T^*\check{G} \simeq \check{G} \times \check{\mathfrak{g}}^* \xrightarrow{\mathrm{pr}_2} \check{\mathfrak{g}}^*
$$
given by projection to the second factor. The moment map for the induced $\check{U}$-action is the composition $T^*\check{G} \xrightarrow{\mathrm{pr}_2} \check{\mathfrak{g}}^* \twoheadrightarrow \check{\mathfrak{u}}^*$. Since $\psi_I \in \check{\mathfrak{u}}^*$ is $\check{U}$-invariant, we can form the Hamiltonian reduction of the Hamiltonian $\check{U}$-variety $T^*\check{G} \rightarrow \check{\mathfrak{u}}^*$ at the level $\psi_I$ to obtain the \textit{partial Kostant-Whittaker reduction}
$$
T^*(\check{G}/(\check{U},\psi_I)) := T^*\check{G} \tripslash (\check{U},\psi_I) := (T^*\check{G} \times_{\check{\mathfrak{u}}^*} \{ \psi_I\})/\check{U}. 
$$
Our entire §\ref{Hamiltonian Reduction} is dedicated to a discussion of this construction and its basic properties. 

The variety $T^*(\check{G}/(\check{U},\psi_I))$ carries a natural left $\check{G}$-action. Moreover, we have a canonical $\check{G}$-equivariant projection 
$$
\pi_I: T^*(\check{G}/(\check{U}, \psi_I)) \hookrightarrow (T^*\check{G})/\check{U} \twoheadrightarrow \check{\mathfrak{g}}^*/\check{U} \twoheadrightarrow \check{\mathfrak{l}}^*/\check{U} \rightarrow \check{\mathfrak{c}}_I
$$
to the GIT quotient $\check{\mathfrak{c}}_I := \check{\mathfrak{l}}^*\doubslash \check{L}$. The homomorphism $\pi_I^*$ equips $\mathcal{O}(T^*(\check{G}/(\check{U}, \psi_I)))$ with the structure of a $\check{G}$-equivariant $\mathcal{O}(\check{\mathfrak{c}}_I)$-algebra. 

The action of $\mathbb{G}_m$ on $\check{\mathfrak{l}}^*$ given by $t \cdot \xi = t^{-2} \xi$ commutes with the coadjoint action of $\check{L}$, hence descends to a $\mathbb{G}_m$-action on $\check{\mathfrak{c}}_I$ (which places the generators of the polynomial algebra $\mathcal{O}(\check{\mathfrak{c}}_I)$ in degrees given by \textit{twice} the characteristic exponents of $\check{L}$). We let $\mathbb{G}_m$ act on $T^*\check{G} \simeq \check{G} \times \check{\mathfrak{g}}^*$ by the cocharacter $2\check{\rho}_I$ on the $\check{G}$ factor and by $t \mapsto t^{-2}$ on $\check{\mathfrak{g}}^*$ (we will discuss gradings more carefully in Construction~\ref{spectralgrading}). These actions of $\mathbb{G}_m$ make $\pi_I$ a $\mathbb{G}_m$-equivariant map. Hence, $\pi_I^*$ equips $\mathcal{O}(T^*(\check{G}/(\check{U}, \psi_I)))$ with the structure of a \textit{graded} $\check{G}$-equivariant $\mathcal{O}(\check{\mathfrak{c}}_I)$-algebra. 

Let $\mathfrak{J}_I \rightarrow \check{\mathfrak{c}}_I$ denote Ng\^{o}'s regular centralizer group scheme \cite[Lemme 2.1.1]{Ngo10} (see also Riche's treatment \cite{Ri17}) associated to the group $\check{L}_I$. It is a commutative affine group scheme over $\check{\mathfrak{c}}_I$. Moreover, when we discuss the regular centralizer in more detail in §§\ref{spectralactionofthecentralizer}, we shall see that it acts canonically on the variety $T^*(\check{G}/(\check{U}, \psi_I))$ (or \textit{any} partial Kostant-Whittaker reduction, for that matter). When the derived subgroup of $G$ is almost simple, Yun-Zhu \cite{yun2009integral} provide a canonical isomorphism of affine $\check{\mathfrak{c}}_I$-group schemes 
\begin{equation}\label{introduction yunzhuiso}
\operatorname{Spec}{H^{L(\mathcal{O})}_*(\Gr_L, \CC)} \simeq \mathfrak{J}_I.
\end{equation}
We will discuss their results in detail in both §§\ref{equivarianthomologysubsection} and §§\ref{yun zhu comparison}.

On the other hand, we will recall (again following Yun-Zhu \cite{yun2009integral}) in §§\ref{equivarianthomologysubsection} how the group scheme $\operatorname{Spec}{H^{L(\mathcal{O})}_*(\Gr_L, \CC)}$ acts canonically on $H_{L(\mathcal{O})}^*(\Gr_L, i^{!,\natural} \mathcal{F}_\reg)$. 

Finally, recall that there is a canonical algebra isomorphism $R_L := H_L^*(\mathrm{pt}, \CC) \simeq \mathcal{O}(\check{\mathfrak{c}}_I)$. This isomorphism identifies our grading on $\mathcal{O}(\check{\mathfrak{c}}_I)$ with the cohomological grading on $R_L$. 

With this notational setup in place, we can state our main result in its basic form.

\begin{theorem}\label{Main Theorem}
    There is a canonical $\check{G}$-equivariant isomorphism of graded $R_L$-algebras 
    $$
    \Upsilon_I: H_{L(\mathcal{O})}^*(\Gr_L, i^{!,\natural} \mathcal{F}_\reg) \simeq \mathcal{O}(T^*(\check{G}/(\check{U},\psi_I))). 
    $$
    The left hand side is regarded as a ring via the lax monoidal structure on $i^{!,\natural}$ of §§\ref{monoidal structure section}. Assume that the derived subgroup of $G$ is almost simple (so that the Yun-Zhu isomorphism \eqref{introduction yunzhuiso} is available). Then, $\Upsilon_I$ is moreover $\mathfrak{J}_I$-equivariant. 
\end{theorem}
In fact, our result is stronger, and describes the canonical $\check{G}$-equivariant graded $R_L$-algebra homomorphism of \eqref{introductionXi}
$$
\xi^\natural := \Xi^\natural(\mathcal{F}_\reg): H_{L(\mathcal{O})}^*(\Gr_L, i^{!,\natural} \mathcal{F}_\reg) \rightarrow H^*_{L(\mathcal{O})}(\Gr_I, \operatorname{Res}^{G,\natural}_L(\mathcal{F}_\reg))
$$
in terms of the Langlands dual group. However, to avoid defining yet more notation in this introduction, we will simply note that the version of Theorem~\ref{Main Theorem} stated and proved in §§\ref{theproof} as Theorem~\ref{Main Theorem 1} asserts and demonstrates these additional compatibliities. 

\begin{remark}\label{introduction hartogs}
    We want to draw attention to what we find to be one of the more interesting aspects of the proof of Theorem~\ref{Main Theorem}. Namely, our choice to treat the regular object $\mathcal{F}_\reg$ rather than the complexes $\mathrm{IC}_\lambda$ individually is \textit{essential} to our argument. In \cite[Proposition 3.2.3]{GR13}, Ginzburg-Riche establish that for any dominant $\lambda \in \Lambda^+$, the $\mathcal{O}(\check{\mathfrak{t}}^*)$-module 
    $$
    H^0(\widetilde{\check{\mathfrak{g}}}, \mathcal{O}_{\widetilde{\check{\mathfrak{g}}}}(\lambda)) = \mathrm{Ind}_{\check{B}}^{\check{G}}(\mathrm{Sym}(\check{\mathfrak{g}}/\check{\mathfrak{u}}) \otimes \CC(-\lambda))^{\check{B}}
    $$
    is \textit{free}. Here, $\widetilde{\check{\mathfrak{g}}} \rightarrow {\check{\mathfrak{g}}}$ is the Grothendieck-Springer alteration and $\mathcal{O}_{\widetilde{\check{\mathfrak{g}}}}(\lambda)$ is the pullback of the Borel-Weil line bundle $\mathcal{O}_{\check{G}/\check{B}}(\lambda)$ on the flag variety $\check{G}/\check{B}$ along the natural projection $q: \widetilde{\check{\mathfrak{g}}} \rightarrow \check{G}/\check{B}$. Their proof appeals to cohomological computations of Broer \cite[Proposition 2.5]{broer}. By working instead with the entire coordinate ring $\mathcal{O}(T^*(\check{G}/\check{U}))$, (a suitable adaptation of) our application of Hartog's principle (Lemma \ref{Spectral Hartogs}) can be used in place of this freeness. We then obtain the freeness of $H^0(\widetilde{\check{\mathfrak{g}}}, \mathcal{O}_{\widetilde{\check{\mathfrak{g}}}}(\lambda))$ as a \textit{corollary} to the Ginzburg-Riche theorem. 
\end{remark}

We can pass to $\check{G}$-isotypic components in Theorem~\ref{Main Theorem} and undo the grading twist in the definition of $i^{!,\natural}$ to obtain the following more concrete result. 

\begin{corollary}\label{main theorem simple by simple}
    Let $\lambda \in \Lambda^+$ denote a dominant weight of $\check{T}$ and $L(\lambda) = \mathbb{S}_{G}(\mathrm{IC}_\lambda)$ the corresponding simple $\check{G}$-module. Then, we have a canonical graded $R_L \simeq \mathcal{O}(\check{\mathfrak{c}}_I)$-module isomorphism
    $$
    \Upsilon_{I,\lambda}: H_{L(\mathcal{O})}^*(\Gr_L, i^{!}\mathrm{IC}_\lambda) \simeq (L(\lambda) \otimes \mathcal{O}(\check{\mathfrak{u}}^\perp + \psi_I))^{\check{U}}.
    $$
    The grading of $H_{L(\mathcal{O})}^*(\Gr_L, i^{!}\mathrm{IC}_\lambda)$ by the component group $\pi_0(\Gr_L) \simeq \Lambda/Q_L \simeq X^*(Z(\check{L}))$ corresponds to the grading of $(L(\lambda) \otimes \mathcal{O}(\check{\mathfrak{u}}^\perp + \psi_I))^{\check{U}}$ by the eigenvalues of the center $Z(\check{L})$. Moreover, the cohomological grading on $H_{L(\mathcal{O})}^*(\Gr_L, i^{!}\mathrm{IC}_\lambda)$ corresponds under $\Upsilon_{I,\lambda}$ to the \emph{Brylinski-Kostant} grading  on $(L(\lambda) \otimes \mathcal{O}(\check{\mathfrak{u}}^\perp + \psi_I))^{\check{U}}$. 

    Assume moreover that $G$ has almost simple derived group. Then, $\Upsilon_{I,\lambda}$ is an isomorphism of $\mathfrak{J}_I$-modules. 
\end{corollary}

To reiterate Remark~\ref{introduction hartogs}, we don't know how to prove Corollary~\ref{main theorem simple by simple} without first proving Theorem~\ref{Main Theorem}, since our approach to the latter leverages the algebraic geometry of the moment map $T^*(\check{G}/\check{V}) \rightarrow \check{\mathfrak{l}}^*$.

\subsection{Notation and conventions} We will now review some of our notation pertaining to root data and dual groups. 

\subsubsection{Group theoretic notation} We fix a connected reductive group $G$ over $\CC$. At certain points in §\ref{comparison}, we will take the additional hypothesis that the derived subgroup $G^{\mathrm{der}} \subseteq G$ is almost simple; we will make this assumption explicit in theorem statements whenever it is needed. We will fix a Borel subgroup $B \subseteq G$ and a maximal torus $T \subseteq B$ with weight lattice $\check{\Lambda} = X^*(T)$ and coweight lattice $\Lambda = X_*(T)$. Let $U \subseteq B$ denote the unipotent radical. We also have the opposite Borel subgroup $B^- \subseteq G$ and its unipotent radical $U^- \subseteq B^-$. 

Let $\check{\Phi} \subseteq X^*(T)$ denote the set of roots, $\check{\Phi}^+ \subseteq \check{\Phi}$ the subset of positive roots (relative to $B$), and $\check{\Delta} \subseteq \check{\Phi}^+$ the subset of simple roots. Let $\Phi \subseteq \Lambda = X_*(T)$ denote the set of coroots, ${\Phi}^+ \subseteq {\Phi}$ the positive coroots, and ${\Delta} \subseteq {\Phi}^+$ the simple coroots. We typically denote coroots $\alpha \in \Phi$ without a check; the corresponding root is denoted $\check{\alpha} \in \check{\Phi}$. Let $\Lambda^+ \subseteq \Lambda$ denote the subset of dominant coweights. 

Given a subset $I \subseteq \check{\Delta}$, we let $P_I \subseteq G$ denote the corresponding standard parabolic subgroup. That is, $P_I$ is the unique parabolic subgroup of $G$ containing $B$ such that the negative simple root spaces $\mathfrak{g}_{-\check{\alpha}} \subseteq \mathfrak{g}$ contained in $\mathfrak{p}_I := \mathrm{Lie}(P_I)$ are exactly those labelled by the negative simple roots $\check{\alpha} \in I$. We let $V_I \subseteq P_I$ denote the unipotent radical of $P^-_I$, $P^-_I \subseteq G$ the opposite parabolic subgroup, and $V_I^{-} \subseteq P^{-}_I$ its unipotent radical. Let $L_I = P_I \cap P_I^-$ denote the unique Levi subgroup of $P_I$ containing $T$. Let $B_I = B/V_I \subseteq L_I$ and $U_I = U/V_I \subseteq L_I$; $B_I$ is a Borel subgroup of $L_I$ and $U_I \subseteq B_I$ is the unipotent radical of $B_I$. 

Let $\check{\Phi}_I \subseteq \check{\Phi}$ denote the set of roots of $L_I$ and $\check{\Phi}^+_I \subseteq \check{\Phi}_I$ the subset of positive roots (relative to $\check{B}_I$). Let $2\check{\rho}_I = \sum_{\check{\alpha} \in \check{\Phi}_I^+} \check{\alpha} \in X^*(T)$ and $2\rho_I = \sum_{\alpha \in \Phi_I^+} \alpha$ denote the sum of the positive roots of $L_I$ and positive coroots of $L_I$, repsectively. Note that $L_I = Z_{G}(2\rho_I)$ is the centralizer of the homomorphism $2\rho_I: \mathbb{G}_m \rightarrow G$ in $G$. When $I = \Delta$, we simply write $2\rho$ and $2\check{\rho}$ in place of $2\rho_\Delta$ and $2\check{\rho}_\Delta$, respectively. 

Let $\check{G}$ denote the Langlands dual group. Recall that the geometric Satake equivalence induces a \textit{canonical} pinning of $\check{G}$ (independently of our choice of $T \subseteq B \subseteq G$, which is made only for notational convenience). We adopt the conventions of \cite[Section 6.5]{GR13} and refer the reader there for a discussion of the pinning that we use (although note that their notation differs from ours in that what we call $G$, they call $\check{G}$). In particular, let $\check{T} \subseteq \check{B} \subseteq \check{G}$ denote the maximal torus and Borel subgroup, $\check{B}^- \subseteq \check{G}$ the opposite Borel subgroup, $\check{U} \subseteq \check{B}$ and $\check{U}^- \subseteq \check{B}^-$ the unipotent radical and opposite unipotent radical, respectively. We have canonical identifications $X^*(T) = X_*(\check{T})$ and $X_*(T) = X^*(\check{T})$ under which $\check{\Phi} \subseteq X^*(T)$ and $\Phi \subseteq X_*(T)$ correspond to the coroots and roots of $\check{G}$, respectively. We abuse notation and let $I \subseteq \Delta$ denote the set of simple coroots of $G$ (roots of $\check{G}$) corresponding a given subset $I \subseteq \check{\Delta}$ of simple roots of $G$ (coroots of $\check{G}$). Given a subset $I \subseteq \check{\Delta}$, we let $\check{P}_I \subseteq \check{G}$, $\check{P}_I^- \subseteq \check{G}$, $\check{V}_I \subseteq \check{P}_I$, $\check{V}_I^- \subseteq \check{P}_I^-$, $\check{L}_I = \check{P}_I \cap \check{P}^-_I$ denote the corresponding parabolic subgroup, opposite parabolic subgroup, unipotent radical, opposite unipotent radical, and Levi subgroup, respectively. 

If the subset $I \subseteq \Delta$ is understood, we will sometimes drop the subscript $I$ if it is unlikely to cause confusion; for example, we may write $\check{L}$ in place of $\check{L}_I$. 

Given a subset $I \subseteq \Delta$, we let $i_I: \Gr_I := \Gr_{L_I} \hookrightarrow \Gr_G$ denote the induced closed immersion of affine Grassmannians. When $I = \emptyset$, we often write $\Gr_T$ in place of $\Gr_\emptyset$ and $i_T$ in place of $i_\emptyset$. 

\subsubsection{Groups attached to the formal disc} As in the introduction, we write $G(\mathcal{O})$ for the arc group associated to $G$ in place of the more technically correct notation $L^+G$ (where $L^+G(R) := G(R[[t]])$), except when we are intentionally being more precise. Similarly, we write $G(\mathcal{K})$ for the loop group $LG$ (where $LG(R) = G(R((t)))$).

\subsubsection{Group actions} All group actions are \textit{left} actions. Given an action of an algebraic group $G$ on an affine variety $X$, the induced action of $G$ on the coordinate ring $\mathcal{O}(X)$ is characterized by the formula $(g \cdot f )(x) = f(g^{-1}x)$ (where $g \in G$, $f \in \mathcal{O}(X)$ and $x \in X$). These conventions are relevant for determining the gradings produced by $\mathbb{G}_m$-actions. For example, suppose that $\mathbb{G}_m$ acts on the vector space $\mathfrak{\check{t}}^*$ by the formula $t \cdot \xi = t^{-2} \xi$ (for $t \in \mathbb{G}_m$, $\xi \in \mathfrak{\check{t}}^*$). Then, $t \in \mathbb{G}_m$ acts on a linear form $x \in \mathfrak{\mathfrak{t}}$ by $t \cdot x = t^2 x$, since $(t \cdot x) (\xi) = x(t^{-1} \cdot \xi) = x(t^2 \xi) = t^2 x(\xi)$ for any $\xi \in \mathfrak{\check{t}}^*$. Therefore, this $\mathbb{G}_m$-action on $\check{\mathfrak{t}}$ equips the polynomial algebra $\mathrm{Sym}({\mathfrak{t}}) \simeq \mathcal{O}(\check{\mathfrak{t}}^*)$ with the grading in which the generators have degree $+2$, i.e. the cohomological grading on $H_T^*(\mathrm{pt}, \CC)$. 

\subsubsection{Sheaf-theoretic notation}\label{sheaf conventions} Let $H$ denote a complex algebraic group acting on a complex algebraic variety $X$. We write  $D_H(X) := D_H(X, \CC)$ for the triangulated $H$-equivariant constructible derived category of sheaves of complex vector spaces on the underlying topological space $X^{\mathrm{an}}(\CC)$, as defined in \cite[Definition 6.2.11]{achar2020}. We omit the superscript ``$b$'' on $D_H(X)$ as we will not have occasion to consider the non-constructible derived category. Let $\mathcal{P}_H(X) \subseteq D_H(X)$ denote the heart of the perverse $t$-structure, the abelian category of $H$-equivariant perverse sheaves on $X$. Of course, we will also need to consider the situation in which $X$ is a ind-scheme over $\CC$ and $H$ is a pro-algebraic $\CC$-group. In this situation, we follow the conventions of \cite[§2.2]{nadler2004perverse} (see also \cite[§9.1]{achar2020}) in defining $D_H(X)$. In particular, our convention is that every object $F \in D_{G(\mathcal{O})}(\Gr_G)$ of the spherical Hecke category is supported on a closed stratum $\Gr_G^{\leq \lambda}$ (for some $\lambda \in \Lambda^+$).

We have found it convenient to organize the constructions of this paper around the ``regular object'' $\mathcal{F}_\reg$ (see Remark~\ref{ring structure on the regular sheaf}), the ``perverse sheaf'' on $\Gr_G$ corresponding under the geometric Satake equivalence $\mathcal{P}_{G(\mathcal{O})}(\Gr_G) \simeq \mathrm{Rep}(\check{G})$ to the (left) regular representation $\mathcal{O}(\check{G})$ of $\check{G}$. Of course, under the conventions above, there is no such object in the category $\mathcal{P}_{G(\mathcal{O})}(\Gr_G)$ (just as the regular representation $\mathcal{O}(\check{G})$ does not belong to the category $\mathrm{Rep}(\check{G})$ of \textit{finite dimensional} $\check{G}$-modules). It is equally obvious that this observation does not pose a genuine issue. The least technical solution is to simply define $\mathcal{F}_\reg$ as a $\Lambda^+$-graded object of the abelian category $\mathcal{P}_{G(\mathcal{O})}(\Gr_G)$ (in the sense of \cite[§2.1.3]{nadler2004perverse}) through the formula
$$
\mathcal{F}_\reg := \bigoplus_{\lambda \in \Lambda^+} \mathrm{IC}_\lambda \boxtimes L(\lambda)^*.
$$
Here, $L(\lambda) \in \mathrm{Rep}(\check{G})$ is the simple $\check{G}$-module of highest weight $\lambda \in \Lambda^+$ and $\mathrm{IC}_\lambda \in \mathcal{P}_{G(\mathcal{O})}(\Gr_G)$ denotes the corresponding irreducible perverse sheaf supported on $\Gr_G^{\leq \lambda} := \overline{\Gr^\lambda_G}$, where $\Gr^\lambda_G = G(\mathcal{O}) t^\lambda \subseteq \Gr_G$ is the corresponding spherical orbit. The corestriction $i_I^! \mathcal{F}_\reg$ should then be understood as the $\Lambda^+$-graded object 
$$
i_I^!\mathcal{F}_\reg := \bigoplus_{\lambda \in \Lambda^+} i_I^!\mathrm{IC}_\lambda \boxtimes L(\lambda)^*
$$
of $D_{L_I(\mathcal{O})}(\Gr_I)$. However, we have chosen to completely suppress this ultimately notational issue in the body of the paper, and will instead refer to $\mathcal{F}_\reg$ as an object of $\mathcal{P}_{G(\mathcal{O})}(\Gr_G)$; the concerned reader should have no issue in making this abuse of notation precise. Similarly, we will refer to the regular representation $\mathcal{O}(\check{G})$ as an object of $\mathrm{Rep}(\check{G})$. 

\subsection{Acknowledgements}

We thank David Nadler, Tsao-Hsien Chen, John O'Brien, Jeremy Taylor, Yixuan Li, Peter Haine, Guglielmo Nocera, and Connor Halleck-Dube for many useful comments and discussions. We are especially grateful for the extensive mathematical and academic support that we have received from David Nadler over the past four years. 

Of course, this paper was directly motivated by the works of Ginzburg-Riche \cite{GR13} and Ginzburg-Kazhdan \cite{ginzburg2022differential} and borrows heavily from their ideas. Moreover, the foundational works of Achar \cite{achar2020} and Achar-Riche \cite{acharriche2023} were invaluable in preparing this work. 

The author was supported by an NSF Graduate Research Fellowship during the writing of this paper. 

\section{Automorphic Side}\label{automorphic}

\subsection{Equivariant localization}\label{equivariant localization subsection} We begin on the ``geometric side'' of Theorem~\ref{Main Theorem}. Let $I \subseteq \check{\Delta}$ denote a set of simple roots of $G$. Let $\lambda \in \Lambda^+$ denote a dominant coweight and let $\mathrm{IC}_\lambda \in \mathcal{P}_{G(\mathcal{O})}(\Gr_G)$ denote the corresponding irreducible perverse sheaf supported on $\Gr_G^{\leq \lambda} := \overline{\Gr^\lambda_G}$, where $\Gr^\lambda_G = G(\mathcal{O}) t^\lambda \subseteq \Gr_G$ is the corresponding spherical orbit. Let $\Gr_I = \Gr_{L_I}$ denote the affine Grassmannian of the Levi subgroup $L_I \subseteq G$ and $i_I: \Gr_I \hookrightarrow \Gr_G$ the inclusion. We will study the cohomology $H^*_{L_I(\mathcal{O})}(\Gr_I, i_I^! \mathrm{IC}_\lambda)$ through the technique of equivariant localization. We abuse notation and write $i_I^! \mathrm{IC}_\lambda \in D_T(\Gr_L)$ for the image of $i_I^! \mathrm{IC}_\lambda \in D_{L(\mathcal{O})}(\Gr_L)$ under the functor forgetting the $L(\mathcal{O})$-equivariance down to $T$-equivariance. Recall from \cite{GKM98} (especially the discussion surrounding \cite[Theorem 1.6.2]{GKM98}) that a complex $F \in D_T(\Gr_L)$ is \textit{equivariantly formal} if the spectral sequence 
$$
E_2^{p,q} = H_T^p(\mathrm{pt}, H^q(\Gr_I, F)) \implies H^{p + q}_T(\Gr_I, F)
$$
degenerates at $E_2$. 

\begin{proposition}\label{equivariantformality}
The complex $i_I^!\mathrm{IC}_\lambda \in D_{T}(\Gr_L)$ is equivariantly formal.
\end{proposition}
\begin{proof} For each point $x \in \Gr_I$, we have that $H^j(i_x^!i_I^!\mathrm{IC}_\lambda) \simeq H^j(i_x^!\mathrm{IC}_\lambda)$ vanishes for $j \not\equiv \langle 2\check{\rho}, \lambda \rangle$ (mod 2), by the well-known parity vanishing property of $\mathrm{IC}_\lambda$ \cite[Lemma 4.5]{baumann2018notes}. Since $i^!\mathrm{IC}_\lambda$ is constructible with respect to an affine paving of $\Gr_I$ (the Schubert stratification by the orbits of an Iwahori subgroup $\mathcal{I} \subseteq L(\mathcal{O})$), a standard argument with the Cousin spectral sequence implies that $H^j(\Gr_I, i_I^!\mathrm{IC}_\lambda)$ vanishes for $j \not\equiv \langle 2\check{\rho}, \lambda \rangle$. 

It is now clear from parity considerations that the spectral sequence
$$
E_2^{p,q} = H^p_{T}(\mathrm{pt}, H^q(\Gr_I, i_I^!\mathrm{IC}_\lambda)) \implies H^{p+q}_T(\Gr_I, i_I^!\mathrm{IC}_\lambda)
$$
degenerates at $E_2$. 
\end{proof}

\begin{remark}\label{parityandformality}
    Proposition~\ref{equivariantformality} is a special case of the following very simple but quite general observation. Let $X$ denote a $T$-variety equipped with a $T$-stable affine paving. Let $F \in D_T(X)$ denote a $T$-equivariant complex whose underlying non-equivariant complex $\mathrm{For}^{T}(F) \in D(X)$ is \textit{$!$-parity} (in the language of \cite{paritysheaves}, and with respect to the dimension pariversity $\diamond$). Then, $F$ is equivariantly formal. Moreover, the property of being $!$-parity is preserved by $!$-restriction to any $T$-stable closed subvariety $i: Y \hookrightarrow X$. Thus, if $Y$ also carries a $T$-stable affine paving, then $i^!F$ is equivariantly formal. 
\end{remark}

\begin{remark}
    Let $\mathrm{IC}_\lambda$ be equipped with its natural structure of $T$-equivariant mixed Hodge module (the theory of equivariant mixed Hodge modules is not very well documented, but see \cite{equivariantmhm} for a treatment) on $\Gr_G$, for which it is pure of weight 0. Note that $i_I^!\mathrm{IC}_\lambda$ is not necessarily pure. Indeed,  $i^!_I\mathrm{IC}_\lambda$ is usually not a semisimple complex (which implies that it is not pure, by the BBD decomposition theorem). Therefore, the equivariant formality of Proposition~\ref{equivariantformality} really requires the costalk parity considerations used in the proof and does not follow from the proof of \cite[Theorem 14.1(7)]{GKM98}. On the other hand, the same argument works without change in settings which lack a theory of weights but in which strong parity vanishing results hold; for example, on the quaternionic affine Grassmannian of \cite{quaternionicsatake}. 
\end{remark}

Recall (see \cite[Section 3]{MV04}) that $\Gr_G$ admits the Iwasawa decomposition 
$$
\Gr_G = \coprod_{\nu \in \Lambda} S_\nu,
$$
where $S_\nu = U(\mathcal{K}) \cdot t^\nu$ is the orbit of the group ind-scheme $U(\mathcal{K})$ through the point $t^\nu \in \Gr_G$. Following \cite{Faltings2003}, we write $L^{-}B$ for the ind-affine group ind-scheme representing the functor $R \mapsto L^-B(R) = B(R[t^{-1}])$ and $L^{--}B \subseteq L^-B$ for the kernel of the reduction morphism $L^-B \twoheadrightarrow B$ given by $t^{-1} \mapsto 0$. Then, by the discussion in \cite{acharriche2023} surrounding Equation 1.2.6, there is a $T$-equivariant isomorphism
\begin{equation}\label{semiinfiniteorbits}
L^{--}B \simeq S_\nu
\end{equation}
given on $\CC$-points by $g \mapsto t^\nu g G(\mathcal{O}) \in \Gr_G$. Here, $T$ acts on $L^{--}B$ acts by conjugation. In the proof of Proposition~\ref{stabilizercomputation}, we write $V = V_I$ for the unipotent radical of $P = P_I$ to ease notation. 

\begin{proposition}\label{stabilizercomputation}
For $x \in \Gr_G$, let $T_x \subseteq T$ denote the stabilizer of $x$ in $T$. Assume that $\alpha \vert_{T_x}: T_x \rightarrow \mathbb{G}_m$ is a dominant morphism for every $\alpha \in \Phi^+_G \setminus \Phi_L^+$. Then, $x \in \Gr_L$.
\end{proposition}
\begin{proof}
    The point $x$ belongs to the semi-infinite orbit $S_\nu$ for a unique $\nu \in \Lambda$. We will use the $T$-equivariant isomorphism 
    $L^{--}B \simeq S_\nu$ of Equation~\ref{semiinfiniteorbits}. Let $B' = B \cap L$. We also have a $T$-equivariant isomorphism of schemes $V \times B' \simeq B$ ($(v,b') \mapsto vb'$), which induces a $T$-equivariant isomorphism $L^{--}V \times L^{--}B' \simeq L^{--}B$. Since $B'$ is a Borel subgroup of $L$, we also obtain a $T$-equivariant isomorphism $L^{--}B' \simeq S_{\nu}'$, where $S_\nu'$ is the semi-infinite orbit in $\Gr_L$ through $t^\nu$. Hence, there is a $T$-equivariant isomorphism $S_\nu \simeq L^{--}V \times S_\nu'$. Let $\pi_1: S_\nu \rightarrow L^{--}V$ denote the first projection. To show that $x \in \Gr_L$, it suffices to show that $\pi_1(x)$ is the identity element $e \in L^{--}V$. Since $\pi_1$ is $T$-equivariant, the subgroup $T_x \subseteq T$ also fixes $\pi_1(x)$. Furthermore, we have a $T$-equivariant isomorphism of ind-affine ind-schemes
    $$
    L^{--}V \simeq \prod_{\alpha \in \Phi^+_G \setminus \Phi^+_L} L^{--}U_\alpha,
    $$
    where $U_\alpha \subseteq U$ is the root subgroup corresponding to the root $\alpha$. Let $\pi_\alpha$ denote the projection onto the factor indexed by $\alpha$. We have an isomorphism $L^{--}U_\alpha \simeq L^{--}\mathbb{A}^1$ under which the adjoint action of $T$ on $L^{--}U_\alpha$ is induced by the homomorphism $\alpha: T \rightarrow \mathbb{G}_m$ and the natural scaling action of $\mathbb{G}_m$ on $L^{--}\mathbb{A}^1$. Since $\alpha \vert_{T_x}$ dominates $\mathbb{G}_m$, we deduce that $(L^{--}\mathbb{A}^1)^{T_x} = (L^{--}\mathbb{A}^1)^{\mathbb{G}_m} = \{0\}$. Therefore, $\pi_\alpha (\pi_1(x)) = \pi_\alpha(e) \in L^{--}U_\alpha$ for each $\alpha \in \Phi^+_G \setminus \Phi^+_L$, which implies that $\pi_1(x) = e$, as claimed. 
\end{proof}

We will now apply the localization theorem in $T$-equivariant cohomology. We again refer the reader to \cite{GKM98} for a general treatment as well as \cite[Theorem A.1.13]{zhu2016introduction} for the specific statement that we use. Recall that we have a canonical graded ring isomorphism 
$$
H^*_T(\mathrm{pt}, \CC) \simeq \operatorname{Sym}({\mathfrak{t}}^*)
$$
obtained (for example) by identifying $H_T^*(\mathrm{pt}, \CC)$ with the cohomology $H^*(BT, \CC)$ of the (topological) classifying space of $T$. Here, the elements of $\mathfrak{t}^*$ are placed in graded degree 2. 

Let $i_T: \Gr_T \hookrightarrow \Gr_G$ denote the inclusion of the Grassmannian of $T$. We use the same symbol to denote the inclusion $\Gr_T \hookrightarrow \Gr_I$.

\begin{proposition}\label{localization}
    (a) Let 
    $$
    f_I = \prod_{\alpha \in \Phi^+_G \setminus \Phi^+_L} \alpha \in \Sym{\mathfrak{h}^*} = H_T^*(\mathrm{pt}).
    $$
    The natural map
    \begin{equation}\label{localizationeq1}
    H^*_T(\Gr_I, i_I^!\mathrm{IC}_\lambda) \rightarrow H^*_T(\Gr_G, \mathrm{IC}_\lambda)
    \end{equation}
    is an injective map of $H_T^*(\mathrm{pt}, \CC)$-modules, and becomes an isomorphism after inverting $f_I$. 

    (b) Let 
    $$
    g_I = \prod_{\alpha \in \Phi_L^+} \alpha \in \Sym{\mathfrak{h}^*} = H_T^*(\mathrm{pt}).
    $$
    The natural map
    $$
    H_T^*(\Gr_T, i_T^!\mathrm{IC}_\lambda) \simeq H_T^*(\Gr_T, i_T^!i_I^!\mathrm{IC}_\lambda)\rightarrow H_T^*(\Gr_I, i_I^!\mathrm{IC}_\lambda)
    $$
    is an injective map of $H_T^*(\mathrm{pt}, \CC)$-modules, and becomes an isomorphism after inverting $g_I$. 
\end{proposition}
\begin{proof}
    (a) Let $Z = \Gr^{\leq \lambda}_G$, $Z' = Z \cap \Gr_L$, and $U = Z \setminus Z'$.  By Proposition~\ref{stabilizercomputation}, $f_I$ vanishes on $\mathrm{Lie}(T_x) \subseteq \mathfrak{t}$ for any point $x \in U \subseteq \Gr_G \setminus \Gr_L$ (at least one of the roots $\alpha \in \Phi^+_G \setminus \Phi_L^+$ is trivial on $T_x$ since $x\not\in \Gr_I$). Hence, by the equivariant localization theorem (applied to the projective $T$-scheme $Z$, the $T$-stable closed subscheme $Z'$, and the complex $\mathrm{IC}_\lambda \in D_T(Z)$), say in the form of \cite[Theorem A.1.13]{zhu2016introduction}, the morphism \ref{localizationeq1} becomes an isomorphism after inverting $f_I$. By Proposition~\ref{equivariantformality}, $H_T^*(\Gr_I, i_I^!\mathrm{IC}_\lambda)$ is a free $H_T^*(\mathrm{pt},\CC)$-module, so the injectivity of the map \ref{localizationeq1} follows. The proof of (b) is similar, with the pair $(G,L)$ replaced by $(L,T)$. 
\end{proof}

\begin{remark}
    Suppose that the Levi subgroup $L = L_I$ is replaced by an arbitrary connected reductive subgroup $K \subseteq G$. Let $i: \Gr_K \hookrightarrow \Gr_G$ denote the induced closed immersion. If $K$ contains a maximal torus $T$ of $G$, then the arguments of this section carry over to show that the natural map 
    $$
    H_T^*(\Gr_K, i^!\mathrm{IC}_\lambda) \rightarrow H_T^*(\Gr_G, \mathrm{IC}_\lambda)
    $$
    is injective and becomes an isomorphism after inverting a suitable element of $H_T^*(\mathrm{pt}, \CC)$. However, without an analysis of the $T$-stabilizers of points $x \in \Gr_G \setminus \Gr_K$ analogous to Proposition~\ref{stabilizercomputation}, it is not clear \textit{which} element of $H_T^*(\mathrm{pt}, \CC)$ must be inverted. 
\end{remark}

\begin{remark}
    Proposition~\ref{localization} can be understood in geometric terms as follows. The morphisms
    $$
    H^*_T(\Gr_T, i_T^!\mathrm{IC}_\lambda) \rightarrow H_T^*(\Gr_I, i_I^!\mathrm{IC}_\lambda) \rightarrow H^*_T(\Gr_G, \mathrm{IC}_\lambda)
    $$
    of free $H_T^*(\mathrm{pt}, \CC) \simeq \operatorname{Sym}(\mathfrak{t}^*)$-modules can be viewed as morphisms of (trivial) vector bundles over the affine space $\operatorname{Spec}{\operatorname{Sym}(\mathfrak{t}^*)} = \mathfrak{t}$. Then, the second morphism restricts to an isomorphism \textit{away from} the walls of the root system of $\check{G}$ which are not walls of the root system of $\check{L}_I$. On the other hand, the first morphism restricts to an isomorphism away from the walls of the root system of $\check{L}_I$.
\end{remark}

\subsection{Monoidal structure on corestriction}\label{monoidal structure section} We will write $L = L_I$ and $i = i_I$ for now. We emphasize that the isomorphism of Theorem~\ref{Main Theorem} is an isomorphism of \textit{algebras} over the ring $R_L := H_L^*(\mathrm{pt},\CC)$. However, our proof of Theorem~\ref{Main Theorem} constructs this isomorphism of $R_L$-\textit{modules} without any appeal to the existence of an algebra structure on the left hand side. The existence of a ring structure on $H^*_{L(\mathcal{O})}(\Gr_L, i^!\F_{\reg})$ appeared to us as an unexpected consequence of Theorem~\ref{Main Theorem}, since the spectral side $\mathcal{O}(T^*(\check{G}/(\check{U}, \psi_I)))$ is naturally a ring under the pointwise multiplication of regular functions. 

The purpose of this section is therefore to equip $H^*_{L(\mathcal{O})}(\Gr_L, i^!\F_{\reg})$ with an $R_L$-algebra structure defined entirely on the geometric side. In the proof of Theorem~\ref{Main Theorem}, we will show that our identification $H_{L(\mathcal{O})}^*(\Gr_L, i^!\F_{\reg}) \simeq \mathcal{O}(T^*(\check{G}/(\check{U}, \psi_I)))$ is in fact an isomorphism of $R_L$-\textit{algebras}. 

We note that the existence of an algebra structure on $H^*_{L(\mathcal{O})}(\Gr_L, i^!\mathcal{F}_\reg)$ is asserted in \cite[Section 5(xi)]{braverman2018ring}. However, we are unaware of a complete construction, so we will provide two in this section (and we will use both). The reader may wish to skim §§\ref{automorphic summary} to see how the results of this section fit together before reading further. 

\begin{remark}\label{ring structure on the regular sheaf}
     Consider the case $L = G$. The $R_G := H_G^*(\mathrm{pt})$-module $H^*_{G(\mathcal{O})}(\F_\reg)$ carries an evident $R_G$-algebra structure which was extensively exploited in \cite{arkhipov2004quantum}. It also plays a prominent role in the construction of Coulomb branches in \cite{braverman2018ring}. Under the geometric Satake equivalence $\mathcal{P}_{G(\mathcal{O})}(\Gr_G) \simeq \operatorname{Rep}(\check{G})$, $\F_{\reg}$ corresponds to the left regular representation $\mathcal{O}(\check{G})$. The pointwise multiplication of regular functions equips $\mathcal{O}(\check{G})$ with the structure of a \textit{ring object} in the monoidal category $\operatorname{Rep}(\check{G})$. Passing back along geometric Satake to perverse sheaves, we find that $\F_{\reg}$ carries a natural ring structure in the monoidal category $\mathcal{P}_{G(\mathcal{O})}(\Gr_G)$. 
    
    Recall that the fiber functor $H^*: \mathcal{P}_{G(\mathcal{O})}(\Gr_G) \rightarrow \mathrm{Mod}(\CC)$ carries a natural monoidal structure \cite[Lemma 6.1]{MV04}. This construction of Mirkovi\'{c}-Vilonen can be adapted (as in \cite[Section 2.3]{yun2009integral} and \cite[Remark 2.5]{yun2009integral}) to produce a monoidal structure on the equivariant cohomology functor $H^*_{G(\mathcal{O})}: \mathcal{P}_{G(\mathcal{O})}(\Gr_G) \rightarrow \operatorname{Mod}(R_G)$. Since a monoidal functor takes ring objects to ring objects, we find a ring structure on $H_{G(\mathcal{O})}^*(\Gr_G, \F_\reg) \in \operatorname{Mod}(R_G)$.
\end{remark}

\begin{remark}
    We will generalize Remark~\ref{ring structure on the regular sheaf} by equipping the object $i^!\F_\reg \in D_{L(\mathcal{O})}(\Gr_L)$ with a ring structure. However, we emphasize that the functor $i^!: D_{G(\mathcal{O})}(\Gr_G) \rightarrow D_{L(\mathcal{O})}(\Gr_L)$ is \textit{not} monoidal. For example, consider the case $L = T$. The functor 
    $$
    i_0^!: D_{T(\mathcal{O})}(\Gr_T) \rightarrow \mathrm{Mod}(R_T)
    $$ 
    of costalk at the basepoint $t^0 \in \Gr_T$ is in fact monoidal. However, the composition $i_0^! \circ i^! \simeq \mathrm{Hom}_{T(\mathcal{O})}(\mathrm{IC}_0, -)$ does not admit a monoidal structure: its restriction to $\mathcal{P}_{G(\mathcal{O})}(\Gr_G) \simeq \mathrm{Rep}(\check{G})$ is naturally isomorphic to the functor $V \mapsto (V \otimes \mathrm{Sym}(\check{\mathfrak{g}}/\check{\mathfrak{u}}))^{\check{B}}$ (by the main result of \cite{GR13}, recalled above as Theorem~\ref{GR13 Main Theorem}). In particular, $i^!$ does not admit a monoidal structure.
    
    Therefore, we will only aim to establish the weaker (but sufficient) result that $i^!$ admits a natural \textit{lax} monoidal structure. This will suffice: lax monoidal functors carry algebras to algebras, so $i^!\F_\reg$ will acquire a ring structure in $D_{L(\mathcal{O})}(\Gr_L)$. 
\end{remark}

\begin{remark}\label{apology}
    It is surely possible to give a more elegant construction of the lax monoidal structure on $i^!$ (even at the ``cochain level'') by appealing to a suitable theory of sheaves and correspondences, such as that of \cite[Section A.5]{mann2022padic}. However, as our Construction~\ref{colax definition} makes essential use of non-invertible natural transformations (specifically, base change morphisms arising from non-Cartesian commutative squares), the kind of formalism described in \cite{mann2022padic} is not sufficient. It would be necessary to use a theory like that of \cite{gaitsgory2016derived}. However, as we do not wish to make our relatively pedestrian computations dependent on the (only partially documented) machinery of $(\infty, 2)$-category theory, we will proceed directly.
\end{remark}

\begin{remark}\label{arbitraryreductivesubgroup}
    In the remainder of this section, we will allow the Levi subgroup $L \subseteq G$ to be replaced by an arbitrary connected reductive subgroup of $G$. However, our only use case in this paper is when $L$ is a Levi subgroup. 
\end{remark}

\begin{remark}\label{lax definition}
    Verdier duality defines an equivalence of categories 
    $$
    \mathbb{D}_G: D_{G(\mathcal{O})}(\Gr_G) \simeq D_{G(\mathcal{O})}(\Gr_G)^{\mathrm{op}}.
    $$
    For objects $A,B \in D_{G(\mathcal{O})}(\Gr_G) $, there is moreover a natural isomorphism
    $$
    \mathbb{D}_G(A \star B) \simeq \mathbb{D}_G(A) \star \mathbb{D}_G(B).
    $$
    See \cite[Lemma 7.2.8]{achar2020} for the analogous fact about the convolution of $B$-equivariant sheaves on the flag variety $G/B$; the proof in our setting is identical. We also have a natural isomorphism $\mathbb{D}_L \circ i^! \simeq i^* \circ \mathbb{D}_G$. Therefore, the problem of equipping $i^!$ with a lax monoidal structure is formally equivalent to that of equipping $i^*$ with a \textit{colax} monoidal structure (and that is what we will do). 
\end{remark}

\begin{remark}\label{convolution}
    In order to fix notation, we recall the definition of the convolution product on $D_{G(\mathcal{O})}(\Gr_G)$ (from \cite[Section 4]{MV04}), together with its associativity constraint (described briefly in \cite[Proposition 4.6]{MV04}). Let $\Gr_G^{(2)} = G(\mathcal{K}) \times^{G(\mathcal{O})} \Gr_G =: \Gr_G \tildetimes \Gr_G$ denote the convolution Grassmannian. Let $m: \Gr_G^{(2)} \rightarrow \Gr_G$ denote the multiplication map. We denote by $- \tildeboxtimes -: D_{G(\mathcal{O})}(\Gr_G) \times D_{G(\mathcal{O})}(\Gr_G) \rightarrow D_{G(\mathcal{O})}(\Gr_G^{(2)})$ the twisted external product. 

    Similarly, we let 
    $$
    \Gr_G^{(3)} = G(\mathcal{K}) \times^{G(\mathcal{O})} G(\mathcal{K}) \times^{G(\mathcal{O})} \Gr_G =: \Gr_G \tildetimes \Gr_G \tildetimes \Gr_G 
    $$
    denote the threefold convolution Grassmannian. Let $m^{23}: \Gr_G^{(3)} \rightarrow \Gr_G^{(2)}$ denote the action map (given by $m^{23}(g_1, g_2, g_3 G(\mathcal{O})) = (g_1, g_2 g_3 G(\mathcal{O}))$), and let $m^{12}: \Gr_G^{(3)} \rightarrow \Gr_G^{(2)}$ denote the multiplication of the first two factors (given by $m^{12}(g_1,g_2,g_3G(\mathcal{O})) = (g_1g_2, g_3G(\mathcal{O}))$). We also have the twisted external product functors:
    \begin{align*}
    - \tildeboxtimes - &: D_{G(\mathcal{O})}(\Gr_G^{(2)}) \times D_{G(\mathcal{O})}(\Gr_G) \rightarrow D_{G(\mathcal{O})}(\Gr_G^{(3)})\\
    - \tildeboxtimes -&: D_{G(\mathcal{O})}(\Gr_G) \times D_{G(\mathcal{O})}(\Gr_G^{(2)}) \rightarrow D_{G(\mathcal{O})}(\Gr_G^{(3)})
    \end{align*}
    Finally, there is a triple multiplication map 
    $$
    m^{123}: \Gr_G^{(3)} \rightarrow \Gr_G
    $$
    given by $m^{123}(g_1, g_2, g_3G(\mathcal{O})) = g_1 g_2 g_3G(\mathcal{O})$, and a triple external product 
    $$
    - \tildeboxtimes - \tildeboxtimes -: D_{G(\mathcal{O})}(\Gr_G) \times D_{G(\mathcal{O})}(\Gr_G) \times D_{G(\mathcal{O})}(\Gr_G) \rightarrow D_{G(\mathcal{O})}(\Gr_G^{(3)}). 
    $$
    The convolution product of $A, B \in D_{G(\mathcal{O})}(\Gr_G)$ is defined by
    $$
    A \star B := m_*(A \tildeboxtimes B). 
    $$
    The triple convolution product of $A,B,C \in D_{G(\mathcal{O})}(\Gr_G)$ is similarly defined by
    $$
    A \star B \star C = m^{123}_*(A \tildeboxtimes B \tildeboxtimes C). 
    $$
    The associativity constaint $\alpha_{A,B,C}$ is a trifunctorial isomorphism
    \begin{equation}\label{associativityconstraint}
    \alpha_{A,B,C}: (A \star B) \star C \simeq A \star (B \star C)
    \end{equation}
    defined as a composition $\alpha_{A,B,C} = \gamma_{A,B,C}^{-1} \circ \beta_{A,B,C}$, where $\beta_{A,B,C}$ is a natural isomorphism
    \begin{equation}\label{betadefinition}
    \beta_{A,B,C}: (A \star B) \star C \simeq A \star B \star C
    \end{equation}
    and $\gamma_{A,B,C}$ is a natural isomorphism
    $$
    \gamma_{A,B,C}: A \star (B \star C) \simeq A \star B \star C.
    $$
    We spell out the definition of $\beta_{A,B,C}$ (that of $\gamma_{A,B,C}$ is analogous). We have $m^{123} = m \circ m^{12}$. Hence, we have a compositional isomorphism $m^{123}_* \simeq m_* \circ m_{*}^{12}$. It defines an isomorphism
    $$
    \mathrm{comp}_{A,B,C}^{12}: m_*m_{*}^{12}(A \tildeboxtimes B \tildeboxtimes C) \simeq m^{123}_*(A \tildeboxtimes B \tildeboxtimes C) = A \star B \star C. 
    $$
    The next ingredient is the associativity constraint on the twisted external product
    $$
    \delta_{A,B,C}: (A \tildeboxtimes B) \tildeboxtimes C \simeq A \tildeboxtimes B \tildeboxtimes C. 
    $$
    Applying $m^{12}_*$ yields an isomorphism
    $$
    m^{12}_* (\delta_{A,B,C}): m^{12}_*((A \tildeboxtimes B) \tildeboxtimes C) \simeq m_{*}^{12}(A \tildeboxtimes B \tildeboxtimes C). 
    $$
    For any objects $E\in D_{G(\mathcal{O})}(\Gr_G^{(2)})$, $F \in D_{G(\mathcal{O})}(\Gr_G)$, we have an isomorphism
    \begin{equation}\label{epsilondefinition}
    \epsilon_{E,F}: m_*E \tildeboxtimes F \simeq  m^{12}_*(E \tildeboxtimes F).
    \end{equation}
    By definition, $\epsilon_{E,F}$ is obtained by adjunction from the composition of the evident natural maps 
    $$
    m^{12,*}(m_*E \tildeboxtimes F) \xrightarrow{\xi^{12}_{E,F}} m^*m_*E \tildeboxtimes F \xrightarrow{\eta_E \tildeboxtimes 1} E \tildeboxtimes F. 
    $$
    We may now take $E = A \tildeboxtimes B$ and $F = C$ to obtain the isomorphism
    $$
    \epsilon_{A,B,C} := \epsilon_{A\tildeboxtimes B,C}: m_*(A \tildeboxtimes B) \tildeboxtimes C \simeq m_{*}^{12}((A \tildeboxtimes B) \tildeboxtimes C). 
    $$
    Let 
    \begin{equation}\label{sigmadefinition}
    \sigma_{A,B,C} = m_*( m_{*}^{12}(\delta_{A,B,C}) \circ \epsilon_{A,B,C}). 
    \end{equation}
    The isomorphism $\beta_{A,B,C}$ is now defined to be the composition
    $$
    \beta_{A,B,C} = \mathrm{comp}^{12}_{A,B,C} \circ \sigma_{A,B,C}.
    $$
    When dealing with both groups $G$ and $L$, we will generally use the same notation for the analogous maps and isomorphisms. 
\end{remark}

\begin{remark}\label{convolution with two groups} In addition to the morphism $i: \Gr_L \rightarrow \Gr_G$, we obtain morphisms $i_{2}: \Gr^{(2)}_L \rightarrow \Gr^{(2)}_G$ and $i_3: \Gr^{(3)}_L \rightarrow \Gr^{(3)}_G$. We have the identities $m \circ i_{2} = i \circ m$, $m^{12} \circ i_3 = i_2 \circ m^{12}$, $m^{23} \circ i_3 = i_2 \circ m^{23}$, and $m^{123} \circ i_3 = i \circ m^{123}$. We also have evident natural isomorphisms
\begin{align}\label{omegadefinition}
\omega_{A,B} &: i_2^*(A \tildeboxtimes B) \simeq i^*A \tildeboxtimes i^*B\\
\omega_{A,B,C} &: i_3^*(A \tildeboxtimes B \tildeboxtimes C) \simeq i^*A \tildeboxtimes i^*B \tildeboxtimes i^*C.
\end{align}
    
\end{remark}

\begin{construction}\label{colax definition}
    We can now define the (non-unital part of the) colax monoidal structure on $i^*: D_{G(\mathcal{O})}(\Gr_G) \rightarrow D_{L(\mathcal{O})}(\Gr_L)$. For objects $A,B \in D_{G(\mathcal{O})}(\Gr_G)$, we must define a bifunctorial map
    $$
    \theta_{A,B}: i^*(A \star B) \rightarrow i^* A \star i^*B. 
    $$
    By definition, $A \star B = m_*(A \tildeboxtimes B)$. Since $m \circ i_2 = i \circ m$, we deduce the existence of a base change morphism $\rho: i^*m_* \rightarrow m_*i_2^*$. Hence, we have a natural map
    \begin{equation}\label{rhodefinition}
    \rho_{A,B,C} := \rho_{A \tildeboxtimes B}: i^*m_*(A \tildeboxtimes B) \rightarrow m_*i_2^*(A \tildeboxtimes B). 
    \end{equation}
    We also have the natural isomorphism $\omega_{A,B}: i_2^*(A \tildeboxtimes B) \simeq i^*A \tildeboxtimes i^*B$. Thus, we can define 
    \begin{equation}\label{thetadefinition}
    \theta_{A,B} = m_*(\omega_{A,B}) \circ \rho_{A,B,C}
    \end{equation}
\end{construction}
The proof of the following proposition is the price that we pay for Remark~\ref{apology}, i.e. our choice to construct $\theta_{A,B}$ directly without the use of a higher categorical machinery that would presumably encode the commutativity of the diagrams below implicitly. The reader is encouraged to skip it on a first reading.

\begin{proposition}\label{associativity}
    The natural transformation $\theta: i^*(- \star -) \rightarrow i^*(-) \star i^*(-)$ of functors $D_{G(\mathcal{O})}(\Gr_G) \times D_{G(\mathcal{O})}(\Gr_G) \rightarrow D_{L(\mathcal{O})}(\Gr_G)$ of Construction~\ref{colax definition} is compatible with the associativity constraints of Remark~\ref{convolution} underlying the monoidal categories $D_{G(\mathcal{O})}(\Gr_G)$ and $D_{L(\mathcal{O})}(\Gr_L)$. That is, for objects $A,B,C \in D_{G(\mathcal{O})}(\Gr_G)$, the following diagram in $D_{L(\mathcal{O})}(\Gr_L)$ commutes:
    $$
    \begin{tikzcd}
        i^*((A \star B) \star C) \ar[rr,"i^*\alpha_{A,B,C}"] \ar[d, "\theta_{A \star B,C}"] & & i^*(A \star (B \star C)) \ar[d, "\theta_{A, B \star C}"]\\
        i^*(A \star B) \star i^*C \ar[d, "\operatorname{\theta}_{A,B} \star \operatorname{id}"] & & i^*A \star i^*(B \star C) \ar[d, "\operatorname{id} \star \operatorname{\theta}_{B,C}"] \\
        (i^*A \star i^*B) \star i^*C \ar[rr, "\alpha_{i^*A,i^*B,i^*C}"] & & i^*A \star (i^*B \star i^*C)
    \end{tikzcd}
    $$
\end{proposition}
\begin{proof}
    We expand the diagram horizontally, using the definition~\eqref{associativityconstraint} of the associativity isomorphism $\alpha$:
    $$
    \begin{tikzcd}
        i^*((A \star B) \star C) \ar[rr,"i^*\beta_{A,B,C}"] \ar[d, "\theta_{A \star B,C}"] & & i^*(A \star B \star C) \ar[rr, "i^*\gamma_{A,B,C}^{-1}"] \ar[dd, "\theta_{A,B,C}"]& & i^*(A \star (B \star C)) \ar[d, "\theta_{A, B \star C}"]\\
        i^*(A \star B) \star i^*C \ar[d, "\operatorname{\theta}_{A,B} \star \operatorname{id}"] & & & & i^*A \star i^*(B \star C) \ar[d, "\operatorname{id} \star \operatorname{\theta}_{B,C}"] \\
        (i^*A \star i^*B) \star i^*C \ar[rr, "\beta_{i^*A,i^*B,i^*C}"] & & i^*A \star i^*B \star i^*C \ar[rr, "\gamma_{i^*A,i^*B,i^*C}^{-1}"] & & i^*A \star (i^*B \star i^*C)
    \end{tikzcd}
    $$
    It suffices to show that both of the small rectangles in this diagram commute. We will show that the one on the left commutes; the proof that the one on the right commutes is completely analogous. We now expand the diagram on the left, using the definitions \eqref{thetadefinition} and \eqref{betadefinition} of $\theta$ and $\beta$, respectively:
    $$
    \begin{tikzcd}
        i^*m_*(m_*(A \tildeboxtimes B) \tildeboxtimes C) \ar[rr,"i^*m_*\sigma_{A,B,C}"] \ar[d, "\rho_{m_*(A \tildeboxtimes B) \tildeboxtimes C}"] & & i^*m_*m_*^{12}(A \tildeboxtimes B \tildeboxtimes C) \ar[rr, "\mathrm{comp}^{12}_{A,B,C}"] \ar[d,"\rho_{m^{12}_*(A \tildeboxtimes B \tildeboxtimes C)}"] && i^*m^{123}_*(A \tildeboxtimes B \tildeboxtimes C) \ar[dd, "\rho^{123}_{A \tildeboxtimes B \tildeboxtimes C}",swap] \\
        m_*i_{2}^*(m_*(A \tildeboxtimes B) \tildeboxtimes C) \ar[d, "m_*\omega_{m_*(A \tildeboxtimes B), C}"] \ar[rr, "m_*i_2^*\sigma_{A,B,C}"] & & m_*i_2^*m_*^{12}(A \tildeboxtimes B \tildeboxtimes C) \ar[d, "m_*\rho^{12,3}_{A \tildeboxtimes B \tildeboxtimes C}"] \\
        m_*(i^*m_*(A \tildeboxtimes B) \tildeboxtimes i^*C) \ar[d, "m_*(\rho_{A,B} \tildeboxtimes 1)"] & & m_*m_*^{12}i_3^*(A \tildeboxtimes B \tildeboxtimes C) \ar[rr, "\mathrm{comp}^{12}_{i_3^*(A \tildeboxtimes B \tildeboxtimes C)}"] \ar[dd, "m_*m_*^{12}\omega_{A,B,C}"] & & m_{*}^{123} i_3^*(A \tildeboxtimes B \tildeboxtimes C) \ar[dd, "m^{123}_* \omega_{A,B,C}",swap] \\
        m_*(m_*i_2^*(A \tildeboxtimes B) \tildeboxtimes i^*C) \ar[d, "m_*(m_*\omega_{A,B} \tildeboxtimes 1)"] \\
        m_*(m_*(i^*A \tildeboxtimes i^*B) \tildeboxtimes i^*C) \ar[rr, "m_*\sigma_{i^*A,i^*B,i^*C}"] & & m_*m_*^{12}(i^*A \tildeboxtimes i^*B \tildeboxtimes i^*C) \ar[rr, "\mathrm{comp}_{i^*A,i^*B,i^*C}^{12}"] & & m^{123}_*(i^*A \tildeboxtimes i^*B \tildeboxtimes i^*C)
    \end{tikzcd}
    $$
    Here, we have introduced the base change maps $\rho^{12,3}: i_2^*m_*^{12} \rightarrow m_*^{12}i_3^*$ and $\rho^{123}: i^*m_*^{123} \rightarrow m_*^{123}i_3^*$ deduced from the equalities $i_2m^{12}=m^{12}i_3$ and $i m^{123} = m^{123}i_3$, respectively. The upper left rectangle commutes by the naturality of the base change transformation $\rho: i^*m_* \rightarrow m_* i_2^*$ of \eqref{rhodefinition} applied to the isomorphism $\sigma_{A,B,C}$. The upper right rectangle commutes by the compatiblity of compositional isomorphisms with base change (see Proposition 2.11.7 of \cite{achar2020} and note that the Cartesian hypothesis is superfluous). The lower right rectangle commutes by the naturality of the compositional isomorphism $\mathrm{comp}^{12}: m_*m_*^{12} \simeq m_*^{123}$ applied to the isomorphism $\omega_{A,B,C}$. 
    
    It therefore suffices to show that the lower left rectangle commutes. It is obtained by applying $m_*$ to the the following diagram:
    $$
    \begin{tikzcd}
        i_{2}^*(m_*(A \tildeboxtimes B) \tildeboxtimes C) \ar[d, "\omega_{m_*(A \tildeboxtimes B), C}"] \ar[rr, "i_2^*\sigma_{A,B,C}"] & & i_2^*m_*^{12}(A \tildeboxtimes B \tildeboxtimes C) \ar[d, "\rho^{12,3}_{A \tildeboxtimes B \tildeboxtimes C}"] \\
        i^*m_*(A \tildeboxtimes B) \tildeboxtimes i^*C \ar[d, "\rho_{A,B} \tildeboxtimes 1"] & & m_*^{12}i_3^*(A \tildeboxtimes B \tildeboxtimes C) \ar[dd, "m_*^{12}\omega_{A,B,C}"] \\
        m_*i_2^*(A \tildeboxtimes B) \tildeboxtimes i^*C \ar[d, "m_*\omega_{A,B} \tildeboxtimes 1"] & & \\
        m_*(i^*A \tildeboxtimes i^*B) \tildeboxtimes i^*C \ar[rr, "\sigma_{i^*A,i^*B,i^*C}"] & & m_*^{12}(i^*A \tildeboxtimes i^*B \tildeboxtimes i^*C)
    \end{tikzcd}
    $$
    To show that the above diagram commutes, we expand it horizontally using the definition~\eqref{sigmadefinition} of $\sigma$ to obtain:
    $$\hspace{-0.5cm}
    \begin{tikzcd}
        i_{2}^*(m_*(A \tildeboxtimes B) \tildeboxtimes C) \ar[d, "\omega_{m_*(A \tildeboxtimes B), C}"] \ar[rr, "i_2^*\epsilon_{A,B,C}"] & & i_2^*m_*^{12}((A \tildeboxtimes B) \tildeboxtimes C) \ar[rr, "i_2^*m_*^{12}\delta_{A,B,C}"] \ar[d, "\rho^{12,3}_{(A \tildeboxtimes B) \tildeboxtimes C}"] & & i_2^*m_*^{12}(A \tildeboxtimes B \tildeboxtimes C) \ar[d, "\rho^{12,3}_{A \tildeboxtimes B \tildeboxtimes C}"] \\
        i^*m_*(A \tildeboxtimes B) \tildeboxtimes i^*C \ar[d, "\rho_{A,B} \tildeboxtimes 1"] & (\star) & m_*^{12}i_3^*((A \tildeboxtimes B) \tildeboxtimes C) \ar[rr, "m^{12}_* i_3^*\delta_{A,B,C}"] \ar[d, "m_*^{12}\omega_{A \tildeboxtimes B, C}"] & & m_*^{12}i_3^*(A \tildeboxtimes B \tildeboxtimes C) \ar[dd, "m_*^{12}\omega_{A,B,C}"] \\
        m_*i_2^*(A \tildeboxtimes B) \tildeboxtimes i^*C \ar[d, "m_*\omega_{A,B} \tildeboxtimes 1"] \ar[rr, "\epsilon_{i_2^*(A \tildeboxtimes B), i^*C}"] & & m_*^{12}(i_2^*(A \tildeboxtimes B) \tildeboxtimes i^*C) \ar[d, "m_*^{12}(\omega_{A,B} \tildeboxtimes 1)"] & &\\
        m_*(i^*A \tildeboxtimes i^*B) \tildeboxtimes i^*C \ar[rr, "\epsilon_{i^*A,i^*B,i^*C}"] & & m_*^{12}((i^*A \tildeboxtimes i^*B) \tildeboxtimes i^*C) \ar[rr, "m_{*}^{12}\delta_{i^*A,i^*B,i^*C}"] & & m_*^{12}(i^*A \tildeboxtimes i^*B \tildeboxtimes i^*C)
    \end{tikzcd}
    $$
    The upper right rectangle commutes by the naturality of the base change map $\rho^{12,3}: i_2^*m_*^{12} \rightarrow m^{12}_*i^*_3$ applied to the isomorphism $\delta_{A,B,C}$. The lower left rectangle commutes by the naturality of the isomorphism $\epsilon_{-,i^*C}: m_*(-) \tildeboxtimes i^*C \simeq m_*^{12}(- \tildeboxtimes i^*C)$ of \eqref{epsilondefinition} applied to the isomorphism $\omega_{A,B}: i_2^*(A \tildeboxtimes B) \simeq i_2^*A \tildeboxtimes i_2^*B$ of \eqref{omegadefinition}.  The lower right rectangle commutes because it is obtained by applying the functor $m_*^{12}$ to the commutative diagram:
    $$
    \begin{tikzcd}
        i_3^*((A \tildeboxtimes B) \tildeboxtimes C) \ar[rr, " i_3^*\delta_{A,B,C}"] \ar[d, "\omega_{A \tildeboxtimes B, C}"] & & i_3^*(A \tildeboxtimes B \tildeboxtimes C) \ar[dd, "\omega_{A,B,C}"] \\
        i_2^*(A \tildeboxtimes B) \tildeboxtimes i^*C \ar[d, "\omega_{A,B} \tildeboxtimes 1"] & & \\
        (i^*A \tildeboxtimes i^*B) \tildeboxtimes i^*C \ar[rr, "\delta_{i^*A,i^*B,i^*C}"] & & i^*A \tildeboxtimes i^*B \tildeboxtimes i^*C
    \end{tikzcd}
    $$
    It remains to show that the upper left rectangle $(\star)$ commutes, which follows from the more general assertion that for $E \in D_{G(\mathcal{O})}(\Gr_G^{(2)})$, $F \in D_{G(\mathcal{O})}(\Gr_G)$, the following diagram commutes:
    $$
    \begin{tikzcd}
        i_{2}^*(m_*E \tildeboxtimes F) \ar[d, "\omega_{m_*E, F}"] \ar[rr, "i_2^*\epsilon_{E,F}"] & & i_2^*m_*^{12}(E \tildeboxtimes F) \ar[d, "\rho^{12,3}_{E \tildeboxtimes F}"] \\
        i^*m_*E \tildeboxtimes i^*F \ar[d, "\rho_{E} \tildeboxtimes 1"] & & m_*^{12}i_3^*(E \tildeboxtimes F) \ar[d, "m^{12}_*\omega_{E,F}"] \\
        m_*i_2^*E \tildeboxtimes i^*F \ar[rr, "\epsilon_{i_2^*E, i^*F}"] & & m_*^{12}(i_2^*E \tildeboxtimes i^*F)
    \end{tikzcd}
    $$
    By adjunction (and the definition~\eqref{epsilondefinition} of $\epsilon$, which is used below to expand the bottom row horizontally), it suffices to show that the outer rectangle in the following diagram commutes (note that the middle vertical arrow $\mathrm{comp}$ refers to the compositional isomorphism $m^{12,*}i_2^* \simeq i_3^*m^{12,*}$ evaluated on $i_2^*m_*^{12}(E \tildeboxtimes F)$):
    $$
    \begin{tikzcd}
        m^{12,*}i_{2}^*(m_*E \tildeboxtimes F) \ar[dd, "m^{12,*}\omega_{m_*E, F}"] \ar[rr, "m^{12,*}i_2^*\epsilon_{E,F}"] & & m^{12,*}i_2^*m_*^{12}(E \tildeboxtimes F) \ar[rr, equals] \ar[d,"\mathrm{comp}"] & & m^{12,*}i_2^*m_*^{12}(E \tildeboxtimes F) \ar[d, "m^{12,*}\rho^{12,3}_{E \tildeboxtimes F}"] \\
        & & i_3^*m^{12,*}m_{*}^{12}(E \tildeboxtimes F) \ar[d, "i_3^*\eta^{12}_{E \tildeboxtimes F}"] & & m^{12,*}m^{12}_*i_3^*(E \tildeboxtimes F) \ar[d, "\eta^{12}_{i_3^*(E \tildeboxtimes F)}"] \\
        m^{12,*}(i^*m_*E \tildeboxtimes i^*F) \ar[d, "m^{12,*}(\rho^{12}_{E} \tildeboxtimes 1)"] & & i_3^*(E \tildeboxtimes F) \ar[rr, equals] & & i_3^*(E \tildeboxtimes F) \ar[d, "\omega_{E,F}"] \\
        m^{12,*}(m_*i_2^*E \tildeboxtimes i^*F) \ar[rr, "\xi^{12}_{m_*i_2^*E, i^*F}"] & & m^{*}m_*i_2^*E \tildeboxtimes i^*F \ar[rr, "\eta_{i_2^*E} \tildeboxtimes 1"] & & i_2^*E \tildeboxtimes i^*F
    \end{tikzcd}
    $$
    The upper right rectangle commutes by the very definition of the base change morphism $\rho^{12,3}_{E \tildeboxtimes F}$. Thus, it suffices to show that the outer rectangle in the following diagram commutes.
    $$
    \begin{tikzcd}
        m^{12,*}i_{2}^*(m_*E \tildeboxtimes F) \ar[dd, "m^{12,*}\omega_{m_*E, F}"] \ar[rr, equals] & & m^{12,*}i_2^*(m_*E \tildeboxtimes F) \ar[rr, "m^{12,*}i_2^*\epsilon_{E,F}"] \ar[d, "\mathrm{comp}"] & & m^{12,*}i_2^*m_*^{12}(E \tildeboxtimes F) \ar[d,"\mathrm{comp}"] \\
        & & i_3^*m^{12,*}(m_*E \tildeboxtimes F) \ar[d, "i_3^*\xi^{12}_{m_*E, F}"] \ar[rr, "i_3^*m^{12,*}\epsilon_{E,F}"] & &  i_3^*m^{12,*}m_{*}^{12}(E \tildeboxtimes F) \ar[d, "i_3^*\eta^{12}_{E \tildeboxtimes F}"] \\
        m^{12,*}(i^*m_*E \tildeboxtimes i^*F) \ar[dd, "m^{12,*}(\rho^{12}_{E} \tildeboxtimes 1)"] & & i_3^*(m^*m_*E \tildeboxtimes F) \ar[rr, "i_3^*(\eta_E \tildeboxtimes 1)"] \ar[d, "\omega_{m^*m_*E, F}"] & & i_3^*(E \tildeboxtimes F) \ar[d, "\omega_{E,F}"] \\
        & & i_2^*m^*m_*E \tildeboxtimes F \ar[rr, "i_2^*\eta_E \tildeboxtimes 1"] & & i_2^*E \tildeboxtimes i^*F \ar[d, equals]\\
        m^{12,*}(m_*i_2^*E \tildeboxtimes i^*F) \ar[rr, "\xi^{12}_{m_*i_2^*E, i^*F}"] & & m^*m_* i_2^*E \tildeboxtimes i^*F \ar[rr, "\eta_{i_2^*E} \tildeboxtimes 1"] & & i_2^*E \tildeboxtimes i^*F
    \end{tikzcd}
    $$
    The upper right rectangle commutes by the naturality of the isomorphism $\mathrm{comp}: m^{12,*}i_2^* \simeq i_3^*m^{12,*}$. The rectangle below it commutes by ($i_3^*$ applied to) the definition~\eqref{epsilondefinition} of $\epsilon_{E,F}$. Hence, it suffices to show that the outer rectangle in the following diagram commutes:
    $$
    \begin{tikzcd}
        m^{12,*}i_{2}^*(m_*E \tildeboxtimes F) \ar[dd, "m^{12,*}\omega_{m_*E, F}"] \ar[rrrr, "\mathrm{comp}"] & & & & i_3^*m^{12,*}(m_*E \tildeboxtimes F) \ar[d, "i_3^*\xi^{12}_{m_*E, F}"] \\
         & & (\dagger) & &  i_3^*(m^*m_*E \tildeboxtimes F) \ar[d, "\omega_{m^*m_*E, F}"] \\
        m^{12,*}(i_2^*m_*E \tildeboxtimes i^*F) \ar[d, "m^{12,*}(\rho^{12}_{E} \tildeboxtimes 1)"] \ar[rr, "\xi^{12}_{i_2^*m_*E, i^*F}"] & & m^*i_2^*m_*E \tildeboxtimes i^*F \ar[rr, "\mathrm{comp} \tildeboxtimes 1"] \ar[d, "m^*\rho_E \tildeboxtimes 1"] & & i_2^*m^*m_*E \tildeboxtimes i^*F \ar[d, "i_2^*\eta_E \tildeboxtimes 1"]\\
        m^{12,*}(m_*i_2^*E \tildeboxtimes i^*F) \ar[rr, "\xi^{12}_{m_*i_2^*E, i^*F}"] && m^*m_*i_2^*E \tildeboxtimes i^*F \ar[rr, "\eta_{i_2^*E} \tildeboxtimes 1"] && i_2^*E \tildeboxtimes i^*F
    \end{tikzcd}
    $$
    The lower left rectangle commutes by the naturality of the isomorphism $\xi^{12}: m^{12,*}(- \tildeboxtimes -) \simeq m^*(-) \tildeboxtimes -$ applied to the map $\rho_E^{12} \tildeboxtimes 1$. The lower right rectangle is obtained by applying the functor $- \tildeboxtimes i^*F$ to the diagram
    $$
    \begin{tikzcd}
        m^*i_2^*m_*E \ar[rr, "\mathrm{comp}"] \ar[d, "m^*\rho^{12}_E"] & & i_2^*m^*m_*E \ar[d, "i_2^*\eta_E"] \\
        m^*m_*i_2^*E \ar[rr, "\eta_{i_2^*E}"] & & i_2^*E.
    \end{tikzcd}
    $$
    This square commutes by the very definition of the base change map $\rho^{12}_E: i_2^*m_*E \rightarrow m_*i_2^*E$. We must verify the commutativity of the upper rectangle $(\dagger)$. The claim is that for $X,Y \in D_{G(\mathcal{O})}(\Gr_G)$, the following diagram commutes.
    $$
    \begin{tikzcd}
        m^{12,*}i_{2}^*(X \tildeboxtimes Y) \ar[dd, "m^{12,*}\omega_{X, Y}"] \ar[rrrr, "\mathrm{comp}"] & & & & i_3^*m^{12,*}(X \tildeboxtimes Y) \ar[d, "i_3^*\xi^{12}_{X, Y}"] \\
         & & & &  i_3^*(m^*X \tildeboxtimes Y) \ar[d, "\omega_{m^*X, Y}"] \\
        m^{12,*}(i_2^*X \tildeboxtimes Y) \ar[rr, "\xi^{12}_{i_2^*X, i^*Y}"] & & m^*i_2^*X \tildeboxtimes i^*Y \ar[rr, "\mathrm{comp} \tildeboxtimes 1"] & & i_2^*m^*X \tildeboxtimes i^*Y \\
    \end{tikzcd}
    $$
    Checking the commutativity of this diagram is left to the reader (reduce to verifying the commutativity of the corresponding diagram with untwisted external products, which follows from a straightforward compatiblity between the isomorphism $f^*(- \otimes -) \simeq f^*(-) \otimes f^*(-)$ and the compositional isomorphisms $(gf)^* \simeq f^*g^*$). 
\end{proof}

Next, we turn to the following problem. Suppose that $M \subseteq L$ is a connected reductive subgroup of $L$ (the reader should keep in the mind the case in which $M \subseteq L$ is an inclusion of Levi subgroups of $G$; for instance, when $M = T$ is a maximal torus). Let $j: \Gr_M \hookrightarrow \Gr_L$ denote the induced map on affine Grassmannians. We also have the composition $k = i \circ j: \Gr_M \hookrightarrow \Gr_G$. In Construction~\ref{colax definition}, we equipped the functors $i^*$, $j^*$, and $k^*$ with (non-unital, for now) colax monoidal structures. On the other hand, we have a canonical isomorphism $k^* \simeq j^* \circ i^*$. We would like to show that this natural isomorphism is an isomorphism of (non-unital) colax monoidal functors.

\begin{proposition}\label{transitivity of restriction}
    Let $j: \Gr_M \hookrightarrow \Gr_L$, $i: \Gr_L \hookrightarrow \Gr_G$, and $k = i \circ j: \Gr_M \hookrightarrow \Gr_G$ be as above. Equip the functors $i^*$, $j^*$, and $k^*$ with the (non-unital) colax monoidal structures of Construction~\ref{colax definition}. Then, the compositional isomorphism $k^* \simeq j^* \circ i^*$ is an isomorphism of (non-unital) colax monoidal functors. More precisely, let $A,B \in D_{G(\mathcal{O})}(\Gr_G)$. Then, the following diagram in $D_{M(\mathcal{O})}(\Gr_M)$ commutes:
    \begin{equation}\label{transitivity of restriction diagram0}
    \begin{tikzcd}
        j^*i^* (A \star B) \ar[d, "\mathrm{comp}"] \ar[r, "j^*\theta_{A,B}"] & j^*(i^*A \star i^*B) \ar[r, "\theta_{i^*A,i^*B}"] & j^*i^*A \star j^*i^*B \ar[d, "\operatorname{comp} \star \operatorname{comp}"] \\
        k^*(A \star B) \ar[rr, "\theta_{A,B}"] & & k^*A \star k^*B
    \end{tikzcd}
    \end{equation}
\end{proposition}
\begin{proof}
    We will notationally suppress all composition isomorphisms in the following argument (there are several that appear, but there is never ambiguity about which is meant). We can more explicitly express the above diagram \eqref{transitivity of restriction diagram0} as follows.
    \begin{equation}
    \begin{tikzcd}\label{transitivity of restriction diagram1}
        j^*i^*m_*(A \tildeboxtimes B) \ar[d] \ar[r, "j^*\theta_{A,B}"] & j^*m_*(i^*A \tildeboxtimes i^*B) \ar[r, "\theta_{i^*A,i^*B}"] & m_*(j^*i^*A \tildeboxtimes j^*i^*B) \ar[d] \\
        k^*m_*(A \tildeboxtimes B) \ar[rr, "\theta_{A,B}"] & & m_*(k^*A \tildeboxtimes k^*B)
    \end{tikzcd}
    \end{equation}
    The following diagram commutes by the compatibility of composition with base change (Proposition 2.11.7 of \cite{achar2020}).
    \begin{equation}
    \begin{tikzcd}\label{transitivity of restriction diagram2}
        j^*i^*m_*(A \tildeboxtimes B) \ar[d] \ar[r, "j^*\rho_{A \tildeboxtimes B}"] & j^*m_*i_2^*(A \tildeboxtimes B) \ar[r,"\rho_{i_2^*(A \tildeboxtimes B)}"] & m_*j_2^*i_2^*(A \tildeboxtimes B) \ar[d] \\
         k^*m_*(A \tildeboxtimes B) \ar[rr,"\rho_{A \tildeboxtimes B}"] & & m_*k_2^*(A \tildeboxtimes B)
    \end{tikzcd}
    \end{equation}
    The following diagram commutes because it is obtained by applying the functor $m_*$ to the diagram which witnesses the compatibility of $\omega_{A,B}$ with composition (see the last step in the proof of Proposition~\ref{associativity}).
    \begin{equation}
    \begin{tikzcd}\label{transitivity of restriction diagram3}
        m_*j_2^*i_2^*(A \tildeboxtimes B) \ar[d] \ar[r, "m_*\omega_{A,B}"] & m_*j_2^*(i^*A \tildeboxtimes i^*B) \ar[r, "m_*\omega_{i^*A, i^*B}"] & m_*(j^*i^*A \tildeboxtimes j^*i^*B) \ar[d]\\
        m_*k_2^*(A \tildeboxtimes B) \ar[rr,"m_*\omega_{A,B}"] & & m_*(k^*A \tildeboxtimes k^*B). 
    \end{tikzcd}
    \end{equation}
    Pasting diagrams \eqref{transitivity of restriction diagram2} and \eqref{transitivity of restriction diagram3} together horizontally yields a commutative rectangle in which the left vertical, right vertical, and bottom horizontal arrows are equal to those of the original diagram \eqref{transitivity of restriction diagram1}. It therefore suffices to show that the top horizontal arrows coincide. Since the first and last morphisms in these four-fold compositions are equal, the claim amounts to verifying the commutativity of the following diagram.
    $$
    \begin{tikzcd}
    j^*m_*i_2^*(A \tildeboxtimes B) \ar[r] \ar[d] & j^*m_*(i^*A \tildeboxtimes i^*B)  \ar[d]\\
    m_*j_2^*i_2^*(A \tildeboxtimes B)\ar[r] & m_*j_2^*(i^*A \tildeboxtimes i^*B)
    \end{tikzcd}
    $$
    This diagram commutes by the naturality of the base change map $j^*m_* \rightarrow m_*j_2^*$ applied to the isomorphism $i_2^*(A \tildeboxtimes B) \simeq i^*A \tildeboxtimes i^*B$. 
\end{proof}

\begin{remark}
    Recall (for example, from \cite[\href{https://kerodon.net/tag/00CC}{Tag 00CC}]{kerodon}) that to give a monoidal structure on a category $\mathcal{C}$, it suffices to give an associative product $\star$ on $\mathcal{C}$ together with a \textit{unit object} of $\mathcal{C}$ (with respect to $\star$), that is, an object $1_\mathcal{C} \in \mathcal{C}$ equipped with an isomorphism
    $$
    1_{\mathcal{C}} \star 1_{\mathcal{C}} \simeq 1_{\mathcal{C}}.
    $$
    We use this simplification below when discussing the unit constraint on the derived Satake category. 
\end{remark}

\begin{remark}
    We recall the definition of the unit constraint on $(D_{G(\mathcal{O})}(\Gr_G), \star)$ (see \cite[Lemma 9.2.2]{achar2020}). Let $\mathrm{IC}_0 = \underline{\CC}_{\Gr_G^0}$ denote the denote ${G}(\mathcal{O})$-equivariant skyscraper at the basepoint $t^0 \in \Gr_G$ and let $j: \mathrm{pt} \hookrightarrow \Gr_G$ denote the closed inclusion of $t^0$. Let $j_2: \mathrm{pt} \hookrightarrow \Gr_G \tildetimes \Gr_G$ denote the inclusion of the point $\Gr_G^0 \tildetimes \Gr_G^0$. Firstly, there is an isomorphism
    $$
    \xi: \mathrm{IC}_0 \tildeboxtimes \mathrm{IC}_0
    = j_*\underline{\CC}_{\mathrm{pt}} \tildeboxtimes j_*\underline{\CC}_{\mathrm{pt}} \simeq (j_2)_*(\underline{\CC}_{\mathrm{pt}} \boxtimes \underline{\CC}_{\mathrm{pt}}) \simeq (j_2)_*\underline{\CC}_{\mathrm{pt}}.
    $$
    Since $m \circ j_2 = j$, we have an isomorphism $(m j_2)_* \simeq m_*(j_2)_{*}$. Applying $m_*$ to $\xi$ therefore yields an isomorphism
    $$
    \eta: \mathrm{IC}_0 \star \mathrm{IC}_0 \simeq m_*(j_2)_* \underline{\CC}_{\mathrm{pt}} \simeq j_*\underline{\CC}_{\mathrm{pt}} = \mathrm{IC}_0. 
    $$
    The isomorphism $\eta$ is the unit contraint. 
\end{remark}

\begin{construction}\label{unit}
    Let $\mathrm{IC}_0' \in D_{L(\mathcal{O})}(\Gr_L)$ denote the monoidal unit and $j': \mathrm{pt} \hookrightarrow \Gr_L$ the inclusion of the basepoint. We have an isomorphism $\chi: i^*\mathrm{IC}_0 \xrightarrow{\sim} \mathrm{IC}_0'$, which can be precisely defined as the composition 
    $$
    \chi: i^*\mathrm{IC}_0 = i^* j_*\underline{\CC}_{\mathrm{pt}} \simeq i^*i_*j'_*\underline{\CC}_{\mathrm{pt}} \xrightarrow{\sim} j'_*\underline{\CC}_{\mathrm{pt}} = \mathrm{IC}_0'. 
    $$
\end{construction}

\begin{proposition}\label{unitality of restriction}
    The natural transformation $\theta: i^*(- \star -) \rightarrow i^*(-) \star i^*(-)$ of functors $D_{G(\mathcal{O})}(\Gr_G) \times D_{G(\mathcal{O})}(\Gr_G) \rightarrow D_{L(\mathcal{O})}(\Gr_G)$ defined in Construction~\ref{colax definition} is compatible with the unit constraints underlying the monoidal categories $D_{G(\mathcal{O})}(\Gr_G)$ and $D_{L(\mathcal{O})}(\Gr_L)$. More precisely, the following diagram in $D_{L(\mathcal{O})}(\Gr_L)$ commutes:
    \begin{equation}\label{unitality of restriction diagram0}
    \begin{tikzcd}
        i^*(\mathrm{IC}_0 \star \mathrm{IC}_0) \ar[rr, "i^*\eta"]\ar[d, "\theta_{\mathrm{IC}_0, \mathrm{IC}_0}"] & & i^*\mathrm{IC}_0  \ar[d, "\chi"] \\
        i^*\mathrm{IC}_0 \star i^*\mathrm{IC}_0 \ar[r, "\chi \star \chi"] & \mathrm{IC}'_0 \star \mathrm{IC}'_0 \ar[r, "\eta"] & \mathrm{IC}_0'
    \end{tikzcd}
    \end{equation}
\end{proposition}
\begin{proof}
    It is at least obvious that all of the terms in \eqref{unitality of restriction diagram0} are isomorphic and supported at $t^0 \in \Gr_L$. Hence, it suffices to prove the commutativity of the diagram after applying the functor $j^*$, where $j: \mathrm{pt} \rightarrow \Gr_L$ denotes the inclusion of the basepoint. Let $k = ji : \mathrm{pt} \rightarrow \Gr_G$ denote the composition. The diagram $j^*$\eqref{unitality of restriction diagram0} can be embedded into the following diagram:
    $$
    \begin{tikzcd}
        k^*(\mathrm{IC}_0 \star \mathrm{IC}_0) \ar[rrrr, "k^*\eta"]\ar[d, "\mathrm{comp}"] & & & & k^*(\mathrm{IC}_0) \ar[d, "\mathrm{comp}"]\\
        j^*i^*(\mathrm{IC}_0 \star \mathrm{IC}_0) \ar[rrrr, "j^*i^*\eta"]\ar[d,"j^*\theta_{\mathrm{IC}_0, \mathrm{IC}_0}"] & & & &j^*i^*\mathrm{IC}_0  \ar[d, "j^*\chi"] \\
        j^*(i^*\mathrm{IC}_0 \star i^*\mathrm{IC}_0) \ar[rr, "j^*(\chi\, \star \, \chi)"] \ar[d, "\theta_{i^*\mathrm{IC}_0, i^*\mathrm{IC}_0}"] & & j^*(\mathrm{IC}'_0 \star \mathrm{IC}'_0) \ar[d, "\theta_{\mathrm{IC}'_0, \mathrm{IC}'_0}"] \ar[rr, "j^*\eta"]& & j^*\mathrm{IC}_0'\ar[d, "\chi'"] \\
        j^*i^*\mathrm{IC}_0 \star j^*i^*\mathrm{IC}_0 \ar[rr,"(j^*\chi)\, \star \,(j^*\chi)"] \ar[d, "\operatorname{comp} \star \operatorname{comp}"] & &j^*\mathrm{IC}_0' \star j^*\mathrm{IC}_0' \ar[rr] \ar[d, "\chi' \star \chi'"] & & \underline{\CC}_{\mathrm{pt}} \ar[d, equals] \\
        k^*\mathrm{IC}_0 \star k^* \mathrm{IC}_0 \ar[rr, "\chi''\, \star \, \chi''"] & & \underline{\CC}_{\mathrm{pt}} \star \underline{\CC}_{\mathrm{pt}} \ar[rr, "1 \otimes 1 \mapsto 1"] & & \underline{\CC}_{\mathrm{pt}}
    \end{tikzcd}
    $$
    Here, the isomorphism $\chi': j^*\mathrm{IC}_0' \rightarrow \underline{\CC}_{\mathrm{pt}}$ (respectively, $\chi'': k^*\mathrm{IC}_0 \underline{\CC}_{\mathrm{pt}}$) is obtained from Construction~\ref{unit} after replacing the pair $L \subseteq G$ by the pair $\mathrm{pt} \subseteq L$ (respectively, $\mathrm{pt} \subseteq G$). Moreover, the unlabelled map $j^*\mathrm{IC}'_0 \star j^*\mathrm{IC}_0' \rightarrow \underline{\CC}_{\mathrm{pt}}$ is defined as the composition $j^*\mathrm{IC}_0' \star j^*\mathrm{IC}_0' \rightarrow \underline{\CC}_{\mathrm{pt}} \star \underline{\CC}_{\mathrm{pt}}  \simeq \underline{\CC}_{\mathrm{pt}}$ (and is included only for formatting purposes). 
    
    The commutativity of the top rectangle follows from the naturality of the isomorphism $\mathrm{comp}: j^*i^* \simeq k^*$ applied to the morphism $\eta$. The commutativity of the upper left small rectangle follows from the naturality of the morphism $\theta: j^*(- \star -) \rightarrow j^*(-) \star j^*(-)$ applied to $\chi \star \chi$. The commutativity of the upper right small rectangle will follow from the case $L = \mathrm{pt}$. The commutativity of the lower right small rectangle is trivial. The lower left small rectangle commutes by the compatiblity of $\chi$ with composition. The commutativity of the large outer rectangle will follow from the case $L = \mathrm{pt}$. These observations imply the commutativity of the second rectangle from the top (that is, the diagram $j^*$\eqref{unitality of restriction diagram0}). 

    We may therefore assume that $L = \mathrm{pt}$ is the trivial subgroup. In this case, the claim is obvious. 
\end{proof}

\begin{corollary}
    \begin{enumerate}
        \item The constructions of \ref{colax definition} and \ref{unit} equip the functor $i^*: D_{G(\mathcal{O})}(\Gr_G) \rightarrow D_{L(\mathcal{O})}(\Gr_L)$ with a colax monoidal structure. 
        \item The constructions of \ref{lax definition}, \ref{colax definition}, and \ref{unit} equip the functor $i^!: D_{G(\mathcal{O})}(\Gr_G) \rightarrow D_{L(\mathcal{O})}(\Gr_L)$ with a lax monoidal structure. 
    \end{enumerate}
\end{corollary}

\subsection{Fusion}\label{Fusion} In this section, we will give an alternative description of our lax monoidal structure on $i^!$. We recall some definitions and constructions from \cite[Section 2.1.2]{gaitsgory2000construction}, to which we refer the reader for more details. Let $\mathrm{Aut}$ denote the automorphism group scheme of the $\CC$-algebra $\mathcal{O}$. It is a pro-algebraic group and acts naturally on the affine Grassmannian $\Gr_G$ by loop rotation. Let $X = \mathbb{A}^1$ and let $\tilde{X} \rightarrow X$ denote the canonical $\mathrm{Aut}$-torsor over $X$ (whose fiber over $x \in X$ is the $\mathrm{Aut}$-torsor parameterizing uniformizers of the completed local ring $\widehat{\mathcal{O}}_{X,x}$). Recall that the \textit{Beilinson-Drinfeld Grassmannian} is the ind-scheme over $X$ given by the twisted product
$$
\Gr_{G,X} = \tilde{X} \times^{\mathrm{Aut}} \Gr_G \rightarrow X. 
$$
For any affine $X$-scheme $x: S \rightarrow X$, the set $\Gr_{G,X}(S)$ can be identified canonically with the set of isomorphism classes of pairs $(\mathcal{P}, \sigma)$, where $\mathcal{P}$ is a $G$-torsor over $X_S = X \times S$ and $\sigma$ is a trivialization of $\mathcal{P}$ over the subscheme $X_S \setminus \Gamma_x$, the complement of the graph $\Gamma_x: S \rightarrow X_S$ (Lemma 3 of \cite[Section 2.1.2]{gaitsgory2000construction}). The group scheme $\mathrm{Aut}$ acts through group automorphisms on the arc group $G(\mathcal{O})$. Therefore, we can form the group scheme $\mathcal{G}_X$ over $X$ given by 
$$
\mathcal{G}_X := \tilde{X} \times^{\mathrm{Aut}} G(\mathcal{O}) \rightarrow X.
$$
Since the action $G(\mathcal{O}) \times \Gr_G \rightarrow \Gr_G$ is $\mathrm{Aut}$-equivariant, we obtain an action of the group $X$-scheme $\mathcal{G}_X$ on the $X$-ind-scheme $\Gr_{G,X}$. The twisted external product defines a functor 
$$
p_2^\dagger:=\underline{\CC}_{\tilde{X}}[1] \tildeboxtimes -: D_{G(\mathcal{O}) \rtimes \mathrm{Aut}}(\Gr_G) \rightarrow D_{\mathcal{G}_X}(\Gr_{G,X}). 
$$
See \cite[Section 2.4.1]{acharriche2023} for a brief discussion of the definition of the above $\mathcal{G}_X$-equivariant derived category, and \cite[Chapter 10]{acharriche2023} for a thorough treatment. On the other hand, we have Gaitsgory's degeneration $\pi: \Gr_{G,X}' \rightarrow X$ from \cite[Section 3.1.1]{gaitsgory2000construction}. It is the ind-scheme over $X$ whose points over an affine $X$-scheme $x: S \rightarrow X$ are given by isomorphism classes of pairs $(\mathcal{P}, \sigma)$, where $\mathcal{P}$ is a $G$-torsor over $X_S$ and $\sigma$ is a trivialization of $\mathcal{P}$ over the subscheme $X_S \setminus (\Gamma_x \cup \Gamma_0)$, where $\Gamma_0$ denotes the graph of the map $S \rightarrow \operatorname{Spec}{\CC} \xrightarrow{0} \mathbb{A}^1 = X$. Note that $\mathcal{G}_X$ acts on $\Gr_{G,X}'
$ through modification of the trivialization $\sigma$. There are canonical isomorphisms 
\begin{align*}
    &\Gr_{G,0}' := \Gr_{G,X}' \vert_{\{0\}} \simeq \Gr_G \\
    &\Gr'_{G, X - 0} := \Gr_{G,X}' \vert_{X \setminus \{0\}} \simeq \Gr_G \times \Gr_{G,X} \vert_{X \setminus \{0\}} = \Gr_G \times \Gr_{G, X -0}. 
\end{align*}
Hence, we obtain a functor 
$$
- \boxtimes p_2^\dagger(-) \vert_{\Gr_{G,X-0}}: D_{G(\mathcal{O})}(\Gr_G) \times D_{G(\mathcal{O}) \rtimes \mathrm{Aut}}(\Gr_G) \rightarrow D_{\mathcal{G}_{X - 0}}(\Gr_{G, X - 0}'). 
$$
Here, we write $\mathcal{G}_{X - 0}$ for the restriction of $\mathcal{G}_X$ to the open subscheme $X \setminus \{0\}$. Finally, let $\Psi: D_{\mathcal{G}_{X - 0}}(\Gr_{G,X}') \rightarrow D_{G(\mathcal{O})}(\Gr_G)$ denote the functor of nearby cycles in the family $\pi$ (we are implicitly using the canonical identification of group schemes $\mathcal{G}_X \vert_{\{0\}} \simeq G(\mathcal{O})$ provided by the uniformizer $t \in \widehat{\mathcal{O}}_{X,0} = \CC[[t]]$ here). Following Gaitsgory, we define
\begin{equation}\label{Cdefinition}
C(A,B) := \Psi(A \boxtimes p_2^\dagger B \vert_{\Gr_{G,X-0}}).
\end{equation}
Because the forgetful functor $D_{G(\mathcal{O}) \rtimes \mathrm{Aut}}(\Gr_G) \rightarrow D_{G(\mathcal{O})}(\Gr_G)$ is an equivalence of categories (it is fully faithful because $\mathrm{Aut}$ is pro-unipotent, and essentially surjective because its image contains $\mathcal{P}_{G(\mathcal{O})}(\Gr_G)$ by \cite[Proposition 3.2.2]{nadler2004perverse} which generates $D_{G(\mathcal{O})}(\Gr_G)$ as a triangulated category), we may regard $C(-,-)$ as a bifunctor 
$$
C(-,-): D_{G(\mathcal{O})}(\Gr_G) \times D_{G(\mathcal{O})}(\Gr_G) \rightarrow D_{G(\mathcal{O})}(\Gr_G). 
$$
As usual, we will use the same notation for the corresponding objects defined using the subgroup $L$. We have closed immersions
$$
i_X: \Gr_{L,X} = \tilde{X} \times^{\mathrm{Aut}} \Gr_L \xrightarrow{\mathrm{id}_X \tildeboxtimes i} \tilde{X} \times^{\mathrm{Aut}} \Gr_G = \Gr_{G,X} 
$$
$$
i'_X: \Gr'_{L,X} \rightarrow \Gr_{G,X}'
$$
globalizing the $\mathrm{Aut}$-equivariant closed immersion $i: \Gr_L \hookrightarrow \Gr_G$. The morphism $i_X'$ takes a point $(\mathcal{P},\sigma)$ of $\Gr_{L,X}'$ over an affine $X$-scheme $x: S \rightarrow X$ to the point $(\mathcal{P}', \sigma') \in \Gr_{G,X}'(S)$ given by the induced $G$-torsor $\mathcal{P'} = \mathcal{P} \times^{L} G$ and the trivialization $\sigma'$ of $\mathcal{P}'$ over $X_S \setminus (\Gamma_x \cup \Gamma_0)$ induced by $\sigma$. Let $i_{X - 0}$, $i'_{X - 0}$ denote their restrictions to $X \setminus \{0\}$. 

\begin{construction}\label{Thetadefinition}
    Let $A,B \in D_{G(\mathcal{O})}(\Gr_G)$. We will define a natural map
    $$
    \Theta_{A,B}: i^*C(A, B) \rightarrow C(i^*A,i^*B) 
    $$
    in $D_{L(\mathcal{O})}(\Gr_G)$. We have a natural isomorphism (witnessing the monoidality of $*$-pullback with respect to the tensor product)
    \begin{equation}\label{Omegadefinition}
    \Omega_{A,B}: (i_{X - 0}')^*(A \boxtimes p_2^\dagger 
    B \vert_{\Gr_{G,X-0}}) \xrightarrow{\Omega_1} i^*A \boxtimes i^* p_2^\dagger B \vert_{\Gr_{G,X-0}} \xrightarrow{1 \, \boxtimes \, \Omega_2} i^*A \boxtimes p_2^\dagger (i^* B) \vert_{\Gr_{G,X-0}}. 
    \end{equation}
    We have also used the compositional isomorphism $i_2^* p_2^* \simeq p_2^* i^*$ to commute $i^*$ with $p_2^\dagger$. There is also a natural transformation \cite[Lemma 4.4.8]{achar2020}
    \begin{equation}\label{zetadefinition}
    \zeta: i^* \circ \Psi \rightarrow \Psi \circ (i_{X - 0}')^*. 
    \end{equation}
    Applying $\Psi$ to the isomorphism $\Omega_{A,B}$ and then precomposing with $\zeta$ defines the natural transformation $\Theta_{A,B}$. That is,
    $$
    \Theta_{A,B} = \Psi(\Omega_{A,B}) \circ \zeta_{A \, \boxtimes \, p_2^\dagger B \vert_{\Gr_{G,X-0}}}. 
    $$
\end{construction}

The following fundamental result of Gaitsgory relates the bifunctor $C(-,-)$ to the convolution product.
\begin{theorem}[{Gaitsgory \cite[Proposition 6]{gaitsgory2000construction}}]\label{convolution and fusion}
    Let $A,B \in D_{G(\mathcal{O})}(\Gr_G)$. There is a natural isomorphism in $D_{G(\mathcal{O})}(\Gr_G)$
    $$
    \varpi_{A,B}: C(A,B) \simeq A \star B. 
    $$
\end{theorem}

\begin{remark}
    Gaitsgory's result only requires $G(\mathcal{O})$-equivariance on one of the factors; the other need only belong to the constructible derived category $D(\Gr_G)$ (with the obvious caveat that $\varpi_{A,B}$ is in that case only an isomorphism of non-equivariant complexes). However, we will only need the version stated above. 
\end{remark}

\begin{construction}\label{varpidefinition}
    We need to unpack the construction of the isomorphism $\varpi_{A,B}$ from Gaitsgory's proof of \cite[Proposition 6]{gaitsgory2000construction}. Let $\widetilde{\Gr}'_{G,X} \rightarrow X$ denote the ind-scheme over $X$ whose set of points $\widetilde{\Gr}'_{G,X}(S)$ over an affine $X$-scheme $x: S \rightarrow X$ is the set of isomorphism classes of tuples $( \mathcal{P}_0, \mathcal{P}_1, \eta_0, \eta_1)$, where $\mathcal{P}_0$ and $\mathcal{P}_1$ are $G$-torsors over $X_S$, $\eta_0$ is an isomorphism of $G$-torsors $\eta_0: \mathcal{P}_0 \vert_{X_S \setminus \Gamma_x} \simeq \mathcal{P}_1 \vert_{X_S \setminus \Gamma_x}$, and $\eta_1$ is a trivialization of the $G$-torsor $\mathcal{P}_1$ over $X_S \setminus \Gamma_0$. We have a map $p: \widetilde{\Gr}_{G,X}' \rightarrow \Gr_{G,X}$ of $X$-ind-schemes taking an $S$-point $(\mathcal{P}_0, \mathcal{P}_1, \eta_0, \eta_1) \in \widetilde{\Gr}'_{G,X}(S)$ to the $S$-point $(\mathcal{P}_1, \eta_1) \in \Gr_{G,X}(S)$ and a map $\mu: \widetilde{\Gr}_{G,X}' \rightarrow \Gr_{G,X}'$ taking an $S$-point $(\mathcal{P}_0, \mathcal{P}_1, \eta_0, \eta_1) \in \widetilde{\Gr}'_{G,X}(S)$ to the $S$-point $(\mathcal{P}_0, \eta_0 \circ \eta_1 \vert_{X_S \setminus (\Gamma_x \cup \Gamma_0)}) \in {\Gr}'_{G,X}(S)$. 

    On the other hand, consider the $G(\mathcal{O})$-torsor $\widetilde{\Gr}_{G,X} \rightarrow \Gr_{G,X}$, defined by the fact that for any affine $X$-scheme $x: S \rightarrow X$, $\widetilde{\Gr}_{G,X}(S)$ is the set of isomorphism classes of tuples $(\mathcal{P},\eta, \gamma)$, where $(\mathcal{P}, \eta)$ is an $S$-point of $\Gr_{G,X}$ and $\gamma$ is a trivialization of $\mathcal{P}$ over the formal neighborhood $\widehat{\Gamma}_0$ of $\Gamma_0 \subseteq X_S$. Then, there is an isomorphism
    $$
    \widetilde{\Gr}_{G,X}' \simeq \Gr_G \times^{G(\mathcal{O})} \widetilde{\Gr}_{G,X} =: \Gr_G \tildetimes \Gr_{G,X}
    $$ 
    of ind-schemes over $X$, under which $p$ corresponds to the natural projection to the second factor. This description allows us to consider the twisted external product
    $$
    - \tildeboxtimes -: D_{G(\mathcal{O})}(\Gr_G) \times D_{\mathcal{G}_X}(\Gr_{G,X}) \rightarrow D_{\mathcal{G}_X}(\widetilde{\Gr}_{G,X}'). 
    $$
    Let $A, B \in D_{G(\mathcal{O})}(\Gr_G)$. Form the complex $A \tildeboxtimes p_2^\dagger B \in D_{\mathcal{G}_X}(\widetilde{\Gr}'_{G,X})$. Then, there is a natural isomorphism \cite[Lemma 5(b)]{gaitsgory2000construction}
    \begin{equation}\label{gaitsgorylemma5b}
    \mu_{*}(A \tildeboxtimes p_2^\dagger B) \vert_{\Gr_{G,X-0}} \simeq A \boxtimes (p_2^\dagger B \vert_{\Gr_{G,X-0}})
    \end{equation}
    induced by the projection formula. Let $\mu_0$ denote the fiber of $\mu$ over $0 \in X$. There is a natural transformation \cite[Lemma 4.8.8]{achar2020}
    \begin{equation}\label{nearbycyclesandpushforward}
    \Psi \mu_* \rightarrow \mu_{0*} \Psi,
    \end{equation}
    which is an isomorphism by the ind-properness of $\mu$ (note that the nearby cycles functor $\Psi$ appearing in the domain of this natural transformation is that corresponding to the family $\Gr_{G,X}' \rightarrow X$ whereas that in the domain is nearby cycles in the family $\widetilde{\Gr}_{G,X}' \rightarrow X$). Combining it with \eqref{gaitsgorylemma5b} yields an isomorphism
    $$
    \mu_{0*} \Psi (A \tildeboxtimes p_2^\dagger B) \simeq C(A,B). 
    $$
    Since $\widetilde{\Gr}_{G,X}' \rightarrow X$ can be trivialized by the choice of a global coordinate $t$ on $\mathbb{A}^1$, we also have a natural isomorphism
    \begin{equation}\label{nearbycyclesinatrivialfamily}
    \Psi(A \tildeboxtimes p_2^\dagger B) \simeq (A \tildeboxtimes p_2^\dagger B) \vert_{\Gr_G^{(2)}} \simeq  A \tildeboxtimes B,
    \end{equation}
    where $\Gr_G^{(2)}$ is identified with with the fiber of $\widetilde{\Gr}_{G,X}' \rightarrow X$ over $0 \in X$. Since $\mu_0 = m$ is the convolution map, we obtain the desired isomorphism 
    $$
    \varpi_{A,B}: A \star B = m_*(A \tildeboxtimes B) \simeq \mu_{0*} \Psi (A \tildeboxtimes p_2^\dagger B) \simeq C(A,B).
    $$
\end{construction}

\begin{construction}
    We can define a morphism 
    $$
    \tilde{i}'_X: \widetilde{\Gr}_{L,X}' \rightarrow  \widetilde{\Gr}'_{G,X} 
    $$
    by sending an $S$-point (for $x: S \rightarrow X$ an affine $X$-scheme) of $\widetilde{\Gr}_{L,X}'$ given by a tuple $(\mathcal{P}_0,\mathcal{P}_1,\eta_0,\eta_1)$ to the $S$-point of $\widetilde{\Gr}_{G,X}'$ given by $(\mathcal{P}'_0,\mathcal{P}'_1,\eta'_0,\eta'_1)$, where $\mathcal{P}'_i = \mathcal{P}_i \times^L G$ is the $G$-torsor associated to the $L$-torsor $\mathcal{P}_i$, $\eta_0'$ is the isomorphism $\mathcal{P}'_0 \vert_{X_S \setminus \Gamma_x} \simeq \mathcal{P}'_1 \vert_{X_S \setminus \Gamma_x}$ induced from $\eta_0'$, and $\eta_1'$ is the trivialization of $\mathcal{P}'_1$ away from $\Gamma_0$ induced from the trivialization $\eta_1$ of $\mathcal{P}_1$ away from $\Gamma_0$. Given $A,B \in D_{G(\mathcal{O})}(\Gr_G)$, we have a natural isomorphism
    \begin{equation}\label{Omegatildedefinition}
        \widetilde{\Omega}_{A,B}: (\tilde{i}'_X)^{*}(A \tildeboxtimes p_2^\dagger B) \simeq i^*A \tildeboxtimes p_2^\dagger i^*B.
    \end{equation}
    Note that the fiber of $\tilde{i}_X'$ over $0 \in X$ identifies with the map $i_2: \Gr_L^{(2)} \rightarrow \Gr_G^{(2)}$. Hence, the construction of \eqref{zetadefinition} defines a natural transformation 
    $$
    \zeta_2: i_2^*\Psi \rightarrow \Psi (\tilde{i}_X')^*.
    $$
\end{construction}

We can now formulate the global interpretation of our lax monoidal structure on $i^!$ (equivalently, of the colax monoidal structure on $i^*$).
\begin{proposition}\label{globalinterpretation}
    Let $A,B \in D_{G(\mathcal{O})}(\Gr_G)$. Then, the following diagram in $D_{L(\mathcal{O})}(\Gr_L)$ commutes:
    $$
    \begin{tikzcd}
        i^*C(A,B) \ar[d, "\varpi_{A,B}"]\ar[r, "\Theta_{A,B}"] & C(i^*A, i^*B) \ar[d, "\varpi_{A,B}"] \\
        i^*(A \star B) \ar[r, "\theta_{A,B}"] & i^*A \star i^*B
    \end{tikzcd}
    $$
\end{proposition}
\begin{proof}
    We expand the diagram using the definition \eqref{Cdefinition} of $C(A,B)$ as well as the definitions \ref{varpidefinition}, \ref{Thetadefinition}, and \ref{thetadefinition} of $\varpi_{A,B}$, $\Theta_{A,B}$, and $\theta_{A,B}$, respectively. We will suppress many subscripts on our natural transformations (indicating the objects that they are evaluated on) to preserve readability. 
    $$\hspace{-1.5cm}
    \begin{tikzcd}
        i^*\Psi(A \boxtimes p_2^\dagger B \vert_{\Gr_{G,X-0}}) \ar[r,"\zeta"] \ar[d,"i^*\Psi{\eqref{gaitsgorylemma5b}}"] & \Psi (i'_{X - 0})^*(A \boxtimes  p_2^\dagger B \vert_{\Gr_{G,X-0}}) \ar[d,"\Psi(i_{X - 0}')^*{\eqref{gaitsgorylemma5b}}"] \ar[rr,"\Psi (\Omega_{A,B})"] & & \Psi(i^*A \boxtimes p_2^\dagger i^*B) \ar[d,"\Psi {\eqref{gaitsgorylemma5b}}",swap] \\
        i^*\Psi\mu_*(A \tildeboxtimes p_2^\dagger B) \ar[d,"i^*{\eqref{nearbycyclesandpushforward}}"] \ar[r,"\zeta"] & \Psi(i'_{X - 0})^*\mu_*(A \tildeboxtimes p_2^\dagger B)\ar[r,"\Psi(\rho')"] & \Psi \mu_* (\tilde{i}_X')^*(A \tildeboxtimes p_2^\dagger B) \ar[r,"\Psi(\mu_*\widetilde{\Omega})"] \ar[d,"{\eqref{nearbycyclesandpushforward}}"] & \Psi\mu_*(i^*A \tildeboxtimes p_2^\dagger i^*B)  \ar[d,"{\eqref{nearbycyclesandpushforward}}",swap] \\
        i^*m_*\Psi(A \tildeboxtimes p_2^\dagger B) \ar[d,"i^*m_*{\eqref{nearbycyclesinatrivialfamily}}"] \ar[r,"\rho"] & m_*i_2^*\Psi(A \tildeboxtimes p_2^\dagger B) \ar[r,"m_*\zeta_2"] \ar[d,"m_*i_2^*{\eqref{nearbycyclesinatrivialfamily}}"] & m_*\Psi (\tilde{i}_X')^*(A \tildeboxtimes p_2^\dagger B) \ar[r,"m_*\Psi( \widetilde{\Omega})"] & m_*\Psi(i^*A \tildeboxtimes p_2^\dagger i^*B) \ar[d,"m_*{\eqref{nearbycyclesinatrivialfamily}}",swap]\\
        i^*m_*(A \tildeboxtimes B) \ar[r,"\rho"] & m_*i_2^*(A \tildeboxtimes B) \ar[rr,"m_*\omega_{A,B}"] & & m_*(i^*A \tildeboxtimes i^*B)
    \end{tikzcd}
    $$
    The upper left rectangle commutes by the naturality of the map $\zeta: i^*\Psi \rightarrow \Psi (i_{X-0}')^*$ of \eqref{zetadefinition} applied to the morphism \eqref{gaitsgorylemma5b}. The middle right rectangle commutes by the naturality of the map $\Psi \mu_* \rightarrow m_* \Psi$ of \eqref{nearbycyclesandpushforward} applied to the isomorphism $\Omega_{A,B}$ of \eqref{Omegadefinition}. The lower left rectangle commutes by the naturality of the base change map $\rho: i^*m_* \rightarrow m_*i_2^*$ of \eqref{rhodefinition} applied to the isomorphism $\Psi(A \tildeboxtimes p_2^\dagger B) \simeq A \tildeboxtimes B$ of \eqref{gaitsgorylemma5b}. The lower right rectangle is obtained by applying the functor $m_*$ to the diagram
    $$
    \begin{tikzcd}
        i_2^*\Psi(A \tildeboxtimes p_2^\dagger B) \ar[r,"\zeta_2"] \ar[d,"i_2^*{\eqref{nearbycyclesinatrivialfamily}}"] &  \Psi (\tilde{i}_X')^* (A \tildeboxtimes p_2^\dagger B) \ar[r,"\Psi( \widetilde{\Omega}_{A,B})"] \ar[d, "\eqref{nearbycyclesinatrivialfamily}"]& \Psi(i^*A \tildeboxtimes p_2^\dagger i^*B) \ar[d,"{\eqref{nearbycyclesinatrivialfamily}}"] \\
        i_2^*(A \tildeboxtimes B) \ar[rr, "\omega_{A,B}", bend right = 12,swap] \ar[r] & (\tilde{i}_X')^* (A \tildeboxtimes p_2^\dagger B) \vert_{\Gr_G^{(2)}} \ar[r,"\widetilde{\Omega}_{A,B} \vert_{\Gr^{(2)}_G}"] & i^*A \tildeboxtimes i^*B. 
    \end{tikzcd}
    $$
    The right rectangle commutes by the naturality of the isomorphism $\Psi \simeq (-)\vert_{\Gr_G^{(2)}}$ of \eqref{nearbycyclesinatrivialfamily} applied to the isomorphism 
    $\widetilde{\Omega}_{A,B}$ of \eqref{Omegatildedefinition}. Note that the isomorphism $\Psi \simeq (-) \vert_{\Gr_G^{(2)}}$ of \eqref{nearbycyclesinatrivialfamily} is defined by applying the construction of \eqref{zetadefinition} to the closed inclusion $\Gr_G^{(2)} \hookrightarrow \widetilde{\Gr}_{G,X}'$ of the zero fiber. Hence, the commutativity of the left rectangle follows from the straightforward compatiblity of the natural transformation $g^*\Psi \rightarrow \Psi g^*$ of \cite[Lemma 4.8.8]{achar2020} with the compositional isomorphisms for $*$-pullback. 
    
    We return to the original diagram and consider its upper right rectangle, which is obtained by applying $\Psi$ to the diagram
    $$
    \begin{tikzcd}
        (i'_{X - 0})^*(A \boxtimes p_2^\dagger B \vert_{\Gr_{G,X-0}}) \ar[rr, "\Omega_{A,B}"] \ar[d,"(i'_{X-0})^*{\eqref{gaitsgorylemma5b}}"] & & i^*A \boxtimes p_2^\dagger i^*B \ar[d,"\eqref{gaitsgorylemma5b}"] \\
        (i'_{X - 0})^*\mu_*(A \tildeboxtimes p_2^\dagger B) \ar[r, "\rho'"] & \mu_*(\tilde{i}_X')^*(A \tildeboxtimes p_2^\dagger B) \ar[r,"\mu_*\widetilde{\Omega}"] & \mu_*(i^*A \boxtimes p_2^\dagger i^*B).
    \end{tikzcd}
    $$
    The proof that this diagram commutes is sufficiently similar to the arguments in the proof of Proposition~\ref{associativity} that we omit it and leave it to the reader (use the adjunction $\mu^* \dashv \mu_*$ and then unwind the definitions of \eqref{gaitsgorylemma5b} and the base change map $\rho'$ to reduce the commutativity of this diagram to an assertion about the compatibility of $\Omega$ and $\widetilde{\Omega}$ with the compositional isomorphisms for $*$-pullback). 
    
    The commutativity of the middle left rectangle of the original diagram follows from the more general assertion that for any object $E \in D_{\mathcal{G}_X}(\widetilde{\Gr}_{G,X}')$, the following diagram commutes in $D_{L(\mathcal{O})}(\Gr_L)$:
    $$
    \begin{tikzcd}
        i^*\Psi \mu_*E \ar[r,"\zeta_{\mu_*E}"] \ar[d, "i^*\eqref{nearbycyclesandpushforward}"] & \Psi (i'_X)^*\mu_*E \ar[r, "\Psi(\rho')"] & \Psi\mu_* (\tilde{i}_X')^*E \ar[d, "\eqref{nearbycyclesandpushforward}"] \\
        i^*m_*\Psi(E) \ar[r,"\rho"] & m_*i_2^*\Psi(E) \ar[r, "m_* \zeta_2"] & m_*\Psi (\tilde{i}_X')^*E
    \end{tikzcd}
    $$
    The commutativity of this diagram is equivalent, under the adjunction $m^* \dashv m_*$, to the commutativity of the outer rectangle in the following diagram:
    $$
    \begin{tikzcd}
        m^*i^*\Psi \mu_*E \ar[rrr, "m^*\zeta_{\mu_*E}"] \ar[dddd,"m^*i^*\eqref{nearbycyclesandpushforward}"] &&& m^*\Psi (i_X')^*\mu_*E \ar[r,"m^*\Psi(\rho')"] \ar[d,"\gamma"] & m^*\Psi\mu_* (\tilde{i}_X')^*E \ar[d,"\gamma"] \\
        &  & & \Psi \mu^*(i_X')^*\mu_*E \ar[d,"\mathrm{comp}"] \ar[r, "\Psi\mu^*(\rho')"] & \Psi \mu^*\mu_* (\tilde{i}'_X)^*E \ar[d,"\Psi(\mathrm{counit})"] \\
        && & \Psi (\tilde{i}'_X)^* \mu^*\mu_* E \ar[r, "\Psi i_2^*(\mathrm{counit})"] & \Psi (\tilde{i}'_X)^*E \ar[dd, equals]\\
        & i_2^* m^*m_*\Psi(E) \ar[rrd, "i_2^*\mathrm{counit}"] & & i_2^*\Psi \mu^*\mu_* E \ar[d,"i_2^*\mathrm{counit}"] \ar[u, "\zeta_2"]\\
        m^*i^*m_*\Psi(E) \ar[r,"m^*\rho_{\Psi(E)}"] \ar[ur, "\mathrm{comp}"] & m^*m_*i_2^*\Psi(E) \ar[rr, "\mathrm{counit}"] & & i_2^*\Psi(E) \ar[r] & \Psi (\tilde{i}'_X)^*E
    \end{tikzcd}
    $$
    The upper right rectangle commutes by the naturality of the map $\gamma: m^*\Psi \rightarrow \Psi \mu^*$ of \cite[Lemma 4.8.8]{achar2020}. The rectangle below it commutes by the definition of the base change map $\rho': (i_X')^*\mu_* \rightarrow \mu_*(\tilde{i}_X')^*$. The rectangle below that commutes by the naturality of the map $\zeta_2: i_2^*\Psi \rightarrow \Psi (\tilde{i}_X')$ applied to the counit map $\mu^*\mu_* E \rightarrow E$. The commutativity of the bottom left triangle follows from the definition of the base change map $\rho: i^*m_* \rightarrow m_*i_2^*$. 
    Thus, we are reduced to showing that the following diagram commutes. 
    $$
    \begin{tikzcd}
        m^*i^*\Psi \mu_*E \ar[rrr,"m^*\zeta_{\mu_*E}"] \ar[rd,"\mathrm{comp}"] \ar[dddd,"m^*i^*\eqref{nearbycyclesandpushforward}"] &&& m^*\Psi (i_X')^*\mu_*E \ar[d,"\gamma"]\\
        & i_2^*m^*\Psi \mu_*E \ar[rrdd, "i_2^*\gamma_{\mu_*E}"] \ar[ddd,"i_2^*m^*\eqref{nearbycyclesandpushforward}"]\ar[rr,phantom,"(\dagger)"] & & \Psi \mu^*(i'_X)^*\mu_*E \ar[d,"\mathrm{comp}"] \\
        && & \Psi (\tilde{i}_X')^*\mu^*\mu_* E \\
        & & & i_2^*\Psi \mu^*\mu_* E \ar[d,"i_2^*(\mathrm{counit})"] \ar[u,"\zeta_2"] \\
        m^*i^*m_*\Psi(E) \ar[r,"\mathrm{comp}"] & i_2^* m^*m_*\Psi(E) \ar[rr,"i_2^*(\mathrm{counit})"] & & i_2^*\Psi(E) 
    \end{tikzcd}
    $$
    The left quadrilateral commutes by the naturality of the compositional isomorphism $\mathrm{comp}: m^*i^* \simeq i_2^*m^*$ applied to the map $\Psi\mu_* \rightarrow m_*\Psi$ of \eqref{nearbycyclesandpushforward}. The lower right quadrilateral is obtained by applying $i_2^*$ to the diagram
    $$
    \begin{tikzcd}
    m^*\Psi \mu_*E \ar[d,"m^*\eqref{nearbycyclesandpushforward}"] \ar[r,"\gamma_{\mu_*E}"] & \Psi \mu^*\mu_*E \ar[d, "\Psi(\mathrm{counit})"] \\
    m^*m_* \Psi(E) \ar[r,"\mathrm{counit}"] & \Psi(E).
    \end{tikzcd}
    $$
    Both composite maps $m^*\Psi\mu_*E \rightarrow \Psi(E)$ are left adjunct to the morphism $\Psi\mu_*E \rightarrow m_*\Psi(E)$ of \eqref{nearbycyclesandpushforward}, hence are equal. Finally, the upper figure $(\dagger)$ commutes by the more general assertion that for any $F \in D_{\mathcal{G}_X}(\Gr_{G,X}')$, the following diagram commutes:
    $$
    \begin{tikzcd}
        m^*i^*\Psi(F) \ar[r,"m^*\zeta_2"] \ar[d,"\mathrm{comp}"] & m^*\Psi (i_X')^*(F) \ar[r,"\gamma_{(i'_X)^*F}"] & \Psi \mu^*(i_X')^*(F) \ar[d,"\Psi(\mathrm{comp})"] \\
        i_2^* m^* \Psi(F) \ar[r,"i_2^*\gamma_F"] & i_2^*\Psi \mu^*(F) \ar[r,"\zeta_2"] & \Psi (\tilde{i}_X')^*\mu^*(F)
    \end{tikzcd}
    $$
    This diagram expresses a straightforward compatibility between the natural transformation of \cite[Lemma 4.8.8]{achar2020} with the compositional isomorphisms for $*$-pullback, whose proof we leave to the reader (since it is easy but requires unravelling the definition of the nearby cycles functor $\Psi$, which we will avoid doing here). 
\end{proof}

\subsection{Parabolic restriction}\label{parabolic restriction} We now return to the situation in which $L \subseteq G$ is a Levi subgroup of $G$. In \cite[Section 5.3.28]{beilinsondrinfeld}, Beilinson and Drinfeld define a natural operation taking complexes on $\Gr_G$ to complexes on $\Gr_L$ which we will refer to as \textit{parabolic restriction} (other possible names are the \textit{constant term functor} or the \textit{Jacquet functor}). We also refer the reader to \cite[Section 15]{baumann2018notes} for a complete treatment. It is a functor 
$$
\mathrm{Res}^G_L: D(\Gr_G) \rightarrow D(\Gr_L)
$$
defined through the following procedure. We identify $L$ with the quotient $P/V$ of $P$ by its unipotent radical. We obtain a correspondence of ind-schemes (the induction diagram):
$$
\begin{tikzcd}
    & \Gr_P \ar[rd, "r"] \ar[ld, "q", swap] \\
    \Gr_L & & \Gr_G
\end{tikzcd}
$$
The functor of (unnormalized) parabolic restriction is defined as the hyperbolic restriction 
$$
\mathrm{Res}^G_L = q_* \circ r^!: D(\Gr_G) \rightarrow D(\Gr_L). 
$$
We will also need an equivariant version of this functor. We define it as the composition 
\begin{equation}\label{equivariant parabolic restriction}
\mathrm{Res}^G_L = q_* \circ r^! \circ \mathrm{For}^{G(\mathcal{O})}_{L(\mathcal{O})}: D_{G(\mathcal{O})}(\Gr_G) \rightarrow D_{L(\mathcal{O})}(\Gr_L),
\end{equation}
where by $r^!$ we mean the $L(\mathcal{O})$-equivariant $!$-pullback functor and by $q_*$ we mean the $L(\mathcal{O})$-equivariant $*$-pushforward functor.  

\begin{remark} 
In keeping with the notational convention of not distinguishing between the equivariant and non-equivariant versions of a sheaf functor, we continue to use the notation $\mathrm{Res}^G_L$ for the equivariant (unnormalized) parabolic restriction.
\end{remark}

\begin{remark}\label{normalized parabolic restriction}
    Let $\Gr^\chi_L \subseteq \Gr_L$ denote a connected component of $\Gr_L$ indexed by a character $\chi \in \Lambda$ of $Z(\check{L})$. According to \cite[Proposition 5.3.29(1)]{beilinsondrinfeld} (see also \cite[Lemma 15.1]{baumann2018notes}), if $F \in\mathcal{P}_{G(\mathcal{O})}(\Gr_G)$ is perverse, then the complex $\mathrm{Res}^G_L(F) \vert_{\Gr_L^{\chi}}$ is concentrated in perverse degree $\langle 2\rho_G - 2\rho_L, \chi \rangle$. It is therefore useful to introduce the \textit{normalized} parabolic restriction functor $\mathrm{Res}^{G,\natural}_L$ defined by
    $$
    \mathrm{Res}^{G,\natural}_L := \bigoplus_{\chi \in \pi_0(\Gr_L)} \mathrm{Res}^{G}_L \vert_{\Gr_L^\chi}[-\langle 2\rho_G - 2\rho_L, \chi \rangle]. 
    $$
    The functor $\mathrm{Res}^{G,\natural}_L$ is now $t$-exact. Of course, these considerations apply verbatim to the equivariant version \eqref{equivariant parabolic restriction} of $\mathrm{Res}_L^G$, and we let $\mathrm{Res}^{G,\natural}_L$ also denote the normalized equivariant parabolic restriction functor.
\end{remark}

\begin{remark}
    Throughout this work, we have used the notation $i^!$ to refer to the functor 
    $$
    i^!: D_{G(\mathcal{O})}(\Gr_G) \rightarrow D_{L(\mathcal{O})}(\Gr_G)
    $$
    given by composing the forgetful functor $\mathrm{For}^{G(\mathcal{O})}_{L(\mathcal{O})}: D_{G(\mathcal{O})}(\Gr_G) \rightarrow D_{L(\mathcal{O})}(\Gr_G)$ with the $L(\mathcal{O})$-equivariant $!$-pullback $i^!: D_{L(\mathcal{O})}(\Gr_G) \rightarrow D_{L(\mathcal{O})}(\Gr_L)$. However, at some points in this section, it will be clearer to use the notation $i^!$ only for the latter functor $i^!: D_{L(\mathcal{O})}(\Gr_G) \rightarrow D_{L(\mathcal{O})}(\Gr_G)$ of $L(\mathcal{O})$-equivariant $!$-pullback and to make the forgetful functor $\mathrm{For}_{L(\mathcal{O})}^{G(\mathcal{O})}$ explicit. However, context should make clear which sense we assign to $i^!$. 
\end{remark}

\begin{construction}\label{Xiconstruction}
    We can compare $\mathrm{Res}^G_L$ to $i^! \circ \mathrm{For}^{G(\mathcal{O})}_{L(\mathcal{O})}$ as follows. Let $j: \Gr_L \hookrightarrow \Gr_P$ denote the closed inclusion. We have $rj = i$. Hence, we have a natural compositional isomorphism 
    $$
    \mathrm{comp}: i^! \circ \mathrm{For}^{G(\mathcal{O})}_{L(\mathcal{O})} \simeq j^! \circ r^! \circ \mathrm{For}^{G(\mathcal{O})}_{L(\mathcal{O})}. 
    $$
    Since $j$ is a closed immersion, we have a morphism $\mathrm{counit}: j_*j^! \rightarrow \mathrm{id}$. Since $qj = \mathrm{id}_{\Gr_L}$, we have an equality $(qj)_* = \mathrm{id}$. Thus, we obtain a natural transformation 
    $$
    \Delta: j^! = (qj)_*j^! \xrightarrow{\mathrm{comp}(j^!)} q_*j_*j^! \xrightarrow{q_*(\mathrm{counit})} q_*,
    $$
    where $\mathrm{comp}: (qj)_* \simeq q_*j_*$ is the compositional isomorphism. Therefore, we have constructed a natural transformation
    \begin{equation}\label{Xidefinition}
    \Xi: i^!\circ \mathrm{For}^{G(\mathcal{O})}_{L(\mathcal{O})} \xrightarrow{\mathrm{comp}} j^! \circ r^! \circ \mathrm{For}^{G(\mathcal{O})}_{L(\mathcal{O})} \xrightarrow{\Delta\left(r^!\circ \mathrm{For}^{G(\mathcal{O})}_{L(\mathcal{O})}\right)}  q_* \circ r^! \circ \mathrm{For}^{G(\mathcal{O})}_{L(\mathcal{O})} = \mathrm{Res}^G_L.
    \end{equation}
\end{construction}

\begin{construction}\label{Xiveeconstruction}
    The natural transformation $\Xi: i^! \circ \mathrm{For}^{G(\mathcal{O})}_{L(\mathcal{O})} \rightarrow \mathrm{Res}^G_L$ is the primary means by which we will compare the direct $!$-restriction with the parabolic restriction. However, it will be convenient (for the proof of Proposition~\ref{parabolic restriction vs corestriction}) to formulate a ``dual'' transformation 
    $$
    \Xi^\vee: q_! \circ r^* \circ \mathrm{For}^{G(\mathcal{O})}_{L(\mathcal{O})} 
    \rightarrow i^* \circ \mathrm{For}^{G(\mathcal{O})}_{L(\mathcal{O})}. 
    $$
    We now mimic Construction~\ref{Xiconstruction}. Namely, we start from the compositional isomorphism
    $$
    \mathrm{comp}: j^* \circ r^* \circ \mathrm{For}^{G(\mathcal{O})}_{L(\mathcal{O})} \simeq i^* \circ \mathrm{For}^{G(\mathcal{O})}_{L(\mathcal{O})}. 
    $$
    Since $j$ is a closed immersion, we have that $j_! = j_*$ and thus have a unit transformation $\mathrm{unit}: \mathrm{id} \rightarrow j_!j^*$. We also have an equality $\mathrm{id} = (qj)_!$. Therefore, we have a natural transformation 
    \begin{equation}\label{Deltaveedefinition}
    \Delta^\vee: q_! \xrightarrow{q_!\mathrm{unit}} q_!j_!j^* \xrightarrow{\mathrm{comp}(j^*)} (qj)_!j^* = j^*,
    \end{equation}
    where $\mathrm{comp}: q_! j_! \simeq (qj)_!$ is the compositional isomorphism. Finally, we obtain the natural transformation
    \begin{equation}\label{Xiveedefinition}
        \Xi^\vee: q_! \circ r^* \circ \mathrm{For}^{G(\mathcal{O})}_{L(\mathcal{O})} \xrightarrow{\Delta^\vee\left(r^* \circ \mathrm{For}^{G(\mathcal{O})}_{L(\mathcal{O})} \right)} j^* \circ r^* \circ \mathrm{For}^{G(\mathcal{O})}_{L(\mathcal{O})} \xrightarrow{\mathrm{comp}} i^* \circ \mathrm{For}^{G(\mathcal{O})}_{L(\mathcal{O})}. 
    \end{equation}
    The following lemma is an immediate consequence of our definition of $\Xi^\vee$ (note the contravariance of $\mathbb{D}$).
\end{construction}

\begin{lemma}\label{Xi vs Xivee}
    Let 
    $$
    \Xi: i^!\circ \mathrm{For}^{G(\mathcal{O})}_{L(\mathcal{O})} \rightarrow q_* \circ r^! \circ \mathrm{For}^{G(\mathcal{O})}_{L(\mathcal{O})}
    $$
    and
    $$
    \Xi^\vee: q_! \circ r^* \circ \mathrm{For}^{G(\mathcal{O})}_{L(\mathcal{O})} 
    \rightarrow i^* \circ \mathrm{For}^{G(\mathcal{O})}_{L(\mathcal{O})}. 
    $$
    be defined as in Construction~\ref{Xiconstruction} and Construction~\ref{Xiveeconstruction}, respectively. Then, we have commutative diagram of natural transformations
    $$
    \begin{tikzcd}
        i^! \circ \mathrm{For}^{G(\mathcal{O})}_{L(\mathcal{O})} \circ \mathbb{D} \ar[r, "\Xi \, \circ \, \mathbb{D}"] \ar[d] & q_* \circ r^! \circ \mathrm{For}^{G(\mathcal{O})}_{L(\mathcal{O})} \circ \mathbb{D} \ar[d] \\
        \mathbb{D} \circ i^* \circ \mathrm{For}^{G(\mathcal{O})}_{L(\mathcal{O})} \ar[r, "\mathbb{D}(\Xi^\vee)"] & \mathbb{D} \circ q_! \circ r^* \circ \mathrm{For}^{G(\mathcal{O})}_{L(\mathcal{O})}. 
    \end{tikzcd}
    $$
    Here, the vertical arrows are compositions of the standard isomorphisms $i^! \circ \mathbb{D} \simeq \mathbb{D} \circ i^*$, $q_* \circ \mathbb{D} \simeq \mathbb{D} \circ q_!$, $r^! \circ \mathbb{D} \simeq \mathbb{D} \circ r^*$ (see \cite[Corollary 3.9.10]{achar2020}), and the isomorphism $\mathrm{For}^{G(\mathcal{O})}_{L(\mathcal{O})} \circ \mathbb{D} \simeq \mathbb{D} \circ \mathrm{For}^{G(\mathcal{O})}_{L(\mathcal{O})}$ witnessing the compatiblity of the forgetful functor with Verdier duality (see \cite[Theorem 3.5.2(3)]{BL94}). 
\end{lemma}

\begin{construction}\label{parabolic restriction monoidal structure}
Beilinson and Drinfeld equip (the non-equivariant version of) $\mathrm{Res}^G_L$ with a monoidal structure in \cite[Section 5.3.30]{beilinsondrinfeld} through the use of the global fusion product. See \cite[Proposition 15.2]{baumann2018notes} for a detailed construction. We will give a slightly different (but conceptually identical) presentation of this monoidal structure (on the equivariant parabolic restriction functor $\mathrm{Res}^G_L$), using Gaitsgory's description Theorem~\ref{convolution and fusion} of the convolution product on $D_{G(\mathcal{O})}(\Gr_G)$ through the nearby cycles functor instead of the fusion product of \cite[Section 5]{MV04}. First of all, we have an analog of the induction diagram for the twofold convolution Grassmannians
$$
\begin{tikzcd}
    & \Gr_P^{(2)} \ar[rd, "r_2"] \ar[ld, "q_2", swap] \\
    \Gr_L^{(2)} & & \Gr_G^{(2)}.
\end{tikzcd}
$$
Here, $q_2$ and $r_2$ are constructed from $q$ and $r$, respectively, just as $i_2: \Gr_L^{(2)} \rightarrow \Gr_G^{(2)}$ was constructed from the map $i: \Gr_L \rightarrow \Gr_G$. That is, we have maps $P(\mathcal{K}) \hookrightarrow G(\mathcal{K})$ (resp. $P(\mathcal{K}) \twoheadrightarrow L(\mathcal{K})$) and $P(\mathcal{O}) \hookrightarrow G(\mathcal{O})$ (resp. $P(\mathcal{O}) \twoheadrightarrow L(\mathcal{O})$) which induce the morphism $r: \Gr_P \rightarrow \Gr_G$ (resp. $q: \Gr_P \rightarrow \Gr_L$) on the quotient ind-schemes. We therefore obtain a map $P(\mathcal{K}) \times \Gr_P \rightarrow G(\mathcal{K}) \times \Gr_G$ (resp. $P(\mathcal{K}) \times \Gr_P \rightarrow L(\mathcal{K}) \times \Gr_L$) which intertwines the diagonal action of $P(\mathcal{O})$ on $P(\mathcal{K}) \times \Gr_P$ with the diagonal action of $G(\mathcal{O})$ (resp. $L(\mathcal{O})$) on $G(\mathcal{K}) \times \Gr_G$ (resp. $L(\mathcal{K}) \times \Gr_L$). Therefore, the map $P(\mathcal{K}) \times \Gr_P \rightarrow G(\mathcal{K}) \times \Gr_G$ (resp. $P(\mathcal{K}) \times \Gr_P \rightarrow L(\mathcal{K}) \times \Gr_L$) descends to a morphism of ind-schemes 
$$
r_2: \Gr_P^{(2)} = \Gr_P \tildetimes \Gr_P \rightarrow \Gr_G \tildetimes \Gr_G = \Gr^{(2)}_G
$$
(resp. 
$$
q_2: \Gr_P^{(2)} = \Gr_P \tildetimes \Gr_P \rightarrow \Gr_L \tildetimes \Gr_L = \Gr^{(2)}_L ).
$$
The ind-schemes and morphisms of §\ref{Fusion} have direct analogs for the algebraic group $P$ (of course, $P$ is not reductive, but that is not essential for the properties of $\Gr_P$ that we are discussing here). Thus, we introduce the $X$-ind-schemes $\Gr_{P,X}$, $\Gr'_{P,X}$, $\widetilde{\Gr}_{P,X}$ and $\widetilde{\Gr}'_{P,X}$ and the $X$-group scheme $\mathcal{P}_X$. For each of the these ind-schemes, we have an analog of the induction diagram in the category of ind-schemes over $X$. In particular, we can define morphisms $r_{X}: \Gr_{P,X} \rightarrow \Gr_{G,X}$ and $q_{X}: \Gr_{P,X} \rightarrow \Gr_{L,X}$, as well as $r_{X}': \Gr_{P,X}' \rightarrow \Gr_{G,X}'$, $\tilde{r}_X: \widetilde{\Gr}_{P,X} \rightarrow \widetilde{\Gr}_{G,X}$, $\tilde{r}_X': \widetilde{\Gr}'_{P,X} \rightarrow \widetilde{\Gr}'_{G,X}$ and $q_{X}': \Gr_{P,X}' \rightarrow \Gr_{L,X}'$, $\tilde{q}_X: \widetilde{\Gr}_{P,X} \rightarrow \widetilde{\Gr}_{L,X}$, $\tilde{q}_X': \widetilde{\Gr}'_{P,X} \rightarrow \widetilde{\Gr}'_{L,X}$. We leave it to the reader to spell out the definitions of these morphisms by giving functorial maps on $S$-points (for any affine $X$-scheme $x: S \rightarrow X$).

We can also define the appropriate analogs of the parabaolic restriction functor $\mathrm{Res}^G_L = q_*\circ r^! \circ \mathrm{For}^{G(\mathcal{O})}_{L(\mathcal{O})}$ for each version of the Grassmannian. For example, the functor 
$$
\mathrm{Res}^G_{L,X}: D_{\mathcal{G}_X}(\Gr_{G,X}) \rightarrow D_{\mathcal{L}_X}(\Gr_{L,X})
$$
is defined by 
$$
\mathrm{Res}^G_{L,X} = (q_X)_* \circ r_X^! \circ \mathrm{For}^{\mathcal{G}_X}_{\mathcal{L}_X}.
$$
Similarly, we have the functor ${^{'}\mathrm{Res}^G_{L,X}}:D_{\mathcal{G}_X}(\Gr_{G,X}') \rightarrow D_{\mathcal{L}_X}(\Gr_{L,X}')$. Following Construction~\ref{Xiconstruction}, we can define natural transformations
$$
\Xi_{X}: i_X^! \circ \mathrm{For}^{\mathcal{G}_X}_{\mathcal{L}_X} \rightarrow \mathrm{Res}^G_{L,X} 
$$
$$
{^{'}{\Xi}}_{X}: (i'_X)^! \circ \mathrm{For}^{\mathcal{G}_X}_{\mathcal{L}_X} \rightarrow {^{'}{\mathrm{Res}^G_{L,X}}}.  
$$
We can also follow Construction~\ref{Xiveeconstruction} to define the natural transformations $\Xi_{X}^\vee$ and ${^{'}{\Xi}}_{X}^\vee$. We leave it to the reader to spell out the appropriate analogs of Lemma~\ref{Xi vs Xivee}. 

We will continue the notational laziness initiated in Remark~\ref{associativityconstraint} by using the same names for analogous maps and functors defined for the groups $L$, $P$, and $G$ (for example, $m: \Gr^{(2)}_P \rightarrow \Gr_P$ denotes the local convolution morphism, while $\mu: \widetilde{\Gr}'_{P,X} \rightarrow \Gr'_{P,X}$ denotes Gaitsgory's global convolution map from Construction~\ref{varpidefinition}). Moreover, we replace the subscript $X$ by $X - 0$ on an $X$-ind-scheme or a morphism of $X$-ind-schemes to indicate the restriction of this object to the open subscheme $X \setminus \{0\} \subseteq X$. For example, we have the morphism $r_{X - 0} := r_X \times_X \mathrm{id}_{X \setminus \{0\}}: \Gr_{P, X - 0} := \Gr_{P,X} \times_X (X \setminus \{0\}) \rightarrow \Gr_{G,X} \times_{X} (X \setminus \{0\}) =: \Gr_{G, X - 0}$ on the restrictions of these Beilinson-Drinfeld Grassmannians to $X \setminus \{0\}$. Sometimes we opt to save space and use the subscript $X$ on a morphism when we really should use $X - 0$, but we will only do so if it is unlikely to cause confusion. 

Let $A,B \in D_{G(\mathcal{O})}(\Gr_G)$. We form the complex $A \boxtimes p_2^\dagger B \vert_{\Gr_{G,X - 0}} \in D_{\mathcal{G}_{X - 0}}(\Gr'_{G,X - 0})$. We have a natural isomorphism in $D_{\mathcal{L}_{X - 0}}(\Gr'_{L,X-0})$
\begin{align}
\Sigma: (q'_{X - 0})_*(r'_{X - 0})^!(A \boxtimes p_2^\dagger B \vert_{\Gr_{G,X-0}}) &\simeq q_*r^!A \boxtimes ({q}_{X - 0})_*{r}_{X - 0}^! \left(p_2^\dagger B \vert_{\Gr_{G,X - 0}}\right)\label{Sigmadefinition1}\\
&\simeq q_*r^!A \boxtimes p_2^\dagger(q_*r^!B) \vert_{\Gr_{G,X - 0}}\label{Sigmadefinition2}.
\end{align}
The first line is the composition of the isomorphisms of \cite[Proposition 3.10.1]{achar2020} (to pass the $*$-pushforward into the external product) and \cite[Proposition 3.10.6(3)]{achar2020} (to pass the $!$-pullback into the external product). The second line uses the evident analogs of these natural isomorphisms for the twisted external product so as to push the $!$-pullback and $*$-pushforward past the functor $p_2^\dagger = \underline{\CC}_X[1] \tildeboxtimes -$. Moreover, the second line uses the fact that $q_*$ and $r^!$ commute with restriction to the \textit{open} sub-ind-scheme $\Gr_{G,X -0} \subseteq \Gr_{G,X}$ (by smooth base change). 

Next, consider a complex $F \in D_{\mathcal{G}_{X - 0}}(\Gr'_{G,X-0})$. We can apply \cite[Lemma 4.8.8]{achar2020} to construct a Verdier dual pair (in the sense that the first map yields the second after applying $\mathbb{D}$ and replacing $F$ by $\mathbb{D}F$) of natural maps
\begin{align}
\Psi ((q_{X- 0}')_*(r'_{X - 0})^!F) &\rightarrow q_*r^!\Psi(F) \label{parabolic restriction monoidal structure eq1}\\
q_!r^*\Psi(F) &\rightarrow \Psi((q_{X- 0}')_!(r'_{X - 0})^*F). \label{parabolic restriction monoidal structure eq2}
\end{align}
On the other hand, we may identify $r: \Gr_P \rightarrow \Gr_G$ (resp. $i: \Gr_L \rightarrow \Gr_G$, $q: \Gr_P \rightarrow \Gr_L$)  with the attracting ind-scheme (resp. the fixed point ind-scheme, the canonical projection from the attracting ind-scheme to the fixed point ind-scheme) for the action of $\mathbb{G}_m$ on $\Gr_G$ through $2\rho_G - 2\rho_L$, see \cite[Theorem 1.2.6]{acharriche2023} as well as Proposition~\ref{stabilizercomputation} above. Therefore, Braden's hyperbolic localization theorem (see \cite{braden2003hyperbolic}, \cite{drinfeld2014theorem}, and \cite{richarz2018spaces}) provides canonical isomorphisms $q_!r^* \simeq q_*r^!$ and $(q_{X- 0}')_!(r'_{X - 0})^* \simeq (q_{X- 0}')_*(r'_{X - 0})^!$. Thus, the functors $q_!r^*$ and $q_*r^!$ are examples of hyperbolic restriction functors, in the sense of \cite{braden2003hyperbolic}. It follows from Richarz's work \cite{richarz2018spaces} (especially \cite[Theorem 3.3]{richarz2018spaces}) that the nearby cycles functor commutes with hyperbolic restriction (even in the \'etale setting, and even with $\mathbb{G}_m$-equivariance weakened to the property of being $\mathbb{G}_m$-monodromic). Hence, the natural transformations \eqref{parabolic restriction monoidal structure eq1} and \eqref{parabolic restriction monoidal structure eq2} are isomorphisms. 

Now we can put the ingredients together. Gaitsgory's Theorem~\ref{convolution and fusion} provides the isomorphism $\varpi_{A,B}: A \star B \simeq C(A,B)$, which yields an isomorphism $q_*r^!\varpi_{A,B}: q_*r^!(A \star B) \simeq q_*r^!C(A,B)$. We can invoke \eqref{parabolic restriction monoidal structure eq1} to obtain an isomorphism 
\begin{equation}\label{parabolic restriction monoidal structure step1}
q_*r^!C(A,B) = q_*r^!\Psi(A \boxtimes p_2^\dagger B \vert_{\Gr_{G,X - 0}}) \simeq \Psi ({q}'_{X - 0})_*({r}'_{X - 0})^!(A \boxtimes p_2^\dagger B \vert_{\Gr_{G, X - 0}}). 
\end{equation}
Then, we can apply $\Psi$ to \eqref{Sigmadefinition1}, \eqref{Sigmadefinition2} to obtain an isomorphism 
\begin{align}\label{parabolic restriction monoidal structure step2}
\Psi(\Sigma): \Psi ({q}'_{X - 0})_*({r}'_{X - 0})^!(A \boxtimes p_2^\dagger B \vert_{\Gr_{G, X - 0}}) &\simeq \Psi\left(q_*r^!A \boxtimes p_2^\dagger(q_*r^!B) \vert_{\Gr_{G,X - 0}} \right)\\
&\simeq \Psi\left( \mathrm{Res}^G_L(A) \boxtimes p_2^\dagger\left(\mathrm{Res}^G_L(B) \right) \right). \nonumber
\end{align}
Finally, we have the isomorphism
\begin{equation}\label{parabolic restriction monoidal structure step3}
\varpi_{\mathrm{Res}^G_L(A), \mathrm{Res}^G_L(B)}^{-1}: \Psi\left( \mathrm{Res}^G_L(A) \boxtimes p_2^\dagger\left(\mathrm{Res}^G_L(B) \right) \right) \simeq \mathrm{Res}^G_L(A) \star \mathrm{Res}^G_L(B).
\end{equation}
Composing the isomorphisms \eqref{parabolic restriction monoidal structure step1}, \eqref{parabolic restriction monoidal structure step2}, and \eqref{parabolic restriction monoidal structure step3} defines the isomorphism
\begin{equation}
\varepsilon_{A,B}: \mathrm{Res}^G_L(A \star B) \simeq \mathrm{Res}^G_L(A) \star \mathrm{Res}^G_L(B). 
\end{equation}
We leave the task of verifying the compatibility of $\varepsilon_{A,B}$ with the associativity and unitality constraints on $D_{G(\mathcal{O})}(\Gr_G)$ and $D_{L(\mathcal{O})}(\Gr_L)$ (described in Remark~\ref{convolution} and Remark~\ref{unit}) to the reader. 
\end{construction}

\begin{proposition}\label{parabolic restriction vs corestriction}
    Regard $i^!: D_{G(\mathcal{O})}(\Gr_G) \rightarrow D_{L(\mathcal{O})}(\Gr_L)$ as a lax monoidal functor through Construction~\ref{colax definition} and Construction~\ref{lax definition}. Regard the parabolic restriction functor of \eqref{equivariant parabolic restriction}
    $$
    \mathrm{Res}^G_L: D_{G(\mathcal{O})}(\Gr_G) \rightarrow D_{L(\mathcal{O})}(\Gr_L)
    $$ 
    as a lax monoidal functor through Construction~\ref{parabolic restriction monoidal structure}. Then, the natural transformation of \eqref{Xidefinition} 
    $$
    \Xi: i^! \rightarrow \mathrm{Res}^G_L
    $$
    of functors $D_{G(\mathcal{O})}(\Gr_G) \rightarrow D_{L(\mathcal{O})}(\Gr_L)$ is a natural transformation of lax monoidal functors.
\end{proposition}
\begin{proof}
Recall that the lax monoidal structure of Construction~\ref{lax definition} on $i^!$ is obtained from the colax monoidal structure of Construction~\ref{colax definition} by Verdier duality. Hence, by Lemma~\ref{Xi vs Xivee}, the proposition is equivalent to the dual assertion that the natural transformation $\Xi^\vee$ of \eqref{Xiveedefinition} is a morphism of colax monoidal functors. As in Construction~\ref{Xiveeconstruction}, we will use the description of $\mathrm{Res}^G_L$ in terms of left adjoints throughout the proof. Moreover, we will use the description of the colax monoidal structure on $i^*$ given in Proposition~\ref{globalinterpretation} through Gaitsgory's interpretation Theorem~\ref{convolution and fusion} of the convolution product.

Let $A,B \in D_{G(\mathcal{O})}(\Gr_G)$. It suffices to show that the following diagram in $D_{L(\mathcal{O})}(\Gr_L)$ commutes. 
$$
\begin{tikzcd}
     \mathrm{Res}^G_L \left( \Psi \left( A \boxtimes p_2^\dagger B \vert_{\Gr_{G,X-0}}\right)\right) \ar[rrrr, "\Xi^\vee"] \ar[dd,"\eqref{parabolic restriction monoidal structure eq2}"] & & & &i^*\Psi \left(A \boxtimes p_2^\dagger B \vert_{\Gr_{G,X - 0}} \right) \ar[dd,"\zeta"]\\
     && (\dagger) \\
     \Psi{^{'}{\mathrm{Res}^G_{L,X-0}}} \left(A \boxtimes p_2^\dagger B \vert_{\Gr_{G,X-0}} \right) \ar[rrrr,"\Psi\left({^{'}{\Xi^\vee_{X - 0}}} \right)"] \ar[d,"\Psi\left(\eqref{Sigmadefinition2} \, \circ \, \eqref{Sigmadefinition1} \right)"]&& &&  \Psi (i'_{X - 0})^*\left(A \boxtimes p_2^\dagger B \vert_{\Gr_{G,X - 0}}\right) \ar[d, "\Psi(\Omega)"]\\ 
    \Psi\left(\mathrm{Res}^G_L(A) \boxtimes p_2^\dagger \mathrm{Res}^G_L(B) \vert_{\Gr_{L,X - 0}} \right) \ar[rrrr, "\Psi\left(\Xi^\vee \,\boxtimes\, p_2^\dagger(\Xi^\vee) \vert_{\Gr_{L,X - 0}}\right)"] & & &  & \Psi\left(i^*A \boxtimes p_2^\dagger i^*B \vert_{\Gr_{G,X-0}} \right)
\end{tikzcd}
$$
The bottom rectangle is obtained by applying the nearby cycles functor $\Psi$ to the outer rectangle in the following diagram.
$$\hspace{-1.5cm}
\begin{tikzcd}
     {^{'}{\mathrm{Res}^G_{L,X-0}\left(A \boxtimes p_2^\dagger B \vert_{\Gr_{G,X-0}} \right)}} \ar[rr,"{^{'}{\Xi^\vee_{X - 0}}}"] \ar[dd,"\eqref{Sigmadefinition1}"] &   & (i'_{X - 0})^*\left(A \boxtimes p_2^\dagger B \vert_{\Gr_{G,X-0}} \ar[dd, "\Omega_1",swap] \right) \\
     & (\star)\\
     \mathrm{Res}^G_L(A) \boxtimes \mathrm{Res}^G_{L,X-0}\left( p_2^\dagger B \vert_{\Gr_{G,X-0}} \right)\ar[r,"\Xi^\vee \, \boxtimes \, 1"] \ar[d,"1 \, \boxtimes \, \eqref{Sigmadefinition2}"] & i^*A \boxtimes \mathrm{Res}^G_{L,X-0}\left(p_2^\dagger B \vert_{\Gr_{G,X-0}} \right) \ar[r,"1 \, \boxtimes \, \Xi_{X-0}^\vee"] \ar[d, "1 \, \boxtimes \, \eqref{Sigmadefinition2}"] &  i^* A \boxtimes i_{X - 0}^*\left(p_2^\dagger B \vert_{\Gr_{G,X-0}}\right) \ar[d,"1 \, \boxtimes \, \Omega_2", swap] \\
     \mathrm{Res}^G_L(A) \boxtimes p_2^\dagger \mathrm{Res}^G_L(B) \vert_{\Gr_{L,X-0}} \ar[r, "\Xi^\vee \boxtimes 1"]  & i^*A \boxtimes p_2^\dagger\mathrm{Res}^G_L(B)\vert_{\Gr_{L,X-0}} \ar[r,"\mathrm{1} \, \boxtimes \, p_2^\dagger(\Xi^\vee) \vert_{\Gr_{L,X-0}}"] & i^* A \boxtimes p_2^\dagger (i^* B) \vert_{\Gr_{L,X-0}} \\
\end{tikzcd}
$$
The lower left rectangle commutes by the naturality of the map $\Xi^\vee \,\boxtimes\, -$ applied to the isomorphism $p_2^\dagger \left( \mathrm{Res}^G_{L,X-0}(B)\right) \vert_{\Gr_{L,X - 0}} \simeq \mathrm{Res}^G_L\left(p_2^\dagger B \vert_{\Gr_{G,X - 0}} \right)$ of \eqref{Sigmadefinition2}. The lower right rectangle commutes because it is obtained by applying the functor $i^*A \boxtimes -$ to the following diagram.
$$
\begin{tikzcd}
    (q_{X - 0})_!r_{X - 0}^*\left(p_2^\dagger B \vert_{\Gr_{G,X - 0}} \right) \ar[rr,"\Xi^\vee_{X - 0}"] \ar[d,"\eqref{Sigmadefinition2}"] && i_{X - 0}^* \left(p_2^\dagger B \vert_{\Gr_{G,X - 0}}\right) \ar[d,"\eqref{Sigmadefinition2}"]\\
    (q_{X - 0})_!p_2^\dagger r^*B \vert_{\Gr_{P,X - 0}} \ar[d,"\eqref{Sigmadefinition2}"] && i_{X - 0}^* \left(p_2^\dagger B \vert_{\Gr_{G,X - 0}}\right) \ar[d,"\eqref{Sigmadefinition2}"]\\
    p_2^\dagger\left(q_!r^*B \right)\vert_{\Gr_{L,X - 0}} \ar[rr,"p_2^\dagger(\Xi^\vee) \vert_{\Gr_{L,X-0}}"] && p_2^\dagger (i^* B)  \vert_{\Gr_{L, X - 0}}
\end{tikzcd}
$$
We expand the diagram horizontally using the definitions of $\Xi^\vee$ \eqref{Xiveedefinition} and of $\Xi^\vee_{X-0}$. 
$$
\begin{tikzcd}
    (q_{X - 0})_!r_{X - 0}^*\left(p_2^\dagger B \vert_{\Gr_{G,X - 0}} \right) \ar[r,"\Delta^\vee_{X-0}"] \ar[d,"\eqref{Sigmadefinition2}"] & j_{X-0}^*r_{X-0}^*\left(p_2^\dagger B \vert_{\Gr_{G,X - 0}} \right) \ar[r,"\mathrm{comp}"] \ar[d,"\eqref{Sigmadefinition2}"] & i_{X - 0}^* \left(p_2^\dagger B \vert_{\Gr_{G,X - 0}}\right) \ar[d,"\eqref{Sigmadefinition2}"]\\
    (q_{X - 0})_!p_2^\dagger r^*B \vert_{\Gr_{P,X - 0}} \ar[r,"\Delta^\vee_{X-0}"] \ar[d,"\eqref{Sigmadefinition2}"] & j_{X-0}^*p_2^\dagger r^*B \vert_{\Gr_{P,X - 0}}  \ar[d,"\eqref{Sigmadefinition2}"] & i_{X - 0}^* \left(p_2^\dagger B \vert_{\Gr_{G,X - 0}}\right) \ar[d,"\eqref{Sigmadefinition2}"]\\
    p_2^\dagger\left(q_!r^*B \right)\vert_{\Gr_{L,X - 0}} \ar[r,"p_2^\dagger \Delta^\vee \vert_{\Gr_{L,X-0}}"] & p_2^\dagger j^*r^*B \vert_{\Gr_{L,X-0}} \ar[r,"p_2^\dagger(\mathrm{comp}) \vert_{\Gr_{L,X-0}}"]& p_2^\dagger (i^* B)  \vert_{\Gr_{L, X - 0}}
\end{tikzcd}
$$
The top left square commutes by the naturality of the map $\Delta_{X - 0}^\vee: (q_{X - 0})_! \rightarrow j_{X - 0}^*$ applied to the morphism $r_{X -0}^*p_2^\dagger B \vert_{\Gr_{G,X-0}} \rightarrow p_2^\dagger r^*B \vert_{\Gr_{P,X-0}}$ of \eqref{Sigmadefinition2}. The right rectangle commutes by the evident compatibility of \eqref{Sigmadefinition2} with the compositional isomorphisms for $*$-pullback. To show that the bottom left square commutes, we may replace $r^*B$ by an arbitrary complex $M \in D_{P(\mathcal{O})}(\Gr_P)$ and show that the outer rectangle in the following diagram commutes.
$$
\begin{tikzcd}
(q_{X - 0})_!p_2^\dagger M \vert_{\Gr_{P,X-0}} \ar[rr,"(q_{X-0})_! \mathrm{unit}"] \ar[dd,"\eqref{Sigmadefinition2}"]&& (q_{X-0})_{!} (j_{X - 0})_!j_{X - 0}^* p_2^\dagger M \vert_{\Gr_{P,X-0}}\ar[r,"\mathrm{comp}"] \ar[d,"\eqref{Sigmadefinition2}"] & j_{X -0}^* p_2^\dagger M \vert_{\Gr_{P,X-0}} \ar[d,"\eqref{Sigmadefinition2}"]\\
&& (q_{X-0})_! (j_{X-0})_{!}p_2^\dagger j^*M \vert_{\Gr_{P,X-0}} \ar[r,"\mathrm{comp}"] \ar[d,"\eqref{Sigmadefinition2}"] & p_2^\dagger j^*M\vert_{\Gr_{L,X-0}} \ar[d, equals]\\
p_2^\dagger q_!M \vert_{\Gr_{L,X-0}} \ar[rr,"\mathrm{p}_2^\dagger(q_!\mathrm{unit})\vert_{\Gr_{L,X-0}}"] && p_2^\dagger q_! j_!j^* M \vert_{\Gr_{L,X-0}} \ar[r,"p_2^\dagger(\mathrm{comp})\vert_{\Gr_{L,X-0}}"] & p_2^\dagger j^*M\vert_{\Gr_{L,X-0}}
\end{tikzcd}
$$
The top right rectangle commutes by the naturality of the compositional isomorphism $(q_{X-0})_! (j_{X-0})_! \simeq \mathrm{id}$ applied to the morphism $j_{X-0}^* p_2^{\dagger} M \vert_{\Gr_{P,X-0}} \rightarrow p_2^{\dagger} j^* M \vert_{\Gr_{L,X-0}}$. The bottom right rectangle commutes by the compatibility of \eqref{Sigmadefinition2} with the compositional isomorphisms for $!$-pushforward. The left rectangle can be expanded to the following diagram.
$$
\begin{tikzcd}
    (q_{X - 0})_!p_2^\dagger M \vert_{\Gr_{P,X-0}} \ar[rrr,"(q_{X-0})_! \mathrm{unit}"] \ar[dd,equals]&&& (q_{X-0})_{!} (j_{X - 0})_!j_{X - 0}^* p_2^\dagger M \vert_{\Gr_{P,X-0}} \ar[d,"\eqref{Sigmadefinition2}"]\\
    &&& (q_{X-0})_! (j_{X-0})_{!}p_2^\dagger j^*M \vert_{\Gr_{P,X-0}} \ar[d,"\eqref{Sigmadefinition2}"] \\
    (q_{X-0})_!p_2^\dagger M \vert_{\Gr_{P,X-0}} \ar[rrr,"(q_{X-0})_! p_2^\dagger (\mathrm{unit}) \vert_{\Gr_{P,X-0}}"] \ar[d,"\eqref{Sigmadefinition2}"] & & &(q_{X-0})_! p_2^\dagger j_!j^*M \vert_{\Gr_{P,X-0}} \ar[d,"\eqref{Sigmadefinition2}"]\\
    p_2^\dagger q_!M \vert_{\Gr_{L,X-0}} \ar[rrr,"\mathrm{p}_2^\dagger(q_!\mathrm{unit})\vert_{\Gr_{L,X-0}}"] & && p_2^\dagger q_! j_!j^* M \vert_{\Gr_{L,X-0}}
\end{tikzcd}
$$
The bottom rectangle commutes by the naturality of the morphism $(q_{X-0})_!p_2^\dagger M \vert_{\Gr_{P,X-0}} \rightarrow p_2^\dagger q_! M \vert_{\Gr_{L,X-0}}$ applied to the unit map $\mathrm{unit}: M \rightarrow j_!j^*M$. The top rectangle is obtained by applying the functor $(q_{X-0})_!$ to the diagram
$$
\begin{tikzcd}
    p_2^\dagger M \vert_{\Gr_{P,X-0}} \ar[r,"\mathrm{unit}"] \ar[dd,equals]& (j_{X - 0})_!j_{X - 0}^* p_2^\dagger M \vert_{\Gr_{P,X-0}} \ar[d,"\eqref{Sigmadefinition2}"]\\
    & (j_{X-0})_{!}p_2^\dagger j^*M \vert_{\Gr_{P,X-0}} \ar[d, "\eqref{Sigmadefinition2}"] \\
    p_2^\dagger M \vert_{\Gr_{P,X-0}} \ar[r,"\mathrm{p}_2^\dagger \mathrm{unit}"] &  p_2^\dagger j_!j^*M \vert_{\Gr_{P,X-0}}.
\end{tikzcd}
$$
This diagram commutes by the very definition of the isomorphism \eqref{Sigmadefinition2}. 

We now turn to showing the commutativity of the top rectangle $(\star)$. In doing so, we may replace $p_2^\dagger B \vert_{\Gr_{G,X - 0}}$ by an arbitrary complex $C \in D_{\mathcal{G}_{X - 0}}(\Gr_{G,X - 0})$ and show that the following diagram commutes in $D_{\mathcal{L}_{X-0}}(\Gr_{L,X}')$. 
$$
\begin{tikzcd}
    (q_{X - 0}')_!(r_{X - 0}')^*(A \boxtimes C)\ar[rrrr, "{^{'}{\Xi_{X - 0}^\vee}}"] \ar[d, "\eqref{Sigmadefinition1}"] & & & & (i_{X - 0}')^*(A \boxtimes C) \ar[d,"\Omega_1"] \\
    q_!r^*(A) \boxtimes (q_{X - 0})_! r_{X-0}^*(C) \ar[rr, "\Xi^\vee \boxtimes \, 1"] &&  i^*A \boxtimes (q_{X - 0})_!r_{X - 0}^*C \ar[rr, "1 \, \boxtimes \, \Xi_{X - 0}^\vee"] && i^*A \boxtimes i^*_{X - 0}C
\end{tikzcd}
$$
Checking the commutativity of this diagram is an easy (if tedious) exercise, which we omit. The interested reader should expand the diagram using the definitions of $\Xi^\vee$ and ${^{'}{\Xi}}^\vee_{X-0}$, and then deduce the commutativity of resulting diagram from an appropriate compatibility between the isomorphisms of \cite[Proposition 2.5.45(a)]{achar2020} (for the $*$-pullback) and \cite[Proposition 3.10.1]{achar2020} (for the $!$-pushforward) and the compositional isomorphisms for the $*$-pullback and $!$-pushforward, respectively.

We can finally return to the original diagram. The commutativity of the top rectangle $(\dagger)$ follows from the more general assertion that for $E \in D_{\mathcal{G}_{X}}(\Gr_{G,X}')$, the following diagram commutes in $D_{L(\mathcal{O})}(\Gr_L)$. 
$$
\begin{tikzcd}
    q_! r^* \Psi(E)\ar[rr,"\Delta^\vee_{r^*\Psi(E)}"] \ar[d] && j^* r^* \Psi(E) \ar[r,"\mathrm{comp}"] \ar[d]  & i^* \Psi(E) \ar[dd,"\zeta"]\\
    q_!\Psi (r'_X)^* E \ar[rr,"\Delta^\vee_{\Psi(r'_X)^*E}"] \ar[d] && j^* \Psi (r_X')^*E \ar[d]  \\
    \Psi (q_X')_!(r_X')^* E \ar[rr,"\Psi \Delta^\vee_{(r_X')^*E}"] && \Psi (j'_X)^*(r'_X)^*E \ar[r,"\Psi(\mathrm{comp})"]& \Psi (i_X')^* E
\end{tikzcd}
$$
All of the unlabelled vertical maps are constructed from \cite[Lemma 4.8.8]{achar2020} (like $\zeta$, which was defined in \eqref{zetadefinition}) The commutativity of the top left square follows from the naturality of the morphism $\Delta^\vee: q_! \xrightarrow{q_!\mathrm{unit}} q_!j_!j^* \simeq j^*$ of \eqref{Deltaveedefinition} applied to the map $r^*\Psi(E) \rightarrow \Psi (r'_X)^*E$ of \cite[Lemma 4.8.8]{achar2020}. The commutativity of the right rectangle follows from the compatibility of the natural transformation $i^*\Psi \rightarrow \Psi (i'_X)^*$ of \cite[Lemma 4.8.8]{achar2020} with the compositional isomorphisms for the $*$-pullback functors (as asserted in the final step of the proof of Proposition~\ref{globalinterpretation}). To show the commutativity of the bottom left square, we may replace $(r'_X)^*E$ by an arbitrary complex $F \in D_{\mathcal{P}_X}(\Gr_{P,X}')$ and show that the following diagram commutes.
$$
\begin{tikzcd}
    q_!\Psi(F) \ar[rr,"q_!\mathrm{unit}"] \ar[dd,equals] && q_!j_!j^*\Psi(F) \ar[r,"\mathrm{comp}"] \ar[d] & j^* \Psi(F) \ar[d] \\
    && q_!j_! \Psi(j_X')^*F \ar[r,"\mathrm{comp}"] \ar[d] & \Psi (j_X')^*F \ar[dd,equals]\\ 
    q_!\Psi(F) \ar[rr,"q_!\Psi(\mathrm{unit})"] \ar[d] && q_!\Psi (j'_X)_!(j_X')^*F \ar[d] \\
    \Psi(q_X')_!F\ar[rr,"\Psi (q'_X)_{!} \mathrm{unit}"] && \Psi (q'_X)_!(j'_X)_!(j'_X)^*F \ar[r,"\mathrm{comp}"] & \Psi (j'_X)^*F
\end{tikzcd}
$$
Again, the unlabelled vertical maps are all the appropriate maps from \cite[Lemma 4.8.8]{achar2020}. The lower left rectangle commutes by the naturality of the map $ q_!\Psi \rightarrow \Psi (q'_X)_!$ applied to the unit $F \rightarrow (j'_X)_!(j'_X)^*F$. The upper right rectangle commutes by the naturality of the isomorphism $\mathrm{comp}: q_!j_! \simeq \mathrm{id}$ applied to the map $j^* \Psi(F) \rightarrow \Psi(j_X')^*F$. The commutativity of the lower right rectangle follows from the compatibility of the natural transformation of \cite[Lemma 4.8.8]{achar2020} with the compositional isomorphisms for the $!$-pushforward functors. The upper left rectangle is obtained by applying the functor $q_!$ to the following diagram 
$$
\begin{tikzcd}
    \Psi(F) \ar[rr,"\mathrm{unit}"] \ar[dd,equals] & & j_!j^* \Psi(F) \ar[d] \\
    & & j_!\Psi(j'_X)^* F \ar[d] \\
    \Psi(F) \ar[rr, "\Psi(\mathrm{unit})"] \ar[rr]& & \Psi(j'_X)_!(j'_X)^* F 
\end{tikzcd}
$$
This compatibility between the natural transformations of \cite[Lemma 4.8.8]{achar2020} (for the $*$-pushforward and $*$-pullback) and the unit of the adjunction between $*$-pushforward and $*$-pullback (recall that $j$, $j_X$, and $j_X'$ are closed immersions) is an easy consequence of the definitions. 
\end{proof}

\begin{remark}\label{parabolic restriction transitivity}
    The parabolic restriction functors enjoy the following transitivity property. Let $Q \subseteq P$ denote an inclusion of standard parabolic subgroups of $G$, and let $M \subseteq L$ denote the induced inclusion of Levi subgroups. Note the $Q \cap L$ is a parabolic subgroup of $L$ with Levi subgroup $M$. In particular, we may consider the parabolic restriction functor 
    $$
    \mathrm{Res}^L_M: D_{L(\mathcal{O})}(\Gr_L) \rightarrow D_{M(\mathcal{O})}(\Gr_M).
    $$
    Then, we claim that there is a natural isomorphism of lax monoidal functors $D_{G(\mathcal{O})}(\Gr_G) \rightarrow D_{M(\mathcal{O})}(\Gr_M)$
    $$
    \tau_{M \subseteq L \subseteq G}: \mathrm{Res}^L_M \circ \mathrm{Res}^G_L \simeq \mathrm{Res}^G_M. 
    $$
    This compatibility is noted in \cite[Eq. 6.3.2]{GR13}. Observe that we have a commutative diagram of $M(\mathcal{O})$-equivariant ind-schemes
    $$
    \begin{tikzcd}
        & & \Gr_Q \ar[ld, "q'",swap] \ar[rd,"s'"] \ar[rrdd, bend left = 25, "r''"] \ar[lldd, bend right=25, "q''",swap]\\
        & \Gr_{Q \cap L} \ar[rd,"s"] \ar[ld,"t",swap] & & \Gr_P \ar[rd,"r"] \ar[ld,"q",swap]\\
        \Gr_M & & \Gr_L & & \Gr_G.
    \end{tikzcd}
    $$
    Note that the middle square is Cartesian. Therefore, we have the isomorphism
    \begin{align*}
    \tau_{M \subseteq L \subseteq G}: \mathrm{Res}^L_M \circ \mathrm{Res}^G_L &= t_*s^! \circ \mathrm{For}^{L(\mathcal{O})}_{M(\mathcal{O})} \circ q_*r^! \circ \mathrm{For}^{G(\mathcal{O})}_{L(\mathcal{O})}\\
    &\simeq  t_*s^!q_*r^!\circ \mathrm{For}^{L(\mathcal{O})}_{M(\mathcal{O})} \mathrm{For}^{G(\mathcal{O})}_{L(\mathcal{O})}\\
    &\simeq t_*s^!q_*r^! \circ \mathrm{For}^{G(\mathcal{O})}_{M(\mathcal{O})}\\
    &\simeq t_*(q')_*(s')^!r^! \circ \mathrm{For}^{G(\mathcal{O})}_{M(\mathcal{O})}\\
    &\simeq (q'')_* (r'')^! \mathrm{For}^{G(\mathcal{O})}_{M(\mathcal{O})}\\
    &= \mathrm{Res}^G_M.
    \end{align*}
    The first (non-equality) isomorphism is given by the compatiblity \cite[Theorem 3.4.1(i)]{BL94} of the forgetful functor $\mathrm{For}^{L(\mathcal{O})}_{M(\mathcal{O})}$ with $!$-pullback and $*$-pushforward. The second isomorphism is given by the transitivity of the forgetful functors. The third isomorphism is the proper base change theorem. The fourth isomorphism is induced by the compositional isomorphisms $t_*(q'_*) \simeq (q'')_*$ and $(s')^!r^! \simeq (r'')^!$. 
\end{remark}

\begin{remark}\label{lax monoidal structure on cohomology}
    We recall the definition of the standard lax monoidal structure on the equivariant cohomology functor 
    $$
    R\Gamma_{G(\mathcal{O})}(\Gr_G, -): D_{G(\mathcal{O})}(\Gr_G) \rightarrow D_{G(\mathcal{O})}(\mathrm{pt}).
    $$
    Let $A,B \in D_{G(\mathcal{O})}(\Gr_G)$. Then, we have a natural map
    \begin{equation}\label{lax monoidal structure on cohomology eq1}
    R\Gamma_{G(\mathcal{O})}(\Gr_G, A) \otimes R\Gamma_{G(\mathcal{O})}(\Gr_G, B) \xrightarrow{- \tildeboxtimes -} R\Gamma_{G(\mathcal{O})}(\Gr_G \tildetimes \Gr_G, A \tildeboxtimes B)
    \end{equation}
    as well as an isomorphism
    \begin{equation}\label{lax monoidal structure on cohomology eq2}
    R\Gamma_{G(\mathcal{O})}(\Gr_G \tildetimes \Gr_G, A \tildeboxtimes B) \simeq R\Gamma_{G(\mathcal{O})}(\Gr_G, m_*(A \tildeboxtimes B)) = R\Gamma_{G(\mathcal{O})}(\Gr_G, A \star B). 
    \end{equation}
    The composition of the morphisms \eqref{lax monoidal structure on cohomology eq1} and \eqref{lax monoidal structure on cohomology eq2} equips $R\Gamma_{G(\mathcal{O})}(\Gr_G,-)$ with a lax monoidal structure. 

    In particular, the functor 
    $$
    H^*_{G(\mathcal{O})}(\Gr_G, -): D_{G(\mathcal{O})}(\Gr_G) \rightarrow \mathrm{mod}(R_G) 
    $$
    is lax monoidal. Here, for any graded $\CC$-algebra $A$, we follow \cite{bgs} and write $\mathrm{mod}(A)$ for the category of \textit{graded} (left) $A$-modules.  
\end{remark}

\subsection{Action of equivariant homology}\label{equivarianthomologysubsection} Let $H^*_{L(\mathcal{O})}(\Gr_G, \CC)$ denote the $L(\mathcal{O})$-equivariant cohomology ring of $\Gr_G$. We explain how the results of Yun and Zhu \cite{yun2009integral} on the $T$-equivariant cohomology $H_{T(\mathcal{O})}^*(\Gr_G, \CC)$ generalize to describe the $L$-equivariant cohomology $H^*_{L(\mathcal{O})}(\Gr_L, \CC)$ (we will have more to say in §\ref{comparison}).

\begin{construction}\label{hopfalgebrastructure}
    The graded vector space $H^*_{L(\mathcal{O})}(\Gr_G, \CC)$ is naturally a graded module over the graded ring $R_L := H^*_{L(\mathcal{O})}(\mathrm{pt},\CC)$. Via the homeomorphism $\Omega_{\mathrm{poly}}G_c \cong \Gr_G$ of $\Gr_G$ with the based polynomial loop group of a maximal compact subgroup $G_c \subseteq G$, the ring $H^*_{L(\mathcal{O})}(\Gr_G, \CC) \simeq H^*_{L_c}(\Omega_{\mathrm{poly}}G_c, \CC)$ acquires the structure of a commutative and cocommutative graded Hopf $R_L$-algebra (independently of the choice of $G_c$). We refer the reader to the beautiful and classical treatment \cite{milnormoore} of Milnor and Moore for a discussion of the Hopf algebra structure on the cohomology of a Lie group.  
    
    Following \cite[Eq. 2.10]{yun2009integral}, we define the $L(\mathcal{O})$-equivariant homology $H^{L(\mathcal{O})}_*(\Gr_G, \CC)$ of $\Gr_G$ to be the $R_L$-linear graded dual of $H^*_{L(\mathcal{O})}(\Gr_G, \CC)$:
    $$
    H^{L(\mathcal{O})}_*(\Gr_G, \CC) := \mathrm{Hom}_{R_L}(H^*_{L(\mathcal{O})}(\Gr_G,\CC), R_L)^{\mathrm{gr}}. 
    $$
    That is, the $n$th graded component $H^{L(\mathcal{O})}_n(\Gr_G, \CC)$ consists of all $R_L$-module homomorphisms $H^*_{L(\mathcal{O})}(\Gr_L, \CC) \rightarrow R_L$ which take $H^i_{L(\mathcal{O})}(\Gr_L, \CC)$ to $R^{i-n}_L$ (this sign convention keeps the homology in non-negative degrees). 
\end{construction}

\begin{remark}\label{equivariantformalityofgr}
    Note that $\Gr_G$ is an equivariantly formal $L(\mathcal{O})$-space. That is, the spectral sequence
    $$
    E_2^{p,q} = H^p_{L(\mathcal{O})}(\mathrm{pt}, H^q(\Gr_G,\CC)) \implies H^{p + q}_{L(\mathcal{O})}(\Gr_G, \CC)
    $$
    degenerates at $E_2$. Indeed, the claim follows from the fact that $H^*(\Gr_G, \CC)$ is concentrated in even degrees (since $\Gr_G$ admits a paving by affine spaces, namely the orbits of the Iwahori subgroup $I \subseteq G(\mathcal{O})$). It follows that $H^*_{L(\mathcal{O})}(\Gr_G,\CC)$ is a free $R_L$-module. As the graded dual of a free and finitely generated Hopf $R_L$-algebra, $H^{L(\mathcal{O})}_*(\Gr_G, \CC)$ inherits a natural structure of Hopf $R_L$-algebra. 
\end{remark}

\begin{construction}\label{comodulestructure}
    Let $A \in D_{L(\mathcal{O})}(\Gr_G)$. We wish to equip $H^*_{L(\mathcal{O})}(\Gr_G, A)$ with the structure of a comodule over $H^{L(\mathcal{O})}_*(\Gr_G, \CC)$. To do so, we fix a basis $\{h^i\}$ of the free $R_L$-module $H^*_{L(\mathcal{O})}(\Gr_G, \CC)$ (none of the constructions below actually depend on this basis). Let $\{h_i\}$ denote the dual basis of $H^{L(\mathcal{O})}_*(\Gr_G, \CC)$. We must specify a graded $R_L$-linear map 
    $$
    \sigma: H^*_{L(\mathcal{O})}(\Gr_G, A) \rightarrow H_{L(\mathcal{O})}^*(\Gr_G, A) \otimes_{R_L} H^{L(\mathcal{O})}_*(\Gr_G, \CC). 
    $$
    The right hand side is a \textit{module} over $H^{L(\mathcal{O})}_*(\Gr_G, \CC)$, so it is equivalent to the define the unique $H^{L(\mathcal{O})}_*(\Gr_G, \CC)$-linear extension of $\sigma$ instead:
    $$
    \tilde{\sigma}: H^*_{L(\mathcal{O})}(\Gr_G, A) \otimes_{R_L} H^{L(\mathcal{O})}_*(\Gr_G, \CC) \rightarrow H_{L(\mathcal{O})}^*(\Gr_G, A) \otimes_{R_L} H^{L(\mathcal{O})}_*(\Gr_G, \CC). 
    $$
    The map $\tilde{\sigma}$ is defined \cite[Lemma 3.1]{yun2009integral} to be the automorphism given explicitly on an element $v \otimes h$ ($v \in H^*_{L(\mathcal{O})}(\Gr_G, A)$, $h \in H_*^{L(\mathcal{O})}(\Gr_G, \CC)$) by the formula
    $$
    \tilde{\sigma}(v \otimes h) = \sum_i (h^i \cup v) \otimes (h_i \wedge h).
    $$
    In this formula, $h^i \cup -$ denotes the action of the cohomology class $h^i \in H_{L(\mathcal{O})}^*(\Gr_G, \CC)$ on the vector space $H_{L(\mathcal{O})}^*(\Gr_G, A)$ and $\wedge$ denotes the (Pontryagin) product on the Hopf algebra $H^{L(\mathcal{O})}_*(\Gr_G, \CC)$. Note that this sum is actually finite because $A$ is supported on $\Gr_{G}^{\leq \lambda}$ for $\lambda$ sufficiently large.     
\end{construction}

\begin{remark}\label{change of group}
    The forgetful map $H^*_{G(\mathcal{O})}(\Gr_G, \CC) \rightarrow H^*_{L(\mathcal{O})}(\Gr_L, \CC) $ is an $R_G$-algebra homomorphism, where $H^*_{L(\mathcal{O})}(\Gr_G, \CC)$ is viewed as an $R_G$-module via restriction of scalars along the forgetful homomorphism $R_G = H_{G(\mathcal{O})}^*(\mathrm{pt},\CC) \rightarrow H_{L(\mathcal{O})}^*(\mathrm{pt}, \CC) = R_L$. Hence, it extends uniquely to an $R_L$-algebra homomorphism
    $$
    \phi: H^*_{G(\mathcal{O})}(\Gr_G, \CC) \otimes_{R_G} R_L \rightarrow H^*_{L(\mathcal{O})}(\Gr_G, \CC).
    $$
    The morphism $\phi$ is in fact an isomorphism. Indeed, note that $R_L$ is a graded ring with augmentation module $R_L/\mathfrak{m}_L \simeq H^*(\mathrm{pt},\CC) \simeq \CC$ (where $\mathfrak{m}_L \subseteq R_L$ denotes the irrelevant ideal). Tensoring $\phi$ with the augmentation module $R_L/\mathfrak{m}_L$ yields the identity map on $H^*(\Gr_G,\CC)$ (by Remark~\ref{equivariantformalityofgr}). Since $\phi$ is a homomorphism of free graded $R_L$-modules concentrated in non-negative degrees, we deduce from the graded form of Nakayama's lemma that $\phi$ is an isomorphism. 

    Recall that the comultiplications on $H^*_{G(\mathcal{O})}(\Gr_G, \CC)$ and $H^*_{L(\mathcal{O})}(\Gr_G, \CC)$ are induced by pullback along the multiplication map $\Gr_G \times \Gr_G \simeq \Omega_{\mathrm{poly}} G_c \times \Omega_{\mathrm{poly}} G_c \rightarrow \Omega_{\mathrm{poly}} G_c \simeq \Gr_G$. The forgetful map $H^*_{G(\mathcal{O})}(\Gr_G, \CC) \rightarrow H^*_{L(\mathcal{O})}(\Gr_G, \CC)$ is compatible with the pullback of cohomology classes (an obvious consequence of the definitions), so we deduce that $\phi$ is an $R_L$-coalgebra homomorphism. Therefore, it an isomorphism of Hopf $R_L$-algebras (note that a bialgebra homomorphism automatically respects antipodes). 

    Applying graded duality over the base ring $R_L$ to $\phi$ yields an isomorphism of Hopf $R_L$-algebras 
    $$
    \phi^\vee_{L \subseteq G} := \phi^\vee: H_*^{L(\mathcal{O})}(\Gr_G, \CC) \simeq H_*^{G(\mathcal{O})}(\Gr_G, \CC) \otimes_{R_G} R_L. 
    $$
    If $M \subseteq L$ is a further Levi subgroup of $L$, then we have the following commutative diagram of Hopf $R_M$-algebras, which expresses the transitivity of this construction.
    \begin{equation}\label{homology transitivity}
    \begin{tikzcd}
        H_*^{L(\mathcal{O})}(\Gr_G, \CC) \otimes_{R_L} R_M \ar[rr, "\phi^\vee_{L \subseteq G} \otimes 1"] \ar[d, "\phi^\vee_{M \subseteq L}"]&& \left( H_*^{G(\mathcal{O})}(\Gr_G, \CC) \otimes_{R_G} R_L\right) \otimes_{R_L} R_M \ar[d, "\sim"]\\
        H_*^{M(\mathcal{O})}(\Gr_G, \CC) \ar[rr,"\phi^\vee_{M \subseteq G}"] && H_*^{G(\mathcal{O})}(\Gr_G,\CC) \otimes_{R_G} R_M.
    \end{tikzcd}
    \end{equation}
    Here, the right vertical map is the standard isomorphism. 
\end{remark}

\begin{construction}\label{homomorphism on regular centralizers}
    Consider now the closed subspace $i: \Gr_L \hookrightarrow \Gr_G$. We have an induced map in equivariant cohomology 
    $$
    i^*: H^*_{L(\mathcal{O})}(\Gr_G, \CC) \rightarrow H^*_{L(\mathcal{O})}(\Gr_L, \CC). 
    $$
    Note that there is a commutative diagram of $L_c$-spaces (where $L_c \subseteq L$ is a maximal compact contained in $G_c$):
    $$
    \begin{tikzcd}
        \Omega_{\mathrm{poly}} L_c \ar[r, "i", hookrightarrow] \ar[d, "\sim"] & \Omega_{\mathrm{poly}} G_c \ar[d, "\sim"]\\
        \Gr_L \ar[r, "i", hookrightarrow] & \Gr_G
    \end{tikzcd}.
    $$
    Since the inclusion $\Gr_L \simeq \Omega_{\mathrm{poly}}L_c \hookrightarrow \Omega_{\mathrm{poly}}G_c \simeq \Gr_G$ is a group homomorphism, it follows immediately from the definition of the comultiplications on $H^*_{L(\mathcal{O})}(\Gr_G, \CC)$ and $H^*_{L(\mathcal{O})}(\Gr_G, \CC)$ that the map 
    $$
    i^*: H^*_{L(\mathcal{O})}(\Gr_G, \CC) \rightarrow H^*_{L(\mathcal{O})}(\Gr_L, \CC)
    $$
    is a homomorphism of $\CC$-coalgebras. Of course, $i^*$ is a $\CC$-algebra homomorphism (the products on these algebras are given by the cup product in equivariant cohomology), so $i^*$ is a Hopf $\CC$-algebra homomorphism. Passing to $R_L$-linear graded duals, we deduce that the pushforward map
    $$
    i_* := (i^*)^{\vee}: H_*^{L(\mathcal{O})}(\Gr_L, \CC) \rightarrow H_*^{L(\mathcal{O})}(\Gr_G, \CC)
    $$
    is a homomorphism of graded Hopf $R_L$-algebras. We can now bring in Remark~\ref{change of group} and define a homomorphism of graded Hopf $R_L$-algebras
    \begin{equation}\label{pushforwarddefinition}
    i_*: H_*^{L(\mathcal{O})}(\Gr_L, \CC) \xrightarrow{i_*} H_*^{L(\mathcal{O})}(\Gr_G, \CC) \xrightarrow{\phi^\vee_{L \subseteq G}} H_*^{G(\mathcal{O})}(\Gr_G, \CC) \otimes_{R_G} R_L. 
    \end{equation}
    Suppose now that $M \subseteq L$ is a further Levi subgroup of $L$ and that $j: \Gr_M \hookrightarrow \Gr_L$ denotes the inclusion of Grassmannians. Let $k = i \circ j: \Gr_M \hookrightarrow \Gr_G$ denote the composition. Then, it follows from \eqref{homology transitivity} that the following diagram of Hopf $R_M$-algebras commutes. 
    \begin{equation}
    \begin{tikzcd}
        H_*^{M(\mathcal{O})}(\Gr_M, \CC) \ar[r, "k_*"] \ar[d, "j_*"] & H_*^{G(\mathcal{O})}(\Gr_G, \CC) \otimes_{R_G} R_M \ar[d, "\sim"]\\
        H_*^{L(\mathcal{O})}(\Gr_L,\CC) \otimes_{R_L} R_M \ar[r,"i_* \otimes 1"] &  \left(H_*^{G(\mathcal{O})}(\Gr_G, \CC) \otimes_{R_G} R_L\right) \otimes_{R_L} R_M
    \end{tikzcd}
    \end{equation}
    Again, the right vertical map is the standard isomorphism. 
\end{construction}

\begin{remark}\label{homomorphism on regular centralizers, group schemes}
    It is useful to reformulate these constructions in the language of group schemes. Since $H^{G(\mathcal{O})}_*(\Gr_G, \CC)$ is a commutative and cocommutative Hopf $R_G$-algebra, its spectrum 
    $$
    \mathfrak{A}_G := \operatorname{Spec}{H_*^{G(\mathcal{O})}(\Gr_G, \CC)} \rightarrow \operatorname{Spec}{H^*_{G(\mathcal{O})}}(\mathrm{pt},\CC) \simeq \mathfrak{c}_{G}
    $$ 
    is a commutative affine $\mathfrak{c}_G$-group scheme, where $\mathfrak{c}_G \simeq \mathfrak{t}//W$ is the Chevalley scheme. Then, \eqref{pushforwarddefinition} of Construction~\ref{homomorphism on regular centralizers} defines a homomorphism of $\mathfrak{c}_L$-group schemes 
    $$
    \rho^G_L := \operatorname{Spec}{i_*} : \mathfrak{A}_G \times_{\mathfrak{c}_G} \mathfrak{c}_L \rightarrow \mathfrak{A}_L. 
    $$
    The transitivity property \eqref{homology transitivity} dualizes to the fact that the composition 
    $$
    \mathfrak{A}_G \times_{\mathfrak{c}_G} \mathfrak{c}_M \simeq \left( \mathfrak{A}_G \times_{\mathfrak{c}_G} \mathfrak{c}_L \right) \times_{\mathfrak{c}_L} \mathfrak{c}_M \xrightarrow{\rho^G_L \times \mathrm{id}} \mathfrak{A}_L \times_{\mathfrak{c}_L} \mathfrak{c}_M \xrightarrow{\rho^L_M} \mathfrak{A}_M
    $$
    is equal to $\rho^G_M$. 
    
\end{remark}

\subsection{Summary}\label{automorphic summary} We will now bring together all of the ingredients from §\ref{automorphic}. We start by reviewing the notation that we use when multiple Levi subgroups of $G$ are in play. 

\begin{notation}
    Recall that $\Delta$ denotes the set of simple roots of $G$. For each subset $I \subseteq \Delta$, we define $L_I$ to be the corresponding standard Levi subgroup. That is, $L_I$ is the unique Levi subgroup of $G$ containing the fixed maximal torus $T$ such that the simple root spaces $\mathfrak{g}_\alpha \subseteq \mathfrak{g}$ contained in $\mathfrak{l}_I := \mathrm{Lie}(L_I) \subseteq \mathfrak{g}$ are exactly those labelled by the roots $\alpha \in I$. For example, $L_{\emptyset} = T$ and $L_\Delta = G$. If $I \subseteq J$ is an inclusion of subsets of $\Delta$, then we have the containment $L_I \subseteq L_J$. Moreover, we define $P_I \subseteq G$ to be the corresponding standard parabolic subgroup. That is, $P_I$ is the unique parabolic subgroup of $G$ containing $T$ such that the negative simple root spaces $\mathfrak{g}_{-\alpha} \subseteq \mathfrak{g}$ contained in $\mathfrak{p}_I := \mathrm{Lie}(P_I)$ are exactly those root spaces labelled by the elements $-\alpha \in -I$. Let $V_I \subseteq P_I$ denote the unipotent radical of $P_I$. Then, $L_I \subseteq P_I$ is a Levi factor of $P_I$ and the quotient $P_I/V_I$ identifies canonically with $L_I$. Once again, an inclusion $I \subseteq J$ yields inclusions $P_I \subseteq P_J$ and $V_J \subseteq V_I$. Let $W_I \subseteq W$ denote the Weyl group of $L_I$ (it is the subgroup of $W$ generated by the simple reflections $s_\alpha \in W$ for $\alpha \in I$). Let $\Phi_I \subseteq \Phi$ denote the set of roots of $L_I$.

    Let $\Gr_I := \Gr_{L_I}$. Instead of $\Gr_\Delta = \Gr_G$, we simply write $\Gr$ in this subsection. If $I \subseteq J$, we have the inclusion $L_I \subseteq L_J$ and therefore have a closed immersion $i_{I \subseteq J}: \Gr_I \hookrightarrow \Gr_J$ of affine Grassmannians. When $J = \Delta$, we simply write $i_I$ for $i_{I \subseteq \Delta}$.
\end{notation}

\begin{notation}
    Let $I \subseteq \Delta$. Let $R_I := R_{L_I} = H_{L_I}^*(\mathrm{pt},\CC)$. We write $R$ instead of $R_\Delta$. For any inclusion $I \subseteq J$, we have a homomorphism $R_J \rightarrow R_I$. Recall that there is a canonical isomorphism of graded $\CC$-algebras
    $$
    R_\emptyset = H_T^*(\mathrm{pt}, \CC) \simeq \operatorname{Sym}{\mathfrak{t}}^*
    $$
    where $\mathfrak{t} = \mathrm{Lie}(T)$. Moreover, the Weyl group $W$ acts naturally on $R_\emptyset$ and the homomorphism $R_I \rightarrow R_\emptyset$ identifies $R_I$ with the $W_I \subseteq W$-invariants. Hence, there is a canonical isomorphism 
    $$
    R_I \simeq \left(\operatorname{Sym}{\mathfrak{t}}^* \right)^{W_I}.
    $$
    In particular, the $W_I$-invariant products 
    \begin{align}
    f_I &:= \prod_{\alpha \in \Phi_I} \alpha \in \left(\operatorname{Sym}{\mathfrak{t}}^* \right)^{W_I} \label{fIdefinition}\\
    g_I &:= \prod_{\alpha \not\in \Phi_I} \alpha\in \left(\operatorname{Sym}{\mathfrak{t}}^* \right)^{W_I} \label{gIdefinition}
    \end{align}
    define elements of $R_I$. 
\end{notation}

\begin{remark}
    Let $I \subseteq J \subseteq \Delta$. In §§\ref{parabolic restriction}, we studied the (unnormalized, equivariant) parabolic restriction functor \eqref{equivariant parabolic restriction} of Beilinson-Drinfeld \cite[§5.3.28]{beilinsondrinfeld}
    $$
    \operatorname{Res}_{I \subseteq J} := \operatorname{Res}^{L_J}_{L_I}: D_{L_J(\mathcal{O})}(\Gr_J) \rightarrow D_{L_I(\mathcal{O})}(\Gr_I). 
    $$
    When $J = \Delta$, we simply write $\operatorname{Res}_I$ for $\operatorname{Res}_{I \subseteq \Delta}$. One of the main constructions of §§\ref{parabolic restriction} is Construction~\ref{Xiconstruction}, which defines a natural transformation \eqref{Xidefinition} 
    $$
    \Xi_{I \subseteq J}: i_{I \subseteq J}^! \rightarrow \operatorname{Res}_{I \subseteq J}
    $$
    of functors $D_{L_J(\mathcal{O})}(\Gr_J) \rightarrow D_{L_I(\mathcal{O})}(\Gr_I)$. Let $A \in D_{L_{J}(\mathcal{O})}(\Gr_J)$ denote a $L_J(\mathcal{O})$-equivariant complex. Then, we can evaluate the natural transformation $\Xi_{I\subseteq J}$ on $A$ and pass to $L_I(\mathcal{O})$-equivariant cohomology to obtain an $R_I$-module homomorphism
    \begin{equation}\label{map induced by XiI on cohomology}
    \xi_{I \subseteq J} := H^*_{L_I(\mathcal{O})}(\Xi_{I \subseteq J}): H_{L_I(\mathcal{O})}^*(\Gr_I, i_{I \subseteq J}^! A) \rightarrow H_{L_I(\mathcal{O})}^*(\Gr_I, \operatorname{Res}_{I \subseteq J}(A)). 
    \end{equation}
    When $J = \Delta$, we make the usual abbreviations $\Xi_I := \Xi_{I \subseteq \Delta}$ and $\xi_{I} := \xi_{I \subseteq \Delta}$. We can now bring in the results of §§\ref{equivariant localization subsection} to establish the following proposition. 
\end{remark}

\begin{proposition}\label{corestriction vs parabolic restriction localization}
    Let $A \in D_{G(\mathcal{O})}(\Gr)$ denote a $G(\mathcal{O})$-equivariant complex on $\Gr$. Assume that the underlying complex $\mathrm{For}^{G(\mathcal{O})}(A) \in D(\Gr)$ is $!$-parity (in the sense of \cite{paritysheaves}; see also Remark~\ref{parityandformality} above). Consider the natural $R_I$-module homomorphism 
    $$
    \xi_{I}: H_{L_I(\mathcal{O})}^*(\Gr_I, i_I^! A) \rightarrow H_{L_I(\mathcal{O})}^*(\Gr_I, \operatorname{Res}_I(A))
    $$
    of \eqref{map induced by XiI on cohomology}. Then, the localization $(\xi_I)_{g_I}$ at the element $g_I \in R_I$ of \eqref{fIdefinition} is an isomorphism. Moreover, $\xi_I$ is injective. These assumptions apply if $A \in \mathcal{P}_{G(\mathcal{O})}(\Gr)$ is perverse. 
\end{proposition}
\begin{proof}
    For each $\chi \in \pi_0(\Gr_I)$ corresponding to a connected component $\Gr^\chi_I \subseteq \Gr_I$, let $i_I^\chi: \Gr_I^\chi \hookrightarrow \Gr$ denote the induced closed immersion into $\Gr$. Let $\Gr_{P_I}^\chi = \Gr_{P_I} \times_{\Gr_I} \Gr_I^{\chi} \subseteq \Gr_{P_I}$. We have inclusions $k^\chi_I: \Gr_I^\chi \hookrightarrow \Gr_{P_I}^\chi$ (a closed immersion) and $r_I^\chi: \Gr_P^{\chi} \rightarrow \Gr$ such that $i_I^\chi = r_I^\chi \circ k^\chi_I$.  
    
    Since $\Gr_{P_I}$ (resp. $\Gr_I$) is the scheme-theoretic disjoint union over $\chi \in \pi_0(\Gr_I)$ of its sub-ind-schemes $\Gr_{P_I}^\chi$ (resp. $\Gr_I^\chi$), it follows from the definition \eqref{equivariant parabolic restriction} that $\xi_I$ is the direct sum over $\chi \in \pi_0(\Gr_I)$ of the natural $R_I$-module homomorphisms
    \begin{align*}
    \xi_{I}^\chi: H_{L_I(\mathcal{O})}^*(\Gr_I^\chi, (i_I^\chi)^! A) \simeq\,  &H_{L_I(\mathcal{O})}^*(\Gr_I^\chi, (k^\chi_I)^!(r_I^\chi)^! A)\\
    \simeq\, &H_{L_I(\mathcal{O})}^*(\Gr_{P_I}^\chi, (k^\chi_I)_*(k^\chi_I)^!(r_I^\chi)^! A)\\
    \rightarrow\, &H^*_{L_I(\mathcal{O})}(\Gr_{P_I}^\chi, (r_I^\chi)^!A).
    \end{align*}
    It therefore suffices to show that each $\xi_I^\chi$ is injective and localizes to an isomorphism away from $g_I \in R_I$. 

    By Proposition~\ref{stabilizercomputation}, the element $g_I \in \operatorname{Sym}{\mathfrak{t}^*}$ vanishes on the stabilizer $\mathrm{Lie}(T_x)$ for any $x \in \Gr_{P_I}^\chi \setminus \Gr^\chi_I$ (where $T_x \subseteq T$ is the stabilizer of the point $x \in \Gr$ in $T$). It follows from the equivariant localization theorem, say in the form of \cite[Theorem A.1.13]{zhu2016introduction}, that the natural $R_\emptyset$-module homomorphism
    $$
    \tilde{\xi}_I^\chi: H_{T(\mathcal{O})}^*(\Gr_I^\chi, (i_I^\chi)^! A) \simeq H_{T(\mathcal{O})}^*(\Gr_{P_I}^\chi, (k^\chi_I)_*(k^\chi_I)^!(r_I^\chi)^! A) \rightarrow H^*_{T(\mathcal{O})}(\Gr_{P_I}^\chi, (r_I^\chi)^!A)
    $$
    becomes an isomorphism after localization at $g_I \in R_I$. See the proof of Proposition~\ref{localization} and the discussion preceding it for a slightly more detailed explanation. 

    Since $A$ is $!$-parity, it follows from the proof of Proposition~\ref{equivariantformality} (see Remark~\ref{parityandformality}) that the complexes $(i^\chi_I)^!A$ and and $(r_I^\chi)^!A$ are $T$-equivariantly formal. Hence, applying the functor of $W_I$-invariants to $\tilde{\xi}_I^\chi$ recovers the morphism $\xi_I^\chi$. Since taking $W_I$-invariants commutes with localization at the $W_I$-invariant element $g_I \in R_I$ (exercise), it follows that $\xi_I^\chi$ becomes an isomorphism after inverting $g_I$. Moreover, the equivariant formality of $(i_I^\chi)^!A$ yields the freeness of $H_{L_I(\mathcal{O})}(\Gr_I^\chi, (i_I^\chi)^!A)$ over the ring $R_I$, from which we deduce the injectivity of $\xi_I^\chi$ (as in the proof of Proposition~\ref{localization}).
    
    If $A \in \mathcal{P}_{G(\mathcal{O})}(\Gr)$ is perverse, then it is a direct sum of irreducible perverse sheaves. Hence, it follows from the proof of Proposition~\ref{equivariantformality} that $A$ is $!$-parity, verifying the last assertion of the proposition. 
\end{proof}

\begin{remark}\label{xiIisaringhomomorphism}
    We can apply Proposition~\ref{corestriction vs parabolic restriction localization} to the case in which $A = \mathcal{F}_\reg$ is the regular object. Recall from Remark~\ref{ring structure on the regular sheaf} that $\mathcal{F}_\reg$ is naturally a ring object in $D_{G(\mathcal{O})}(\Gr)$. In §§\ref{monoidal structure section}, we equipped the functor $i_I^!: D_{G(\mathcal{O})}(\Gr) \rightarrow D_{L_I(\mathcal{O})}(\Gr_I)$ with a lax monoidal structure. Hence, $i_I^!\mathcal{F}_\reg$ is equipped with a ring structure in $D_{L_I(\mathcal{O})}(\Gr_I)$. On the other hand, in §§\ref{parabolic restriction}, we recalled the construction of Beilinson-Drinfeld \cite{beilinsondrinfeld} of a lax monoidal structure on the parabolic restriction functor $\mathrm{Res}_I: D_{G(\mathcal{O})}(\Gr) \rightarrow D_{L_I(\mathcal{O})}(\Gr_I)$. Hence, $\operatorname{Res}_I(\mathcal{F}_\reg)$ is naturally a ring object in $D_{L_I(\mathcal{O})}(\Gr_I)$. In Proposition~\ref{parabolic restriction vs corestriction} we showed that the natural map 
    $$
    \Xi_I: i_I^! \rightarrow \operatorname{Res}_I
    $$
    is a natural transformation of lax monoidal functors. Hence, the morphism 
    $$
    \Xi_I(\mathcal{F}_\reg): i_I^!\mathcal{F}_\reg \rightarrow \operatorname{Res}_I(\mathcal{F}_\reg)
    $$
    is a ring homomorphism. Since $H_{L_I(\mathcal{O})}^*(\Gr_I, -)$ is a lax monoidal functor (see Remark~\ref{lax monoidal structure on cohomology}), we deduce that the natural map 
    $$
    \xi_I: H_{L_I(\mathcal{O})}^*(\Gr_I, i_I^!\mathcal{F}_\reg) \rightarrow H_{L_I(\mathcal{O})}^*(\Gr_I, \operatorname{Res}_I(\mathcal{F}_\reg))
    $$
    is a homomorphism of graded $R_I$-\textit{algebras}. By Proposition~\ref{corestriction vs parabolic restriction localization} above, its localization at $g_I \in R_I$ is a canonical isomorphism 
    \begin{equation}\label{ungraded anti-generic comparison automorphic side}
        (\xi_I)_{g_I}: H_{L_I(\mathcal{O})}^*(\Gr_I, i_I^!\mathcal{F}_\reg)_{g_I} \xrightarrow{\sim} H_{L_I(\mathcal{O})}^*(\Gr_I, \operatorname{Res}_I(\mathcal{F}_\reg))_{g_I}
    \end{equation}
    of graded $(R_I)_{g_I}$-algebras. 
\end{remark}
\begin{remark}\label{normalized corestriction}
    Recall from Remark~\ref{normalized parabolic restriction} that the parabolic restriction functor should be normalized so as to preserve perversity. Mimicking the definition given there of $\operatorname{Res}_I^{\natural} := \operatorname{Res}^{G,\natural}_{L_I}$, we introduce the normalized version of the corestriction functor 
    $$
    i^{!,\natural}_I := \bigoplus_{\chi \in \pi_0(\Gr_I)} i^{!}_I \vert_{\Gr_I^{\chi}}[-\langle 2\rho_G - 2\rho_L, \chi \rangle ]. 
    $$
    The morphism $\Xi_I$ now gives rise to a morphism of $\pi_0(\Gr_I) = \Lambda/\Lambda_I$-graded functors 
    $$
    \Xi_I^\natural: i^{!,\natural}_I \rightarrow \operatorname{Res}_I^{\natural}. 
    $$
    Hence, the isomorphism of \eqref{ungraded anti-generic comparison automorphic side} becomes an isomorphism 
    \begin{equation}\label{graded anti-generic comparison automorphic side}
    (\xi_I^\natural)_{g_I}: H_{L_I(\mathcal{O})}^*(\Gr_I, i_I^{!,\natural}\mathcal{F}_\reg)_{g_I} \xrightarrow{\sim} H_{L_I(\mathcal{O})}^*(\Gr_I, \operatorname{Res}_I^\natural(\mathcal{F}_\reg))_{g_I}
    \end{equation}
    of ($\ZZ$-)graded $(R_I)_{g_I}$-algebras.
\end{remark}
\begin{remark}\label{anti-generic comparison automorphic side and the regular centrlizer}
    Recall the $\mathfrak{c}_I:=\mathfrak{c}_{L_I}$-group scheme $\mathfrak{A}_I := \mathfrak{A}_{L_I} := \operatorname{Spec}{H_*^{L_I(\mathcal{O})}}(\Gr_I, \CC)$ from Remark~\ref{homomorphism on regular centralizers, group schemes}. In Construction~\ref{comodulestructure}, it was shown that the cohomology $H^*_{L_I(\mathcal{O})}(\Gr_I, A)$ of any complex $A \in D_{L_I(\mathcal{O})}(\Gr_I)$ is naturally a $\mathfrak{A}_I$-module (i.e. a $H_*^{L(\mathcal{O})}(\Gr_I, \CC)$-comodule). It is evident that if $A \rightarrow B$ is a morphism in $D_{L_I(\mathcal{O})}(\Gr_I)$, then the induced map on equivariant cohomology $H^*_{L_I(\mathcal{O})}(\Gr_I, A) \rightarrow H^*_{L_I(\mathcal{O})}(\Gr_I, B)$ is a $\mathfrak{A}_I$-module homomorphism. Applying this observation to the morphisms $\Xi_I(\mathcal{F}_\reg)$ and $\Xi_I^\natural(\mathcal{F}_\reg)$, we deduce that the isomorphisms $(\xi_I)_{g_I}$ of \eqref{ungraded anti-generic comparison automorphic side} and $(\xi_I^{\natural})_{g_I}$ of \eqref{graded anti-generic comparison automorphic side} are isomorphisms of graded $\mathfrak{A}_I \times_{\mathfrak{c}_I} \mathfrak{c}_{I-\gen}^-$-algebras. 
\end{remark}

\begin{remark}\label{generic comparison automorphic side}
    On the other hand, by Proposition~\ref{localization}, the natural map of graded $R_\emptyset$-modules
    $$
    H^*_{T(\mathcal{O})}(\Gr_T, i_{\emptyset}^{!} \mathcal{F}_\reg) \rightarrow H^*_{T(\mathcal{O})}(\Gr_I, i_I^!\mathcal{F}_\reg)
    $$
    becomes an isomorphism after localization at $f_I \in R_\emptyset$. By the equivariant formality of $i_I^!\mathcal{F}_\reg$, we have an isomorphism of $R_I$-modules
    $$
    H^*_{L_I(\mathcal{O})}(\Gr_I, i_I^!\mathcal{F}_\reg) \otimes_{R_I} R_{\emptyset} \simeq H^*_{T(\mathcal{O})}(\Gr_I, i_I^!\mathcal{F}_\reg). 
    $$
    Localization at $f_I \in R_\emptyset$ therefore yields an isomorphism of $R_\emptyset$-modules 
    $$
    H^*_{L_I(\mathcal{O})}(\Gr_I, i_I^!\mathcal{F}_\reg)_{f_I} \otimes_{(R_I)_{f_I}} (R_{\emptyset})_{f_I} \simeq H^*_{T(\mathcal{O})}(\Gr_T, i_{\emptyset}^{!} \mathcal{F}_\reg)_{f_I}. 
    $$
\end{remark}

\section{Spectral Side}\label{Hamiltonian Reduction}

We will now study the ``spectral side'' of Theorem~\ref{Main Theorem}. Namely, we will study the relationship between the Hamiltonian $\check{G}$-varieties $T^*(\check{G}/(\check{U}, \psi_I))$ for varying additive characters $\psi_I \in \check{\mathfrak{u}}^*$.

\subsection{Partial Kostant-Whittaker reduction} We will start by defining the main construction appearing on the spectral side of our work, the \textit{partial Kostant-Whittaker reduction}. 

\begin{notation}\label{additivecharacterdefinition}
    Fix a subset $I \subseteq \Delta$ of simple roots. We define the additive character $\psi_I$ of $\check{U}$ to be the composition 
    $$
    \begin{tikzcd}
    \check{U} \ar[r,twoheadrightarrow] & \check{U}/[\check{U},\check{U}] \ar[r,"\sim"] & \bigoplus_{\alpha \in \Delta} \mathbb{G}_a \ar[r,"\mathrm{pr}_I",twoheadrightarrow] & \bigoplus_{\alpha \in I} \mathbb{G}_a \ar[r, "+"] & \mathbb{G}_a.
    \end{tikzcd}
    $$
    Here, $\mathrm{pr}_I$ is given by projection onto the factors indexed by $I$, and the last map is summation. We also denote by $\psi_I$ the linear form $(d\psi_I)_1 \in \check{\mathfrak{u}}^*$, where $1 \in \check{U}$ is the identity. We write $\check{\mathfrak{u}}^\perp \subseteq \check{\mathfrak{g}}^*$ for the linear complement to $\check{\mathfrak{u}} \subseteq \check{\mathfrak{g}}$. 
\end{notation}

Throughout this section, $\check{M}$ denotes a smooth and quasi-affine Hamiltonian $\check{G}$-scheme, equipped with a moment map $\mu: \check{M} \rightarrow \check{\mathfrak{g}}^*$. 

\begin{remark}
    To avoid questions of derived algebraic geometry, we will assume that $\check{G}$ acts freely on $\check{M}$, or equivalently that $\check{M}$ is smooth over $\check{\mathfrak{g}}^*$. In particular, we assume that $\check{M}$ is flat over $\check{\mathfrak{g}}^*$.
\end{remark}

\begin{construction}
    On one hand, we may form the Hamiltonian reduction of $\check{M}$ with respect to the $\check{U}$-action to obtain a Hamiltonian $\check{T}$-scheme $\check{M}\tripslash\check{U} := (\check{M} \times_{\check{\mathfrak{g}}^*} \check{\mathfrak{u}}^\perp)/\check{U}$. Via the reduced moment map 
    $$
    \mu_{\mathrm{red}}: \check{M}\tripslash\check{U} \rightarrow \check{\mathfrak{t}}^*,
    $$
    we regard $\check{M}\tripslash\check{U}$ as a $\check{\mathfrak{t}}^*$-scheme. It carries an action of the constant $\check{\mathfrak{t}}^*$-group scheme $\check{T}_{\check{\mathfrak{t}}^*}$. 
\end{construction}

\begin{construction}
    On the other hand, we may form the \textit{partial Kostant-Whittaker reduction} 
    $$
    \check{M}\tripslash(\check{U}, \psi_I) := (\check{M} \times_{\check{\mathfrak{g}}^*} (\check{\mathfrak{u}}^\perp + \psi_I))/\check{U}.
    $$
    Here, we slightly abuse notation and write $\check{\mathfrak{u}}^\perp + \psi_I$ for the space of linear forms on $\check{\mathfrak{g}}$ which restrict to $\psi_I$ on the subspace $\check{\mathfrak{u}}$. We have a natural projection
    \begin{equation}\label{Characteristic polynomial}
    \chi_{\check{M}}: \check{M}\tripslash(\check{U},\psi_I) \rightarrow (\check{\mathfrak{u}}^\perp + \psi_I)/\check{U} \hookrightarrow \check{\mathfrak{g}}^*/\check{U} \rightarrow \check{\mathfrak{c}}. 
    \end{equation}
    Here, we write $\check{\mathfrak{c}}$ for the (coadjoint) Chevalley space of $\check{G}$; that is, the spectrum of $\check{G}$-invariant polynomials on $\check{\mathfrak{g}}^*$. It is the GIT quotient of $\check{\mathfrak{g}}^*$ under the coadjoint action of $\check{G}$ and the coarse moduli scheme underlying the Artin stack $\check{\mathfrak{g}}^*/\check{G}$. We let $\chi: \check{\mathfrak{g}}^* \rightarrow \check{\mathfrak{c}}$ denote the natural projection. It is $\check{G}$-invariant, hence induces a morphism $\check{\mathfrak{g}}^*/\check{U} \rightarrow \check{\mathfrak{c}}$. This morphism is the last map in the composition \eqref{Characteristic polynomial}. 
\end{construction}

\begin{construction}
    Since the character $\psi_I$ is trivial on the unipotent radical $\check{V}_I$ of $\check{P}_I$, it descends to an additive character $\overline{\psi}_I$ of $\check{U}/\check{V}_I = \check{U}_I$, which we identify with a maximal unipotent subgroup of $\check{L}_I \simeq \check{P}_I/\check{V}_I$. Since $\psi_I$ is non-trivial on each simple root space of $\check{L}_I$, the character $\overline{\psi}_I$ is non-degenerate. Let $\check{\mathfrak{v}}_I = \mathrm{Lie}(\check{V}_I)$ and write $\check{\mathfrak{u}}_I^\perp + \overline{\psi}_I$ for the space of linear forms on $\check{\mathfrak{l}}_I$ restricting to $\overline{\psi}_I$ on $\check{\mathfrak{u}}_I$.  We have an evident equality of closed subschemes of $\check{\mathfrak{g}}^*$
    $$
    \check{\mathfrak{u}}^\perp + \psi_I = \check{\mathfrak{v}}_I^\perp \times_{\check{\mathfrak{l}}_I^*} (\check{\mathfrak{u}}_I^\perp + \overline{\psi}_I).
    $$
    Here, the projection $\check{\mathfrak{v}}_I^\perp \rightarrow \check{\mathfrak{l}}_I^*$ is given by restriction to $\check{\mathfrak{p}}_I$ followed by descent along the projection $\check{\mathfrak{p}}_I \twoheadrightarrow \check{\mathfrak{l}}_I$. We may form the fiber product with $\check{M}$ over $\check{\mathfrak{g}}^*$ to obtain an isomorphism 
    \begin{equation*}
        \tilde{\eta}_I: \check{M} \times_{\check{\mathfrak{g}}^*} \check{\mathfrak{v}}_I^\perp \times_{\check{\mathfrak{l}}_I^*} (\check{\mathfrak{u}}_I^\perp + \overline{\psi}_I) \simeq \check{M} \times_{\check{\mathfrak{g}}} (\check{\mathfrak{u}}^\perp + \psi_I). 
    \end{equation*}
    Passing to quotients by the induced $\check{U}$-actions yields an isomorphism 
    \begin{align}\label{Iterated reduction computation}
        \eta_I: (\check{M}\tripslash\check{V}_I)\tripslash(\check{U}, \overline{\psi}_I) &= ((\check{M} \times_{\check{\mathfrak{g}}^*} \check{\mathfrak{v}}_I^\perp)/\check{V}_I \times_{\check{\mathfrak{l}}_I^*} (\check{\mathfrak{u}}_I^\perp + \overline{\psi}_I))/\check{U}_I\\
        &\simeq (\check{M} \times_{\check{\mathfrak{g}}^*} \check{\mathfrak{v}}_I^\perp \times_{\check{\mathfrak{l}}_I^*} (\check{\mathfrak{u}}_I^\perp + \overline{\psi}_I))/\check{U} \nonumber \\
        &\simeq (\check{M} \times_{\check{\mathfrak{g}}} (\check{\mathfrak{u}}^\perp + \psi_I))/\check{U} \nonumber \\
        &= \check{M}\tripslash(\check{U}, \psi_I).\nonumber
    \end{align}
    Let $\check{\mathfrak{c}}_I = \operatorname{Spec} \mathcal{O}(\check{\mathfrak{l}}_I^*)^{\check{L}_I}$ denote the Chevalley space of $\check{L}_I$. As a Hamiltonian reduction of a Hamiltonian $\check{L}_I$-space, the iterated reduction $(\check{M}\tripslash\check{V}_I)\tripslash(\check{U}_I, \overline{\psi}_I)$ is equipped with a natural map
    \begin{equation}\label{projectiontochevalley}
    \chi_{\check{M}\tripslash\check{V}_I}: (\check{M}\tripslash\check{V}_I)\tripslash(\check{U}_I, \overline{\psi}_I) \rightarrow \check{\mathfrak{c}}_I.
    \end{equation}
\end{construction}

\begin{remark}
    The triangular decomposition $\check{\mathfrak{g}} = \check{\mathfrak{u}} \oplus \check{\mathfrak{t}} \oplus \check{\mathfrak{u}}^-$ yields a projection $\check{\mathfrak{g}} \twoheadrightarrow \check{\mathfrak{l}}_I$. Dualizing, we obtain a closed immersion $\check{\mathfrak{l}}_I^* \hookrightarrow \check{\mathfrak{g}}^*$ intertwining the coadjoint actions of $\check{L}_I$ and $\check{G}$. We obtain an induced map on invariant-theoretic quotients 
    $$
    \pi_I: \mathfrak{c}_I \rightarrow \mathfrak{c}. 
    $$
\end{remark}

The following result is useful for reducing the study of a partial Kostant-Whittaker reduction to the non-degenerate case.

\begin{lemma}\label{Reduction in Steps}
    Equip $\check{M}\tripslash\check{V}_I$ with its natural structure of Hamiltonian $\check{L}_I$-scheme. We have a canonical isomorphism of schemes
    $$
    \eta_I: (\check{M}\tripslash\check{V}_I)\tripslash(\check{U}_I, \overline{\psi}_I) \simeq \check{M}\tripslash(\check{U},\psi_I)
    $$
    fitting into a commutative diagram
    \begin{equation}
    \begin{tikzcd}\label{Reduction in Steps, Diagram 1}
        (\check{M}\tripslash\check{V}_I)\tripslash(\check{U}_I, \overline{\psi}_I) \ar[d,"\eta_I"] \ar[rr, "\chi_{\check{M}\tripslash\check{V}_I}"] & &  \check{\mathfrak{c}}_I \ar[d, "\pi_I"]  \\
        \check{M}\tripslash(\check{U},\psi_I) \ar[rr, "\chi_{\check{M}}"] & &
        \check{\mathfrak{c}}.
    \end{tikzcd}
    \end{equation}
\end{lemma}
\begin{proof}
    The isomorphism $\eta_I$ is that of \eqref{Iterated reduction computation}. It remains to check that the diagram \eqref{Reduction in Steps, Diagram 1} commutes. It suffices to show that following diagram commutes.
    $$
    \begin{tikzcd}
        (\check{M} \times_{\check{\mathfrak{g}}^*} \check{\mathfrak{v}}_I^\perp ) \times_{\check{\mathfrak{l}}_I^*} (\check{\mathfrak{u}}_I^\perp + \overline{\psi}_I) \ar[d, "\tilde{\eta}_I"] \ar[r] & \check{\mathfrak{v}}_I^\perp \times _{\check{\mathfrak{l}}_I^*} (\check{\mathfrak{u}}_I^\perp + \overline{\psi}_I) \ar[d, "\sim"] \ar[r, "\mathrm{pr}_2"] & \check{\mathfrak{u}}_I^\perp + \overline{\psi}_I \ar[r, hookrightarrow] & \check{\mathfrak{l}}_I^* \ar[r, "\chi_I"] & \check{\mathfrak{c}}_I \ar[d, "\pi_I"] \\
        \check{M} \times_{\check{\mathfrak{g}}^*} (\check{\mathfrak{u}}^\perp + \psi_I) \ar[r, "\mathrm{pr}_2"] & \check{\mathfrak{u}}^\perp + \psi_I \ar[rr, hookrightarrow] & & \check{\mathfrak{g}}^* \ar[r, "\chi"] & \check{\mathfrak{c}}.
    \end{tikzcd}
    $$
    Here, $\chi_I: \check{\mathfrak{l}}_I^* \rightarrow \check{\mathfrak{c}}_I$ is the characteristic polynomial map on $\check{\mathfrak{l}}_I^*$. The commutativity of the left square follows immediately from the definition of $\tilde{\eta}_I$. Under the identification $\check{\mathfrak{g}} \simeq \check{\mathfrak{g}}^*$, the commutativity of the right rectangle is equivalent to the commutativity of the following diagram.
    $$
    \begin{tikzcd}
        \check{\mathfrak{b}} + e_I \ar[d, hookrightarrow] \ar[r, twoheadrightarrow] & \check{\mathfrak{b}}_I + \overline{e}_I 
        \ar[r, hookrightarrow] & \check{\mathfrak{l}}_I \ar[d, equals]\ar[r, "\chi_I"] & \check{\mathfrak{c}}_I \ar[d, equals] \\
        \check{\mathfrak{p}}_I \ar[rr, twoheadrightarrow] \ar[d, hookrightarrow] & &\check{\mathfrak{l}}_I \ar[r, "\chi_I"] \ar[r] & \check{\mathfrak{c}}_I \ar[d, "\pi_I"] \\
        \check{\mathfrak{g}} \ar[rrr, "\chi"] & & & \check{\mathfrak{c}}.
    \end{tikzcd}
    $$
    The commutativity of the upper half of the diagram is obvious. The bottom half embeds into the diagram
    $$
    \begin{tikzcd}
        \check{\mathfrak{b}} \ar[d, hookrightarrow] \ar[r, twoheadrightarrow] & \check{\mathfrak{b}}_I \ar[r, twoheadrightarrow] \ar[d, hookrightarrow]& \check{\mathfrak{t}} \ar[r, equals, "\chi_\emptyset"] & \check{\mathfrak{t}} \ar[d, "\pi"] \\
        \check{\mathfrak{p}}_I \ar[r, twoheadrightarrow] \ar[d, hookrightarrow] & \check{\mathfrak{l}}_I \ar[rr, "\chi_I"] & & \check{\mathfrak{c}}_I \ar[d, "\pi_I"] \\
        \check{\mathfrak{g}} \ar[rrr, "\chi"] & & & \check{\mathfrak{c}}.
    \end{tikzcd}
    $$
    The two maps $\check{\mathfrak{p}}_I \rightarrow \check{\mathfrak{c}}$ that we claim are equal are $\check{P}_I$-invariant. Since every element of $\check{\mathfrak{p}}_I$ is conjugate to an element of $\check{\mathfrak{b}}$, it suffices to show that the compositions $\check{\mathfrak{b}} \hookrightarrow \check{\mathfrak{p}}_I \rightarrow \check{\mathfrak{c}}$ coincide. Given the obvious commutativity of the upper left square, it suffices to show that the outer rectangle and upper right rectangle commute. The upper right rectangle becomes the outer rectangle after replacing $\check{\mathfrak{g}}$ by $\check{\mathfrak{l}}_I$, so it suffices to show the commutativity of the outer rectangle. 

    Thus, we are reduced to the observation that the diagram 
    $$
    \begin{tikzcd}
        \check{\mathfrak{b}} \ar[r, twoheadrightarrow] \ar[d, hookrightarrow] & \check{\mathfrak{t}} \ar[r, hookrightarrow] & \check{\mathfrak{g}} \ar[ld, "\chi"] \\
        \check{\mathfrak{g}} \ar[r, "\chi"] &\check{\mathfrak{c}}
    \end{tikzcd}
    $$
    commutes, which is standard. 
\end{proof}
\subsection{Action of the regular centralizer} \label{spectralactionofthecentralizer}

We begin by reviewing Ng\^{o}'s construction of the regular centralizer group scheme as explained in \cite{Ri17}. Let $\check{\mathfrak{g}}^*_{\reg} \subseteq \check{\mathfrak{g}}^*$ denote the open subscheme of regular elements; namely, those elements $\xi \in \check{\mathfrak{g}}^*$ with coadjoint centralizer $Z_{\check{G}}(\xi)$ of minimal dimension. Over $\check{\mathfrak{g}}^*$, we have the universal (coadjoint) centralizer group scheme $\mathfrak{I}_{\mathrm{univ}}$, defined as the fiber product
$$
\begin{tikzcd}
\mathfrak{I}_{\mathrm{univ}} \ar[r, hookrightarrow] \ar[d] & \check{G} \times \check{\mathfrak{g}}^* \ar[d,"\operatorname{act} \times \operatorname{pr}_2"] \\
\check{\mathfrak{g}}^* \ar[r, hookrightarrow, "\Delta"] & \check{\mathfrak{g}}^* \times \check{\mathfrak{g}}^*.
\end{tikzcd}
$$
We form the restriction $\mathfrak{I}_{\mathrm{reg}} = \mathfrak{I}_{\mathrm{univ}} \times_{\check{\mathfrak{g}}^*} \check{\mathfrak{g}}_{\mathrm{reg}}^*$, a commutative affine group scheme over $\check{\mathfrak{g}}^*_{\reg}$. In \cite[Lemme 2.1.1]{Ngo10}, smooth descent (of affine morphisms) is used to construct a smooth affine group scheme $\mathfrak{J}$ over $\check{\mathfrak{c}}$ equipped with a canonical isomorphism of group schemes over $\check{\mathfrak{g}}^*_{\mathrm{reg}}$
$$
\mathfrak{I}_{\mathrm{reg}} \simeq \mathfrak{J} \times_{\mathfrak{c},\chi} \check{\mathfrak{g}}^*_{\mathrm{reg}}. 
$$
Furthermore, this isomorphism extends uniquely (by an application of Hartog's principle) to a homomorphism of group schemes
\begin{equation}\label{Ngohomomorphism}
\rho: \mathfrak{J} \times_{\mathfrak{c},\chi} \check{\mathfrak{g}}^*\rightarrow \mathfrak{I}_{\mathrm{univ}}.
\end{equation}
We will need a well-known alternative construction of $\mathfrak{J}$. 

\begin{remark}
We recall the Kostant slice $\kappa: \check{\mathfrak{c}} \rightarrow \check{\mathfrak{g}}^*$ to the regular nilpotent orbit. Consider the regular nilpotent element\footnote{It is a fundamental observation of Ginzburg, going back to \cite{Gi95}, that the nilpotent element $e$ arises through Tannakian formalism from the Chern class $c_1(\mathcal{L}_{\mathrm{det}}) \in H^2(\Gr_G, \CC)$ of the determinant line bundle.}
$$
e = \sum_{\alpha \in \Delta} X_\alpha,
$$
where $X_\alpha \in \check{\mathfrak{g}}_\alpha$ is the non-zero element provided by the pinning of $\check{G}$. The map $\mathrm{ad}^*_e: \check{\mathfrak{b}}^\perp \rightarrow \check{\mathfrak{u}}^\perp$ is injective. Let $\check{\mathfrak{s}} \subseteq \check{\mathfrak{u}}^\perp$ denote a $\mathbb{G}_m$-stable complement to $\mathrm{ad}^*_e(\check{\mathfrak{b}}^\perp)$, where $\mathbb{G}_m$ acts through the homomorphism $2\check{\rho}: \mathbb{G}_m \rightarrow \check{T}$. Then, Kostant's theorem \cite[Theorem 3.2.2]{Ri17} implies that the subspace $\check{\mathfrak{s}} + \psi$ maps isomorphically onto $\check{\mathfrak{c}} = \check{\mathfrak{g}}^*\doubslash \check{G}$. The inverse map followed by the inclusion $\check{\mathfrak{s}} + \psi \hookrightarrow \check{\mathfrak{g}}^*$ defines the morphism $\kappa: \check{\mathfrak{c}} \rightarrow \check{\mathfrak{g}}^*$. Clearly, we have an isomorphism of $\check{\mathfrak{c}}$-group schemes
\begin{equation}\label{Regular Centralizer as Restriction}
\mathfrak{J} \simeq \mathfrak{I}_{\mathrm{univ}} \times_{\check{\mathfrak{g}}^*,\kappa} \check{\mathfrak{c}}.
\end{equation}
\end{remark}
Similarly, we can always describe a Kostant-Whittaker reduction $\check{M}\tripslash(\check{U}, \psi)$ by restricting $\check{M}$ along the Kostant slice:
\begin{lemma}\label{Reduction by Restriction}
    Let $\psi = \psi_\Delta$ denote a non-degenerate additive character of $\check{U}$. There is a canonical $\check{\mathfrak{c}}$-isomorphism 
    $$
    \delta_\psi^{\check{M}}: \check{M}\tripslash(\check{U}, \psi) \simeq \check{M} \times_{\check{\mathfrak{g}}^*,\kappa} \check{\mathfrak{c}}.
    $$
\end{lemma}
\begin{proof}
    By definition, $\check{M}\tripslash(\check{U},\psi) = (\check{M} \times_{\check{\mathfrak{g}}^*} (\check{\mathfrak{u}}^\perp + \psi_I))/\check{U}$. Let 
    $$
    q: \check{M} \times_{\check{\mathfrak{g}}^*} (\check{\mathfrak{u}}^\perp + \psi_I) \rightarrow \check{M}\tripslash(\check{U},\psi)
    $$
    denote the quotient map. It is a $\check{U}$-torsor. So, we have a commutative diagram
    $$
    \begin{tikzcd}
        \check{M} \times_{\check{\mathfrak{g}}^*,\kappa} \check{\mathfrak{c}} \ar[r] \ar[d,"\operatorname{pr}_2"] & \check{M} \times_{\check{\mathfrak{g}}^*} (\check{\mathfrak{u}}^\perp + \psi) \ar[d, "\mathrm{pr}_2"] \ar[r, "q"] & \check{M}\tripslash(\check{U},\psi) \ar[d, "\chi_{\check{M}}"] \\
        \check{\mathfrak{c}} \ar[r,"\kappa"] \ar[rr, "\mathrm{id}_{\check{\mathfrak{c}}}",bend right = 10,swap]& \check{\mathfrak{u}}^\perp + \psi \ar[r,"\chi"] & \check{\mathfrak{c}}.
    \end{tikzcd}
    $$
    By Kostant's theorem \cite[Theorem 3.2.2]{Ri17}, the map $\chi: \check{\mathfrak{u}}^\perp + \psi \rightarrow \check{\mathfrak{c}}$ is a $\check{U}$-torsor. Hence, the commutative square on the right is a morphism of $\check{U}$-torsors, hence is Cartesian. The commutative square on the left is evidently Cartesian. Therefore, the composite square is Cartesian. It follows that the composition of the top row $(\delta_\psi^{\check{M}})^{-1}: \check{M} \times_{\check{\mathfrak{g}}^*,\kappa} \check{\mathfrak{c}} \rightarrow \check{M}\tripslash(\check{U},\psi)$ is an isomorphism, as needed. 
\end{proof}

\begin{example}\label{universal hamiltonian reduction}
Consider the ``universal'' example of $\check{M} = T^*\check{G}$. Then, Lemma~\ref{Reduction by Restriction} is simply the observation that we have a canonical $\check{\mathfrak{c}}$-isomorphism
\begin{equation}\label{deltapsidefinition}
\delta_\psi := \delta_\psi^{T^*\check{G}}: T^*(\check{G}/(\check{U}, \psi)) \simeq T^*\check{G} \times_{\check{\mathfrak{g}}^*,\kappa} \check{\mathfrak{c}} \simeq \check{G} \times \check{\mathfrak{c}}. 
\end{equation}
\end{example}

\begin{remark}
    In our intended application, there will be another algebraic group acting compatibly on $\check{M}$. Therefore, we introduce into our setup an algebraic group $\check{H}$ with an action on $\check{M}$ that commutes with the given action of $\check{G}$. As we construct various isomorphisms in this section, we will stop to observe that they are in fact $\check{H}$-equivariant. 
    
    Because the $\check{H}$-action on $\check{M}$ commutes with the $\check{G}$-action, it follows easily that the moment map $\mu$ is $\check{H}$-invariant. Therefore, $\check{H}$ acts naturally on $\check{M} \times_{\check{\mathfrak{g}}^*} \kappa$. Similarly, $\check{H}$ acts naturally on $\check{M} \times_{\check{\mathfrak{g}}^*} (\check{\mathfrak{u}}^\perp + \psi)$. This action commutes with that of $\check{U}$, so we obtain an action of $\check{H}$ on the quotient $\check{M}\tripslash(\check{U},\psi)$. 
\end{remark}

\begin{lemma}\label{Reduction by restriction, equivariant}
    Suppose that $\check{M}$ is equipped with an action of an algebraic group $\check{H}$ commuting with the action of $\check{G}$. Then, the isomorphism 
    $$
    \delta_\psi^{\check{M}}: \check{M}\tripslash(\check{U}, \psi) \simeq \check{M} \times_{\check{\mathfrak{g}}^*,\kappa} \check{\mathfrak{c}}
    $$
    of Lemma~\ref{Reduction by Restriction} is $\check{H}$-equivariant.
\end{lemma}
\begin{proof}
    It suffices to check that the maps 
    \begin{align*}
    \mathrm{id}_{\check{M}} \times \kappa &: \check{M} \times_{\check{\mathfrak{g}}^*, \kappa} \check{\mathfrak{c}} \rightarrow \check{M} \times_{\check{\mathfrak{g}}^*} (\check{\mathfrak{u}}^\perp + \psi)\\
    q &: \check{M} \times_{\check{\mathfrak{g}}^*} (\check{\mathfrak{u}}^\perp + \psi) \rightarrow \check{M}\tripslash(\check{U},\psi)
    \end{align*}
    are $\check{H}$-equivariant. We defined the $\check{H}$-action on $\check{M}\tripslash(\check{U},\psi)$ as the unique one making the quotient map $\check{M} \times_{\check{\mathfrak{g}}^*} (\check{\mathfrak{u}}^\perp + \psi) \rightarrow \check{M}\tripslash(\check{U},\psi)$ $\check{H}$-invariant, so its $\check{H}$-equivariance is trivial. The map $\mathrm{id}_{\check{M}} \times \kappa$ is obtained by pulling back the (trivially $\check{H}$-equivariant) map $\kappa: \check{\mathfrak{c}} \rightarrow (\check{\mathfrak{u}}^\perp + \psi)$ along the $\check{H}$-equivariant projection $\mathrm{pr}_2: \check{M} \times_{\check{\mathfrak{g}}^*} (\check{\mathfrak{u}}^\perp + \psi) \rightarrow (\check{\mathfrak{u}}^\perp + \psi)$, so the claim follows. 
\end{proof}

\begin{construction}
    Via the isomorphism of Lemma~\ref{Reduction by Restriction}, we may equip the Kostant-Whittaker reduction $\check{M}\tripslash(\check{U},\psi)$ with an action of the group scheme $\mathfrak{I}_{\mathrm{univ}} \times_{\check{\mathfrak{g}}^*,\kappa} \check{\mathfrak{c}}$. Thus, the identification \eqref{Regular Centralizer as Restriction} equips the $\check{\mathfrak{c}}$-scheme $\check{M}\tripslash(\check{U},\psi)$ with an action of the regular centralizer $\mathfrak{J}$. 
    
    Now, let $\mathfrak{J}_{I} \rightarrow \check{\mathfrak{c}}_I$ denote the regular centralizer associated to the reductive group $\check{L}_I$. By \eqref{Reduction in Steps}, we can realize $\check{M}\tripslash(\check{U},\psi_I)$ as the non-degenerate Kostant-Whittaker reduction of the Hamiltonian $\check{L}_I$-space $\check{M}\tripslash\check{V}_I$. Thus, we obtain as above an action of the regular centralizer $\mathfrak{J}_{I}$ on the $\mathfrak{\check{c}}_I$-scheme $\check{M}\tripslash(\check{U},\psi_I)$. 
\end{construction}

\begin{example}
    If $I = \emptyset$, then $\psi_I = 1$ is the trivial character. Hence, $\check{M}\tripslash(\check{U},\psi_I) = \check{M}\tripslash\check{U}$. Furthermore, $\check{\mathfrak{c}}_I = \check{\mathfrak{t}}^*$ and $\mathfrak{J}_{I} = \check{T} \times \check{\mathfrak{t}}^*$. Thus, the $\mathfrak{J}_{I}$-action on $\check{M}\tripslash(\check{U},\psi_I)$ is simply the natural action of the constant group scheme $\check{T} \times \check{\mathfrak{t}}^*$ on $\check{M}\tripslash\check{U}$, which is regarded as a $\check{\mathfrak{t}}^*$-scheme through the reduced $\check{T}$ moment map. 
\end{example}

\subsection{Generic comparison of Hamiltonian and Kostant-Whittaker reductions} We will now compare the quotients $\check{M}\tripslash(\check{U},\psi)$ and $\check{M}\tripslash\check{U}$ (we write $\psi = \psi_\Delta$ for the non-degenerate additive character on $\check{U}$). Let $\check{\mathfrak{t}}^*_{\gen} \subseteq \check{\mathfrak{t}}^*$ denote the open subscheme defined by the non-vanishing $\check{\alpha} \neq 0$ of each coroot $\check{\alpha} \in \Phi$ of $\check{G}$. It is $W$-stable, hence descends to an open subscheme $\check{\mathfrak{c}}_{\gen} \subseteq \check{\mathfrak{c}}$. Note that the projection $\check{\mathfrak{t}}^*_{\gen} \rightarrow \check{\mathfrak{c}}_{\gen}$ is a $W$-torsor. 

We will now establish a version of Lemma~\ref{Reduction by Restriction} for the trivial character. The penalty for taking the Hamiltonian reduction at level $0 \in \check{\mathfrak{u}}^*$ instead of at the generic level $\psi \in \mathfrak{u}^*$ is that we will only able be able to construct an analog of Lemma~\ref{Reduction by Restriction} after restricting to the open subscheme $\check{\mathfrak{t}}^*$. We adapted the following argument directly from Kostant's proof of \cite[Theorem 3.2.2]{Ri17}.

\begin{lemma}\label{Reduction via Restriction II}
There is a canonical $\check{\mathfrak{t}}^*_\gen$-isomorphism
$$
\delta_\gen^{\check{M}}: \check{M}\tripslash \check{U} \times_{\check{\mathfrak{t}}^*} \check{\mathfrak{t}}^*_\gen \simeq \check{M} \times_{\check{\mathfrak{g}}^*} \check{\mathfrak{t}}^*_\gen. 
$$
\end{lemma}
\begin{proof}
    Note that $\check{\mathfrak{t}}^*_\gen \subseteq \check{\mathfrak{g}}^*_{\reg}$. Indeed, under the $\check{G}$-equivariant isomorphism $\check{\mathfrak{g}} \simeq \check{\mathfrak{g}}^*$, the elements of $\check{\mathfrak{t}}^*_\gen$ map to regular semisimple elements of $\check{\mathfrak{g}}$. Let $\check{\mathfrak{u}}_\gen^\perp = \check{\mathfrak{u}}^\perp \times_{\check{\mathfrak{t}}^*} \check{\mathfrak{t}}^*_\gen$ and let 
    $$
    q: \check{\mathfrak{u}}_\gen^\perp \rightarrow \check{\mathfrak{t}}^*_\gen
    $$
    denote the projection to the second factor. Let $p: \check{M} \times_{\check{\mathfrak{g}}^*} \check{\mathfrak{u}}^\perp \rightarrow \check{M}\tripslash \check{U}$ denote the quotient map. It is a $\check{U}$-torsor (recall that $\check{G}$ acts freely on $\check{M}$ by hypothesis). Hence, we obtain a $\check{U}$-torsor
    $$
    p_{\gen}: \check{M} \times_{\check{\mathfrak{g}}^*} \check{\mathfrak{u}}^\perp_\gen \simeq (\check{M} \times_{\check{\mathfrak{g}}^*} \check{\mathfrak{u}}^\perp) \times_{\check{\mathfrak{t}}^*} \check{\mathfrak{t}}^*_\gen \xrightarrow{p} \check{M}\tripslash\check{U} \times_{\check{\mathfrak{t}}^*} \check{\mathfrak{t}}^*_{\gen}. 
    $$
    Moreover, we have a commutative diagram
    $$
    \begin{tikzcd}
    \check{M} \times_{\check{\mathfrak{g}}^*} \check{\mathfrak{t}}^*_\gen \ar[r, hookrightarrow]  \ar[d, "\mathrm{pr}_2"] & \check{M} \times_{\check{\mathfrak{g}}^*} \check{\mathfrak{u}}^\perp_\gen \ar[r,"p_{\gen}"] \ar[d, "\mathrm{pr}_2"] & \check{M}\tripslash\check{U} \times_{\check{\mathfrak{t}}^*} \check{\mathfrak{t}}^*_\gen \ar[d, "\mathrm{pr}_2"] \\
    \check{\mathfrak{t}}^*_\gen \ar[r, hookrightarrow] \ar[rr, "\mathrm{id}_{\check{\mathfrak{t}}_\gen^*}",bend right = 10, swap]& \check{\mathfrak{u}}^\perp_{\gen} \ar[r,"q"]  & \check{\mathfrak{t}}^*_\gen.
    \end{tikzcd}
    $$
    The map $\check{\mathfrak{t}}^*_\gen \hookrightarrow \check{\mathfrak{u}}^\perp_\gen$ is given by pulling back linear forms from $\check{\mathfrak{t}}$ to $\check{\mathfrak{b}}$ and then extending them trivially over $\check{\mathfrak{u}}^-$. We claim that $q$ is a $\check{U}$-torsor. It suffices to show that the action map $a: \check{U} \times \check{\mathfrak{t}}^*_\gen \rightarrow \check{\mathfrak{u}}^\perp_\gen$ is an isomorphism. Under the $\check{G}$-equivariant isomorphism $\check{\mathfrak{g}} \simeq \check{\mathfrak{g}}^*$, the claim translates to the assertion that $\check{U} \times \check{\mathfrak{t}}_\gen \rightarrow \check{\mathfrak{b}}_\gen$ is an isomorphism, where $\check{\mathfrak{b}}_\gen = \check{\mathfrak{b}} \times_{\check{\mathfrak{t}}} \check{\mathfrak{t}}_\gen$ and $\check{\mathfrak{t}}_\gen \subseteq \check{\mathfrak{t}}$ is the open subscheme defined by ${\alpha} \neq 0$ for each root ${\alpha} \in \check{\Phi}$ of $\check{G}$. 
    
    Let $u \in \check{U}$ and $x \in \check{\mathfrak{t}}_\gen$. Suppose that $\mathrm{ad}_u x \in \check{\mathfrak{t}}_\gen$. Then, $u$ conjugates the centralizer $C_{\check{G}}(x)$ to $C_{\check{G}}(\mathrm{ad}_u x)$. Since $x$ and $\mathrm{ad}_u x$ are regular semisimple and belong to the Cartan subalgebra $\check{\mathfrak{t}}$, it follows that $C_{\check{G}}(x) = C_{\check{G}}(\mathrm{ad}_u x) = \check{T}$. Thus, $u$ normalizes $\check{T}$, hence $u \in N(\check{T}) \cap \check{U} = \{1\}$. So, $u = 1$. If $\mathrm{ad}_u x = \mathrm{ad}_{u'} x'$ for $u,u' \in \check{U}$ and $x,x' \in \check{\mathfrak{t}}_\gen$, we deduce that $u = u'$ and $x = x'$. Therefore, $\check{U} \times \check{\mathfrak{t}}_\gen \rightarrow \check{\mathfrak{b}}_\gen$ is injective (on closed points). 

    Let $x \in \check{\mathfrak{b}}_\gen$. We have a Jordan decomposition $x = x_{{ss}} + x_{{n}}$, where $x_{{n}} \in \check{\mathfrak{u}}$. The image of $x$ in $\check{\mathfrak{t}} = \check{\mathfrak{b}}/\check{\mathfrak{u}}$ is equal to that of $x_{{ss}}$, so we conclude that $x_{{ss}}$ is regular semisimple. On the other hand, $[x_{{ss}}, x_{{n}}] = 0$. Therefore, $x_{{n}}$ belongs to the centralizer $\mathfrak{z}_{\check{\mathfrak{g}}}(x_{{ss}})$. Since $x_{{ss}}$ is regular semisimple, it follows that $\mathfrak{z}_{\check{\mathfrak{g}}}(x_{{ss}})$ is a Cartan subalgebra of $\check{\mathfrak{b}}$, hence that $x_{n} = 0$. Furthermore, $\mathfrak{z}_{\check{\mathfrak{g}}}(x_{{ss}})$ maps isomorphically onto $\check{\mathfrak{t}}$, so by Levi's theorem $\mathfrak{z}_{\check{\mathfrak{g}}}(x_{{ss}})$ is $\check{U}$-conjugate to $\check{\mathfrak{t}}$. In particular, $x = x_{{ss}} \in \mathfrak{z}_{\check{\mathfrak{g}}}(x_{{ss}})$ is $\check{U}$-conjugate to an element of $\check{\mathfrak{t}}_\gen$. Therefore, $a: \check{U} \times \check{\mathfrak{t}}_\gen \rightarrow \check{\mathfrak{b}}_\gen$ is surjective (on closed points). Since $a$ is a bijective morphism of smooth varieties (over $\CC$), it is an isomorphism by Zariski's Main Theorem. 
    
    Thus, the square on the right is a morphism of $\check{U}$-torsors, therefore is Cartesian. The square on the left is clearly Cartesian. Hence, the composite square is Cartesian. Therefore, the composition of the top row is an isomorphism $(\delta_\gen^{\check{M}})^{-1}: \check{M} \times_{\check{\mathfrak{g}}^*} \check{\mathfrak{t}}_\gen^* \rightarrow \check{M}\tripslash\check{U} \times_{\check{\mathfrak{t}}^*} \check{\mathfrak{t}}^*_\gen$, as needed. 
\end{proof}

\begin{example}\label{universal kostant-whittaker reduction}
    Consider the example of $\check{M} = T^*\check{G}$. Since $\check{M}\tripslash\check{U} = T^*(\check{G}/\check{U})$, Lemma~\ref{Reduction via Restriction II} provides a canonical $\check{G}$-equivariant $\check{\mathfrak{t}}_\gen^*$-isomorphism 
    \begin{equation}\label{deltacotangentbundledefinition}
    \delta_\gen := \delta_\gen^{T^*\check{G}}: T^*(\check{G}/\check{U}) \times_{\check{\mathfrak{t}}^*} \check{\mathfrak{t}}^*_\gen \simeq T^*\check{G} \times_{\check{\mathfrak{g}}^*} \check{\mathfrak{t}}^*_\gen \simeq (\check{G} \times \check{\mathfrak{g}}^*) \times_{\check{\mathfrak{g}}^*} \check{\mathfrak{t}}^*_\gen = \check{G} \times \check{\mathfrak{t}}^*_\gen. 
    \end{equation}
    Concretely, this means (in particular) that we have given a section to the moment map $\mu_{\mathrm{red}}: T^*(\check{G}/\check{U}) \rightarrow \check{\mathfrak{t}}^*$ over the open subscheme $\check{\mathfrak{t}}^*_\gen \subseteq \check{\mathfrak{t}}^*$ which picks out an element with trivial $\check{G}$-stabilizer.
\end{example}

\begin{remark}\label{Cotangent bundle comparison}
    We can combine Lemma~\ref{Reduction by Restriction} and Lemma~\ref{Reduction via Restriction II} (in the guise of Example~\ref{universal hamiltonian reduction} and Example~\ref{universal kostant-whittaker reduction}) to obtain a $\check{G}$-equivariant (by Lemma~\ref{Reduction by restriction, equivariant}) $\check{\mathfrak{t}}^*_\gen$-isomorphism
    \begin{align}\label{varepsilondefinitionspectral}
    \tilde{\varepsilon}: T^*(\check{G}/\check{U}) \times_{\check{\mathfrak{t}}^*} \check{\mathfrak{t}}^*_\gen &\operatornamewithlimits{\simeq}^{\delta_\gen} \check{G} \times \check{\mathfrak{t}}^*_\gen \\
    &\, = (\check{G} \times \check{\mathfrak{c}}) \times_{\check{\mathfrak{c}}} \check{\mathfrak{t}}^*_\gen \nonumber \\
    &\, \operatornamewithlimits{\simeq}^{\delta_\psi^{-1}} T^*(\check{G}/(\check{U}, \psi)) \times_{\check{\mathfrak{c}}} \check{\mathfrak{t}}^*_\gen. \nonumber 
    \end{align}
    In fact, we can be slightly more precise. In the proof of Lemma~\ref{Reduction via Restriction II}, we defined a morphism 
    $$
    T^*(\check{G}/(\check{U},\psi)) \times_{\check{\mathfrak{c}}} \check{\mathfrak{t}}^* \operatornamewithlimits{\simeq}^{\delta_\psi} (\check{G} \times \check{\mathfrak{c}}) \times_{\check{\mathfrak{c}}} \check{\mathfrak{t}}^* \simeq \check{G} \times \check{\mathfrak{t}}^* \hookrightarrow \check{G} \times \check{\mathfrak{u}}^\perp \twoheadrightarrow (\check{G} \times \check{\mathfrak{u}}^\perp)/\check{U} \simeq T^*(\check{G}/\check{U})
    $$
    whose pullback to $\check{\mathfrak{t}}^*_\gen$ yields the isomorphism \eqref{varepsilondefinitionspectral}. 
    This comparison between $T^*(\check{G}/\check{U})$ and $T^*(\check{G}/(\check{U},\psi))$ over $\check{\mathfrak{t}}^*_\gen$ will facilitate our comparison between the Hamiltonian and Kostant-Whittaker reductions of any Hamiltonian $\check{G}$-variety below (see Proposition~\ref{Generic Comparison}).
\end{remark}

\begin{lemma}\label{Hamiltonian reduction via the cotangent bundle}
    There is a canonical $\check{T}$-equivariant $\check{\mathfrak{t}}^*$-isomorphism
    $$
    \check{M}\tripslash\check{U} \simeq \check{G} \backslash (\check{M} \times_{\check{\mathfrak{g}}^*} T^*(\check{G}/\check{U})).
    $$
    Here, the map $T^*(\check{G}/\check{U}) \rightarrow \check{\mathfrak{g}}^*$ is the moment map for the left $\check{G}$-action, and $T^*(\check{G}/\check{U})$ is viewed as a $\check{\mathfrak{t}}^*$-scheme via the moment map $T^*(\check{G}/\check{U}) \rightarrow \check{\mathfrak{t}}^*$ for the right $\check{T}$-action. Moreover, $\check{G}$ acts diagonally on the product $\check{M} \times_{\check{\mathfrak{g}}^*} T^*(\check{G}/\check{U})$. 
\end{lemma}
\begin{proof}
    Consider the isomorphism
    $$
    \iota: \check{M} \times_{\check{\mathfrak{g}}^*} (\check{G} \times \check{\mathfrak{u}}^\perp) \simeq \check{M} \times_{\check{\mathfrak{g}}^*} (\check{G} \times \check{\mathfrak{u}}^\perp)
    $$
    given by $\iota(m,g,\chi) = (g^{-1}m, g, \chi)$. On the left, the structure map $\check{G} \times \check{\mathfrak{u}}^\perp \rightarrow \check{\mathfrak{g}}^*$ is the action map, whereas on the right we take the structure map $\check{G} \times \check{\mathfrak{u}}^\perp \rightarrow \check{\mathfrak{g}}^*$ to be the composition of the projection to the second factor followed by the inclusion $\check{\mathfrak{u}}^\perp \hookrightarrow \check{\mathfrak{g}}^*$. 
    
    Let $\check{G}$ act on the left hand side diagonally and on the right hand side by $h \cdot (m,g,\chi)= (m,hg,\chi)$. Then, $\iota$ is $\check{G}$-equivariant. Let $\check{U}$ act on the left hand side by $u \cdot (m, g, \chi) = (m, gu^{-1}, u \cdot \chi)$ (where $u \cdot \chi$ is the coadjoint action) and on the right hand side by $u \cdot (m, g, \chi) = (um, gu^{-1}, u \cdot \chi)$. Then, $\iota$ is $\check{U}$-equivariant. Therefore, we obtain an isomorphism $\overline{\iota}$ on $\check{G} \times \check{U}$-quotients:
    $$
    \overline{\iota}: \check{G} \backslash (\check{M} \times_{\check{\mathfrak{g}}^*} (\check{G} \times \check{\mathfrak{u}}^\perp))/\check{U} \simeq \check{G} \backslash (\check{M} \times_{\check{\mathfrak{g}}^*} (\check{G} \times \check{\mathfrak{u}}^\perp))/\check{U}.
    $$
    On the left hand side, we quotient by $\check{U}$ and then $\check{G}$ to obtain an isomorphism
    $$
    \check{G} \backslash (\check{M} \times_{\check{\mathfrak{g}}^*} (\check{G} \times \check{\mathfrak{u}}^\perp))/\check{U} \simeq \check{G} \backslash (\check{M} \times_{\check{\mathfrak{g}}^*} T^*(\check{G}/\check{U})). 
    $$
    On the right hand side, we quotient by $\check{G}$ and then $\check{U}$ to obtain an isomorphism 
    $$
    \check{G} \backslash (\check{M} \times_{\check{\mathfrak{g}}^*} (\check{G} \times \check{\mathfrak{u}}^\perp))/\check{U} \simeq (\check{M} \times_{\check{\mathfrak{g}}^*} \check{\mathfrak{u}}^\perp)/\check{U} = \check{M}\tripslash\check{U}.  
    $$
    Putting it all together constructs the desired isomorphism. 
\end{proof}

The proof of the following lemma is identical to that of Lemma~\ref{Hamiltonian reduction via the cotangent bundle}. 

\begin{lemma}\label{Kostant-Whittaker reduction via the cotangent bundle}
    There is a canonical $\mathfrak{J}$-equivariant isomorphism of $\check{\mathfrak{c}}$-schemes
    $$
    \check{M}\tripslash(\check{U},\psi) \simeq \check{G} \backslash (\check{M} \times_{\check{\mathfrak{g}}^*} T^*(\check{G}/(\check{U}, \psi))).
    $$
    Here, the map $T^*(\check{G}/(\check{U}, \psi)) \rightarrow \check{\mathfrak{g}}^*$ is the moment map for the left $\check{G}$-action, and the cotangent bundle $T^*(\check{G}/(\check{U},\psi))$ is viewed as a $\check{\mathfrak{c}}$-scheme via the natural projection 
    $$
    T^*(\check{G}/(\check{U},\psi)) = (\check{G} \times (\check{\mathfrak{u}}^\perp + \psi))/\check{U} \rightarrow (\check{\mathfrak{u}}^\perp + \psi)/\check{U} \simeq \check{\mathfrak{c}}. 
    $$
    Moreover, $\check{G}$ acts diagonally on the product $\check{M} \times_{\check{\mathfrak{g}}^*} T^*(\check{G}/(\check{U},\psi))$. 
\end{lemma}

\begin{remark}
    It is not really necessary to use Lemma \ref{Hamiltonian reduction via the cotangent bundle} or Lemma \ref{Kostant-Whittaker reduction via the cotangent bundle} in the following Proposition~\ref{Generic Comparison}. However, we have done so anyway, in order to emphasize the point of view that $T^*(\check{G}/\check{U})$ (resp. $T^*(\check{G}/(\check{U}, \psi)$) is in a sense the ``universal'' example of a Hamilonian (resp. Kostant-Whittaker) reduction of a Hamiltonian $\check{G}$-variety by the action of $\check{U}$. More concretely, we have opted to structure our arguments in a manner which we believe can be adapted to treat the ``quantum version'' of our problem (in which the role of the Poisson algebra $\mathcal{O}(\check{M})$ is taken by a suitable algebra of microlocal differential operators on $\check{M}$). 
\end{remark}

\begin{proposition}\label{Generic Comparison}
    There is a canonical $\check{{T}}_{\check{\mathfrak{t}}^*_\gen}$-equivariant morphism of $\check{\mathfrak{t}}^*$-schemes
    $$
    \eta_{I-\gen}: \check{M}\tripslash(\check{U},\psi) \times_{\check{\mathfrak{c}}} \check{\mathfrak{t}}^* \rightarrow \check{M}\tripslash\check{U} 
    $$
    whose pullback $\eta_{I-\gen} \times_{\check{\mathfrak{t}}^*} \check{\mathfrak{t}}^*_\gen$ to the generic locus yields an isomorphism of $\check{\mathfrak{t}}^*$-schemes
    $$
    \eta_{I-\gen} \times_{\check{\mathfrak{t}}^*} \check{\mathfrak{t}}^*_\gen: (\check{M}\tripslash(\check{U},\psi) \times_{\check{\mathfrak{c}}} \check{\mathfrak{t}}^*) \times_{\check{\mathfrak{t}}^*} \check{\mathfrak{t}}^*_\gen \simeq \check{M}\tripslash\check{U} \times_{\check{\mathfrak{t}}^*} \check{\mathfrak{t}}^*_{\gen}.  
    $$
\end{proposition}
\begin{proof}
    In Lemma~\ref{Hamiltonian reduction via the cotangent bundle} and Lemma~\ref{Kostant-Whittaker reduction via the cotangent bundle}, we have constructed isomorphisms
    \begin{align*}
     \check{M}\tripslash(\check{U},\psi) &\simeq \check{G} \backslash(\check{M} \times_{\check{\mathfrak{g}}^*} T^*(\check{G}/(\check{U},\psi))), \\
     \check{M}\tripslash\check{U} &\simeq \check{G} \backslash (\check{M} \times_{\check{\mathfrak{g}}^*} T^*(\check{G}/\check{U})).
    \end{align*}
    The first isomorphism is $\mathfrak{J}$-equivariant and the second is $\check{T}$-equivariant. Hence, we have a $\check{T}_{\check{\mathfrak{t}}^*_\gen}$-equivariant isomorphism of $\check{\mathfrak{t}}^*_\gen$-schemes 
    \begin{align*}
        (\check{M}\tripslash(\check{U},\psi) \times_{\check{\mathfrak{c}}} \check{\mathfrak{t}}^*) \times_{\check{\mathfrak{t}}^*} \check{\mathfrak{t}}^*_\gen &\simeq \left(\check{G} \backslash \left(\check{M} \times_{\check{\mathfrak{g}}^*} T^*(\check{G}/(\check{U},\psi)) \right) \right) \times_{\check{\mathfrak{t}}^*} \check{\mathfrak{t}}^*_\gen \\
        &\simeq \check{G} \backslash  \left(\check{M} \times_{\check{\mathfrak{g}}^*} T^*(\check{G}/(\check{U},\psi)) \times_{\check{\mathfrak{c}}} \check{\mathfrak{t}}^*_\gen \right)\\
        &\simeq \check{G} \backslash \left(\check{M} \times_{\check{\mathfrak{g}}^*} \left( \left(T^*(\check{G}/(\check{U},\psi)) \times_{\check{\mathfrak{c}}} \check{\mathfrak{t}}^*_\gen \right) \right) \right).
     \end{align*}
     Similarly, we have a $\check{T}_{\check{\mathfrak{t}}^*_\gen}$-equivariant isomorphism of $\check{\mathfrak{t}}^*_\gen$-schemes 
     \begin{align*}
     \check{M}\tripslash\check{U} \times_{\check{\mathfrak{t}}^*} \check{\mathfrak{t}}^*_\gen &\simeq \check{G} \backslash (\check{M} \times_{\check{\mathfrak{g}}^*} T^*(\check{G}/\check{U})) \times_{\check{\mathfrak{t}}^*} \check{\mathfrak{t}}^*_\gen \\
     &\simeq \check{G} \backslash (\check{M} \times_{\check{\mathfrak{g}}^*} (T^*(\check{G}/\check{U}) \times_{\check{\mathfrak{t}}^*} \check{\mathfrak{t}}^*_\gen)).
    \end{align*}
    Therefore, it suffices to exhibit a $\check{G}$-equivariant $ \check{T}_{\check{\mathfrak{t}}^*}$-equivariant morphism of $\check{\mathfrak{g}}^* \times \check{\mathfrak{t}}^*$-schemes 
    $$
    T^*(\check{G}/(\check{U}, \psi)) \times_{\check{\mathfrak{c}}} \check{\mathfrak{t}}^* \rightarrow T^*(\check{G}/\check{U})
    $$
    whose pullback to $\check{\mathfrak{t}}^*_\gen$ is an isomorphism
    $$
    (T^*(\check{G}/(\check{U}, \psi)) \times_{\check{\mathfrak{c}}} \check{\mathfrak{t}}^*) \times_{\check{\mathfrak{t}}^*} \check{\mathfrak{t}}^*_\gen \simeq T^*(\check{G}/\check{U}) \times_{\check{\mathfrak{t}}^*} \check{\mathfrak{t}}^*_\gen. 
    $$
    We did this in Remark~\ref{Cotangent bundle comparison}. 
\end{proof}
For $I \subseteq \Delta$, we define the open subscheme $\check{\mathfrak{t}}^*_{I-\gen}$ of ``$I$-generic'' elements by the non-vanishing of the roots $\check{\alpha} \in \check{\Phi}_I$. Since $\check{\mathfrak{t}}^*_{I-\gen}$ is $W_I$-stable, it descends to an open subscheme $\mathfrak{\check{c}}_{I-\gen} \subseteq \mathfrak{\check{c}}_I$. Note that the map $\check{\mathfrak{t}}^*_{I-\gen} \rightarrow \check{\mathfrak{c}}_I$ is a $W_I$-torsor over $\check{\mathfrak{c}}_{I-\gen}$. Recall that $\mathfrak{J}_I$ denotes the regular centralizer over $\check{\mathfrak{c}}_I$ associated to $\check{L}_I$.  

We have the following generalization of Proposition~\ref{Generic Comparison}. 

\begin{corollary}\label{Generic Comparison Intermediate}
    Let $I \subseteq \Delta$. There is a canonical $\check{T}_{\check{\mathfrak{t}}^*}$-equivariant morphism of $\check{\mathfrak{t}}^*_{I-\gen}$-schemes
    $$
    \eta_{I-\gen}: \check{M}\tripslash(\check{U},\psi_I) \times_{\check{\mathfrak{c}}_I} \check{\mathfrak{t}}^* \rightarrow  \check{M}\tripslash \check{U}
    $$
    whose pullback to $\check{\mathfrak{t}}^*_{I-\gen}$ yields an isomorphism
    $$
    \eta_{I-\gen} \times_{\check{\mathfrak{t}}} \check{\mathfrak{t}}^*_{I-\gen}: \left(\check{M}\tripslash(\check{U},\psi_I) \times_{\check{\mathfrak{c}}_I} \check{\mathfrak{t}}^* \right) \times_{\check{\mathfrak{t}}^*} \check{\mathfrak{t}}^*_{I-\gen} \simeq \check{M}\tripslash\check{U} \times_{\check{\mathfrak{t}}^*} \check{\mathfrak{t}}^*_{I-\gen}. 
    $$
\end{corollary}
\begin{proof}
    In Lemma~\ref{Reduction in Steps}, we produced an isomorphism 
    $$
    \check{M}\tripslash(\check{U},\psi_I) \simeq (\check{M}\tripslash \check{V}_I)\tripslash (\check{U}_I, \overline{\psi}_I). 
    $$
    Applying Proposition~\ref{Generic Comparison} to the Hamiltonian $\check{L}_I$-space $\check{M}\tripslash\check{V}_I$, we obtain a $\check{T}_{\check{\mathfrak{t}}^*}$-equivariant morphism 
    \begin{equation}\label{Generic Comparison Intermediate Eq1}
        \check{M} \tripslash (\check{U}, \psi_I) \times_{\check{\mathfrak{c}}_I} \check{\mathfrak{t}}^* \simeq (\check{M}\tripslash \check{V}_I)\tripslash (\check{U}_I, \overline{\psi}_I) \times_{\check{\mathfrak{c}}_I} \check{\mathfrak{t}}^* \rightarrow (\check{M}\tripslash\check{V}_I) \tripslash \check{U}_I
    \end{equation}
    which induces an isomorphism 
    \begin{equation}\label{Generic Comparison Intermediate Eq2}
         \left(\check{M}\tripslash(\check{U},\psi_I) \times_{\check{\mathfrak{c}}} \check{\mathfrak{t}}^* \right) \times_{\check{\mathfrak{t}}^*} \check{\mathfrak{t}}^*_{I-\gen} \simeq (\check{M}\tripslash\check{V}_I) \tripslash \check{U}_I \times_{\check{\mathfrak{t}}^*} \check{\mathfrak{t}}^*_\gen. 
    \end{equation}
    By Lemma~\ref{Reduction in Steps}, we obtain an isomorphism $(\check{M}\tripslash\check{V}_I)\tripslash\check{U}_I \simeq \check{M}\tripslash\check{U}$. Inserting this isomorphism into \ref{Generic Comparison Intermediate Eq1} and \ref{Generic Comparison Intermediate Eq2}, we obtain the desired morphism 
    $$
    \eta_{I-\gen}: \left( \check{M}\tripslash(\check{U},\psi_I) \times_{\check{\mathfrak{c}}} \check{\mathfrak{t}}^* \right) \times_{\check{\mathfrak{t}}^*} \check{\mathfrak{t}}^*_{I-\gen} \rightarrow \check{M}\tripslash\check{U} \times_{\check{\mathfrak{t}}^*} \check{\mathfrak{t}}^*_{I-\gen}. 
    $$
\end{proof}
\subsection{Anti-generic comparison of Hamiltonian and Kostant-Whittaker reductions} 

We have studied the partial Kostant-Whittaker reduction $\check{M}\tripslash (\check{U},\psi_I)$ over the open locus $\check{\mathfrak{c}}_{I-\gen}$. We'll next restrict to the open subscheme $\check{\mathfrak{c}}_{I-\gen}^{-} \subseteq \check{\mathfrak{c}}_I$ defined by the non-vanishing $\check{\alpha} \neq 0$ for each root $\check{\alpha} \in \check{\Phi} \setminus \check{\Phi}_{I}$. More precisely, we define $\check{\mathfrak{t}}^{*-}_{I-\gen}$ to be the open subscheme of $\check{\mathfrak{t}}^*$ defined by the non-vanishing $\check{\alpha} \neq 0$ for each $\check{\alpha} \in \check{\Phi} \setminus \check{{\Phi}}_I$. Since $\check{{\Phi}} \setminus \check{{\Phi}}_I$ is $W_I$-stable, it follows that $\check{\mathfrak{t}}^{*-}_{I}$ is $W_I$-stable, hence descends to an open subscheme $\check{\mathfrak{c}}_{I-\gen}^- \subseteq \check{\mathfrak{c}}_{I}$. We refer to $\check{\mathfrak{t}}^{*-}_{I-\gen}$ and $\check{\mathfrak{c}}_{I-\gen}^-$ as the \textit{$I$-anti-generic} loci. 

\begin{remark}
    Note that $\check{\mathfrak{t}}_{I-\gen}^{*-} \neq \check{\mathfrak{t}}_{I^\circ-\gen}^*$, where $I^\circ = \Delta \setminus I$: that is, we are going to avoid all of the ``walls'' in the root system of $\check{G}$ associated to roots of $\check{G}$ that are \textit{not} roots of $\check{L}_I$, even if they are not sums of simple roots in $I^\circ$. For example, consider the case $\check{G} = \mathrm{SL}_3$ with its standard Levi subgroup $\check{L}_I = \mathrm{S}(\mathrm{GL}_2 \times \mathbb{G}_m)$. Let $\lambda_1 - \lambda_2$, $\lambda_2 - \lambda_3$, and $\lambda_1 - \lambda_3$ denote the three positive roots. With this notation, $I = \{\lambda_1 - \lambda_2\}$. Therefore, $\check{\mathfrak{t}}_{I-\gen}^{*-} \subseteq \mathbb{A}^2$ is the open locus complementary to the lines $\lambda_2 = \lambda_3$ and $\lambda_1 = \lambda_3$. On the other hand, $\check{\mathfrak{t}}_{I^\circ-\gen}^*$ is the complement of the single line $\lambda_2 = \lambda_3$. 
\end{remark}

\begin{remark}\label{restriction map to levi}
    Consider the affine subspace $\check{\mathfrak{u}}^\perp + \psi_I \subseteq \check{\mathfrak{g}}^*$. It consists of all linear forms on $\check{\mathfrak{g}}$ which restrict to $\psi_I$ on $\check{\mathfrak{u}}$. Since $\psi_I$ vanishes on $\check{\mathfrak{v}}_I$, the same is true of any element of $\check{\mathfrak{u}}^\perp + \psi_I$. Therefore, regarding $(\check{\mathfrak{p}}_I/\check{\mathfrak{v}}_I)^*$ as the space of linear forms on the parabolic subalgebra $\check{\mathfrak{p}}_I$ which vanish on $\check{\mathfrak{v}}_I$, we have a restriction map
    $$
    r_I: \check{\mathfrak{u}}^\perp + \psi_I \rightarrow (\check{\mathfrak{p}}/\check{\mathfrak{v}})^* \simeq \check{\mathfrak{l}}^*. 
    $$
    Its image is exactly $\check{\mathfrak{u}}_I + \overline{\psi}_I$. 
\end{remark}

\begin{lemma}\label{Anti-Generic torsor}
    Let $I \subseteq \Delta$. Regard $\check{\mathfrak{u}}^\perp + \psi_I$ as a $\check{\mathfrak{c}}_I$-scheme via the composition of the restriction map $r_I: \check{\mathfrak{u}}^\perp + \psi_I \rightarrow \check{\mathfrak{l}}_I^*$ (see Remark~\ref{restriction map to levi}) with the characteristic polynomial map $\chi_I: \check{\mathfrak{l}}^*_I \rightarrow \check{\mathfrak{c}}_I$. Then, the natural projection
    \begin{equation}\label{pi2definition}
    \pi_2: (\check{\mathfrak{u}}^\perp + \psi_I) \times_{\check{\mathfrak{c}}_I} \check{\mathfrak{c}}_{I-\gen}^- \rightarrow \check{\mathfrak{c}}_{I-\gen}^-
    \end{equation}
    is a $\check{U}$-torsor. 
\end{lemma}
\begin{proof}
    We can factor the morphism $r_I: \check{\mathfrak{u}}^\perp + \psi_I \rightarrow \check{\mathfrak{l}}_I^*$ as a surjection $\alpha: \check{\mathfrak{u}}^\perp + \psi_I \twoheadrightarrow \check{\mathfrak{u}}_I^\perp + \overline{\psi}_I$ followed by the natural inclusion $\check{\mathfrak{u}}_I^\perp + \overline{\psi}_I \hookrightarrow \check{\mathfrak{l}}_I^*$. Therefore, we have the following commutative diagram.
    $$
    \begin{tikzcd}
        \check{\mathfrak{u}}^\perp + \psi_I \ar[d, "r_I",swap] \ar[r, "\alpha", twoheadrightarrow] & \check{\mathfrak{u}}^\perp_I + \overline{\psi}_I \ar[d, "\chi_I"] \\
        \check{\mathfrak{l}}^*_I \ar[r, "\chi_I"] & \check{\mathfrak{c}}_I
    \end{tikzcd}
    $$    
    We regard $\check{\mathfrak{u}}^\perp_I + \overline{\psi}_I $ as a $\check{\mathfrak{c}}_I$-scheme via $\chi_I: \check{\mathfrak{u}}^\perp_I + \overline{\psi}_I \rightarrow \check{\mathfrak{c}}_I$, so that, by the commutative diagram above, $\alpha$ is a morphism of $\check{\mathfrak{c}}_I$-schemes. We now pull $\alpha$ back to the open subscheme $\check{\mathfrak{c}}_{I-\gen}^- \subseteq \check{\mathfrak{c}}_I$ to obtain a commutative diagram
    $$
    \begin{tikzcd}
        (\check{\mathfrak{u}}^\perp + \psi_I) \times_{\check{\mathfrak{c}}_I} \check{\mathfrak{c}}_{I-\gen}^- \ar[rd, "\pi_2",swap] \ar[rr, "\alpha'"] & &(\check{\mathfrak{u}}_I^\perp + \overline{\psi}_I) \times_{\check{\mathfrak{c}}_I} \check{\mathfrak{c}}_{I-\gen}^- \ar[ld, "\pi_2'"] \\
        & \check{\mathfrak{c}}_{I-\gen}^-.
    \end{tikzcd}
    $$
    Note that $\alpha'$ is $\check{V}_I$-invariant, where $\check{V}_I$ acts through its coadjoint action on the first factor. It therefore suffices to show that (\textit{i}) $\pi_2'$ is a $\check{U}_I = \check{U}/\check{V}_I$-torsor and that (\textit{ii}) $ \alpha'$ is a $\check{V}_I$-torsor. 

    (\textit{i}) The morphism $\pi_2'$ is a base change of the characteristic polynomial $\chi_I: \check{\mathfrak{u}}_I + \overline{\psi}_I \rightarrow \check{\mathfrak{c}}_{I}$, which is a $\check{U}_I$-torsor by Kostant's theorem \cite[Theorem 3.2.2]{Ri17} applied to $\check{\mathfrak{l}}_I$.

    (\textit{ii}) Recall that $\check{\mathfrak{v}}_I^\perp$ denotes the space of linear forms on $\check{\mathfrak{g}}$ vanishing on $\check{\mathfrak{v}}_I = \mathrm{Lie}(\check{V}_I)$. We have a Cartesian diagram
    $$
    \begin{tikzcd}
        (\check{\mathfrak{u}}^\perp + \psi_I) \times_{\check{\mathfrak{c}}_I} \check{\mathfrak{c}}_{I-\gen}^- \ar[r, "\alpha'"] \ar[d, hookrightarrow] &(\check{\mathfrak{u}}_I^\perp + \overline{\psi}_I) \times_{\check{\mathfrak{c}}_I} \check{\mathfrak{c}}_{I-\gen}^- \ar[d,hookrightarrow] \\
        \check{\mathfrak{v}}_I^\perp \times_{\check{\mathfrak{c}}_I} \check{\mathfrak{c}}_{I-\gen}^- \ar[r,"r_I",swap] & \check{\mathfrak{l}}_I^* \times_{\check{\mathfrak{c}}_I} \check{\mathfrak{c}}_{I-\gen}^-.
    \end{tikzcd}
    $$
    Therefore, it suffices to show that the bottom arrow is a $\check{V}_I$-torsor. The $\check{G}$-invariant isomorphisms $\check{\mathfrak{g}}^* \simeq \check{\mathfrak{g}}$ and $\check{\mathfrak{l}}_I^* \simeq \check{\mathfrak{l}}_I$ induce a commutative diagram
    $$
    \begin{tikzcd}
        \check{\mathfrak{v}}_I^\perp \times_{\check{\mathfrak{c}}_I} \check{\mathfrak{c}}_{I-\gen}^-  \ar[d,"\sim"]
        \ar[r,"r_I"] & \check{\mathfrak{l}}_I^* \times_{\check{\mathfrak{c}}_I} \check{\mathfrak{c}}_{I-\gen}^- \ar[d, "\sim"]\\
        \check{\mathfrak{p}}_I \times_{\check{\mathfrak{c}}_I} \check{\mathfrak{c}}_{I,I-\gen}^-  \ar[r,"p_I",swap] & \check{\mathfrak{l}}_I \times_{\check{\mathfrak{c}}_I} \check{\mathfrak{c}}_{I-\gen}^-,
    \end{tikzcd}
    $$
    where $p_I$ is induced by the projection $\check{\mathfrak{p}}_I \twoheadrightarrow \check{\mathfrak{l}}_I$. To show that $p_I$ is a $\check{V}_I$-torsor, we will show that the action map 
    $$
    a: (\check{V}_I \times \check{\mathfrak{l}}_I) \times_{\check{\mathfrak{c}}_I} \check{\mathfrak{c}}_{I-\gen}^- \rightarrow \check{\mathfrak{p}}_I \times_{\check{\mathfrak{c}}_I} \check{\mathfrak{c}}_{I-\gen}^-
    $$
    is an isomorphism. We claim that $a$ is \'etale. Since $a$ is a morphism of smooth varieties, it suffices to check that for any $(v,x) \in (\check{V}_I \times \check{\mathfrak{l}}_I) \times_{\check{\mathfrak{c}}_I} \check{\mathfrak{c}}_{I-\gen}^-$, 
     $$
     da_{(v,x)}: T_{(v,x)}((\check{V}_I \times \check{\mathfrak{l}}_I) \times_{\check{\mathfrak{c}}_I} \check{\mathfrak{c}}_{I-\gen}^-) \rightarrow T_{v \cdot x}(\check{\mathfrak{p}}_I \times_{\check{\mathfrak{c}_I}} \check{\mathfrak{c}}_{I-\gen}^-)
     $$
     is an isomorphism. Since $a$ is $\check{V}_I$-equivariant, it suffices to check the claim when $v = 1$. At such a point, the differential
     $$
     da_{(1,x)}: \check{\mathfrak{v}}_I \oplus \check{\mathfrak{l}}_I \rightarrow \check{\mathfrak{p}}_I
     $$
     is given by
     $$
     da_{(1,x)}(\xi, \zeta) = [\xi, x] + \zeta. 
     $$
     Since the domain and codomain of $da_{(1,x)}$ have the same dimension, it suffices to show surjectivity. The image of $da_{(1,x)}$ is $\mathrm{ad}_x(\check{\mathfrak{v}}_I) + \check{\mathfrak{l}}_I \subseteq \check{\mathfrak{p}}_I$. Since $\check{\mathfrak{p}}_I = \check{\mathfrak{v}}_I + \check{\mathfrak{l}}_I$, it suffices to show that 
     $$
     \mathrm{ad}_x: \check{\mathfrak{v}}_I \rightarrow \check{\mathfrak{v}}_I
     $$
     is surjective. It therefore suffices to demonstrate the injectivity of $\mathrm{ad}_x$. Let $\xi \in \check{\mathfrak{v}}_I$ and suppose that $[\xi, x] = 0$. Let $x = x_{n} + x_{{ss}}$ denote the Jordan decomposition of $x$. We have $[\xi, x_{{ss}}] = [\xi, x_{{n}}] = 0$. Since $x \in \check{\mathfrak{l}}_I$, it follows that $x_{{ss}}, x_{{n}} \in \check{\mathfrak{l}}_I$. Note that $x_{{ss}}$ has the same image in $\check{\mathfrak{c}}_I$ as $x$, so it follows that $x_{{ss}} \in \check{\mathfrak{p}}_I \times_{\check{\mathfrak{c}}_I} \check{\mathfrak{c}}_{I-\gen}^-$ as well. Since $x_{{ss}}$ is semisimple, $\mathrm{ad}_{x_{{ss}}}$ is a semisimple endomorphism of $\check{\mathfrak{v}}_I$. Furthermore, the fact that $x_{{ss}} \in \check{\mathfrak{p}}_I \times_{\check{\mathfrak{c}}_I} \check{\mathfrak{c}}_{I-\gen}^-$ implies that $x_{{ss}}$ acts on $\check{\mathfrak{v}}_I$ with non-zero eigenvalues (the weights of $\check{\mathfrak{v}}_I$ are contained in ${\Phi} \setminus {\Phi}_I$). Hence, the centralizer of $x_{{ss}}$ in $\check{\mathfrak{v}}_I$ is trivial (it is the zero eigenspace of $\mathrm{ad}_{x_{{ss}}}$), and it follows that $\xi = 0$.

     Next, we assert that the action map $a$ is surjective. We must show that any element of $\check{\mathfrak{p}}_I \times_{\check{\mathfrak{c}}_I} \check{\mathfrak{c}}_{I-\gen}^-$ is conjugate under $\check{V}_I$ to an element of $\check{\mathfrak{l}}_I \times_{\check{\mathfrak{c}}_I} \check{\mathfrak{c}}_{I-\gen}^-$. Since $\check{\mathfrak{l}}_I \times_{\check{\mathfrak{c}}_I} \check{\mathfrak{c}}_{I-\gen}^-$ is $\check{L}_I$-stable and $\check{P}_I = \check{V}_I \rtimes \check{L}_I$, it suffices to show that any element of $\check{\mathfrak{p}}_I \times_{\check{\mathfrak{c}}_I} \check{\mathfrak{c}}_{I-\gen}^-$ is conjugate under $\check{P}_I$ to an element of $\check{\mathfrak{l}}_I$. Let $x \in \check{\mathfrak{p}}_I \times_{\check{\mathfrak{c}}_I} \check{\mathfrak{c}}_{I-\gen}^-$. Let $x = x_{{n}} + x_{{ss}}$ denote the Jordan decomposition. Note that $x_{ss} \in \check{\mathfrak{p}}_I \times_{\check{\mathfrak{c}}_I} \check{\mathfrak{c}}_{I-\gen}^-$. After conjugating by $\check{P}_I$, we can assume that $x_{ss} \in \check{\mathfrak{t}}$ and in particular $x_{ss} \in \check{\mathfrak{l}}_I$. The zero eigenspace of $\mathrm{ad}_{x_{ss}}$ on $\check{\mathfrak{u}}$ is the sum of the positive root spaces $\check{\mathfrak{g}}_\alpha$ such that $\alpha(x_{ss}) = 0$. Since $\alpha(x_{ss}) \neq 0$ for every $\alpha \not\in \Phi_{I}$, we deduce that the centralizer of $x_{ss}$ in $\check{\mathfrak{u}}$ is contained in the sum of the positive root spaces $\check{\mathfrak{g}}_\alpha$ such that $\alpha \in {\Phi}_I$, i.e. the subspace $\check{\mathfrak{u}}_I$. Since $[x_n, x_{ss}] = 0$, we conclude that $x_n \in \check{\mathfrak{u}}_I \subseteq \check{\mathfrak{l}}_I$. Therefore, $x \in \check{\mathfrak{l}}_I$ and we are done.

     Therefore, $a$ is an \'etale cover. To show that $a$ is an isomorphism, it suffices to show that it has connected (closed) fibers. Let $x \in \check{\mathfrak{p}}_I$. The fiber $a^{-1}(x)$ is equipped with a natural projection $q_x: a^{-1}(x) \rightarrow \check{\mathfrak{l}}_I$. Since $\check{\mathfrak{l}}_I$ is connected, it suffices to show that the fiber of $q_x$ over any element $y \in \check{\mathfrak{l}}_I$ is connected. Let $v \in \check{V}_I$ denote an element satisfying $a(v,y) = x$ (which exists by the surjectivity of $a$). We have $q_x^{-1}(y) \simeq v\mathrm{Stab}_{\check{V}_I}(y)$. Since $\mathrm{Stab}_{\check{V}_I}(y)$ is a closed subgroup of $\check{V}_I$, it is unipotent, therefore connected, as needed. 
\end{proof}

\begin{construction}\label{partialkostantslice}
    In light of Lemma~\ref{Anti-Generic torsor}, it is natural to introduce a ``partial Kostant section'' 
    $$
    \sigma_I: \check{\mathfrak{c}}_{I} \rightarrow \check{\mathfrak{g}}^*
    $$
    which trivializes the torsor $\pi_2$ of \eqref{pi2definition}. More precisely, we want $\sigma_I$ to factor through $\check{\mathfrak{u}}^\perp + \psi_I$ and for its pullback
    $$
    \sigma_I \times_{\check{\mathfrak{c}}_I} \check{\mathfrak{c}}_{I-\gen}^-: \check{\mathfrak{c}}_{I-\gen}^- \rightarrow (\check{\mathfrak{u}}^\perp + \psi_I)  \times_{\check{\mathfrak{c}}_I} \check{\mathfrak{c}}_{I-\gen}^-
    $$ 
    to be a section to $\pi_2$. Let $\kappa_I: \check{\mathfrak{c}}_I \rightarrow \check{\mathfrak{l}}_I^*$ denote the Kostant section associated to $\check{\mathfrak{l}}_I$. We define $\sigma_I$ to be the composition of $\kappa_I$ with the inclusion $\check{\mathfrak{l}}_I^* \hookrightarrow \check{\mathfrak{g}}^*$ given by pulling back a linear form on $\check{\mathfrak{l}}_I$ to $\check{\mathfrak{p}}_I$ and extending it by zero on $\check{\mathfrak{v}}_I^-$. Note that when $I = \emptyset$, $\sigma_I: \check{\mathfrak{t}}^* \hookrightarrow \check{\mathfrak{g}}^*$ is just the inclusion of $\check{\mathfrak{t}}^*$ into $\check{\mathfrak{g}}^*$ that has already appeared, and when $I = \Delta$, $\sigma_I = \kappa_I = \kappa$ is just the usual Kostant section $\kappa: \check{\mathfrak{c}} \rightarrow \check{\mathfrak{g}}^*$. 
\end{construction}

\begin{lemma}\label{Reduction by Restriction III}
    There is a canonical $\mathfrak{J} \times_{\check{\mathfrak{c}}_I} \check{\mathfrak{c}}_{I-\gen}^-$-equivariant isomorphism of $\check{\mathfrak{c}}_{I-\gen}^-$-schemes
    $$
    \check{M}\tripslash(\check{U},\psi_I) \times_{\check{\mathfrak{c}}_I} \check{\mathfrak{c}}^-_{I-\gen}  \simeq \check{M} \times_{\check{\mathfrak{g}}^*,\sigma_I} \check{\mathfrak{c}}_{I-\gen}^-.  
    $$
\end{lemma}
\begin{proof}
    Let $p: \check{M} \times_{\check{\mathfrak{g}}^*} (\check{\mathfrak{u}}^\perp + \psi_I) \rightarrow \check{M}\tripslash(\check{U},\psi_I)$ denote the quotient map. It is a $\check{U}$-torsor. Now, consider the commutative diagram
    $$
    \begin{tikzcd}
        \check{M} \times_{\check{\mathfrak{g}}^*,\sigma_I} \check{\mathfrak{c}}_{I-\gen}^- \ar[r] \ar[d,"\mathrm{pr}_2"] &\check{M} \times_{\check{\mathfrak{g}}^*} (\check{\mathfrak{u}}^\perp + \psi_I) \times_{\check{\mathfrak{c}}_I} \check{\mathfrak{c}}_{I-\gen}^- \ar[r,"p"] \ar[d,"\mathrm{pr}_2"] & \check{M}\tripslash(\check{U},\psi_I) \times_{\check{\mathfrak{c}}_I} \check{\mathfrak{c}}_{I-\gen}^- \ar[d, "\mathrm{pr}_2"] \\
        \check{\mathfrak{c}}_{I-\gen}^- \ar[r, "\sigma_I"] \ar[rr, "\mathrm{id}", bend right = 10,swap] & (\check{\mathfrak{u}}^\perp + \psi_I) \times_{\check{\mathfrak{c}}_I} \check{\mathfrak{c}}_{I-\gen}^- \ar[r,"\pi_2"] &\check{\mathfrak{c}}_{I-\gen}^-.
    \end{tikzcd}
    $$
    The square on the left is Cartesian by construction. By Lemma~\ref{Anti-Generic torsor}, the square on the right is a morphism of $\check{U}$-torsors, hence is Cartesian. Therefore, the composite square is Cartesian. It follows that the composition of the top row is an isomorphism
    $$
    \check{M} \times_{\check{\mathfrak{g}}^*,\sigma_I} \check{\mathfrak{c}}_{I-\gen}^- \xrightarrow{\sim} \check{M}\tripslash(\check{U},\psi_I) \times_{\check{\mathfrak{c}}_I} \check{\mathfrak{c}}_{I-\gen}^-. 
    $$
\end{proof}

\begin{example}\label{Partial cotangent bundle comparison}
    We can apply Lemma~\ref{Reduction by Restriction III} to the example of $\check{M} = T^*\check{G}$ to obtain a natural isomorphism 
    \begin{equation}\label{Partial cotangent bundle comparison eq1}
        T^*(\check{G}/(\check{U},\psi_I)) \times_{\check{\mathfrak{c}}_I} \check{\mathfrak{c}}_I^- \simeq \check{G} \times \check{\mathfrak{c}}^-_{I-\gen}. 
    \end{equation}
    On the other hand, Example~\ref{universal kostant-whittaker reduction} provides an isomorphism 
    \begin{equation}\label{Partial cotangent bundle comparison eq2}
        T^*(\check{G}/(\check{U},\psi)) \simeq \check{G} \times \check{\mathfrak{c}}. 
    \end{equation}
    Combining \eqref{Partial cotangent bundle comparison eq1} with \eqref{Partial cotangent bundle comparison eq2}, we deduce an isomorphism
    \begin{align}\label{Partial cotangent bundle comparison eq3}
        T^*(\check{G}/(\check{U},\psi_I)) \times_{\check{\mathfrak{c}}_I} \check{\mathfrak{c}}_{I-\gen}^- &\simeq \check{G} \times \check{\mathfrak{c}}^-_{I-\gen} \\
        &\simeq (\check{G} \times \check{\mathfrak{c}}) \times_{\check{\mathfrak{c}}} \check{\mathfrak{c}}_{I-\gen}^- \nonumber \\
        &\simeq T^*(\check{G}/(\check{U},\psi)) \times_{\check{\mathfrak{c}}} \check{\mathfrak{c}}^-_{I-\gen}. \nonumber
    \end{align}
    In fact, we can be slightly more precise. We have a canonical map of $\check{\mathfrak{c}}_I$-varieties
    $$
    T^*(\check{G}/(\check{U},\psi)) \times_{\check{\mathfrak{c}}} \check{\mathfrak{c}}_{I} \simeq \check{G} \times \check{\mathfrak{c}}_I \operatornamewithlimits{\hookrightarrow}^{\mathrm{id} \times \sigma_I} \check{G} \times (\check{\mathfrak{u}}^\perp + \psi_I) \twoheadrightarrow (\check{G} \times (\check{\mathfrak{u}}^\perp + \psi_I))/\check{U} \simeq T^*(\check{G}/(\check{U},\psi_I))
    $$
    which pulls back over the anti-generic locus $\check{\mathfrak{c}}^{-}_{I-\gen}$ to the isomorphism \eqref{Partial cotangent bundle comparison eq3}. 
\end{example}

In order to generalize Remark~\ref{Partial cotangent bundle comparison} to arbitrary $\check{M}$, we note the following easy analog of Lemma~\ref{Hamiltonian reduction via the cotangent bundle} and Lemma~\ref{Kostant-Whittaker reduction via the cotangent bundle}.

\begin{lemma}\label{Partial Kostant-Whittaker reduction via the cotangent bundle}
    There is a canonical $\mathfrak{J}_I$-equivariant isomorphism of $\check{\mathfrak{c}}_I$-schemes
    $$
    \check{M}\tripslash(\check{U},\psi_I) \simeq \check{G} \backslash (\check{M} \times_{\check{\mathfrak{g}}^*} T^*(\check{G}/(\check{U}, \psi_I))).
    $$
    Here, the map $T^*(\check{G}/(\check{U}, \psi_I)) \rightarrow \check{\mathfrak{g}}^*$ is the moment map for the left $\check{G}$-action, and the cotangent bundle $T^*(\check{G}/(\check{U},\psi_I))$ is viewed as a $\check{\mathfrak{c}}_I$-scheme via the natural projection 
    $$
    T^*(\check{G}/(\check{U},\psi_I)) = (\check{G} \times (\check{\mathfrak{u}}^\perp + \psi_I))/\check{U} \rightarrow (\check{\mathfrak{u}}^\perp + \psi_I)/\check{U} \rightarrow \check{\mathfrak{c}}. 
    $$
    Moreover, $\check{G}$ acts diagonally on the product $\check{M} \times_{\check{\mathfrak{g}}^*} T^*(\check{G}/(\check{U},\psi_I))$. 
\end{lemma}

Now we can formulate our general description of a partial Kostant-Whittaker reduction $\check{M}\tripslash(\check{U},\psi_I)$ over the anti-generic locus of $\check{\mathfrak{c}}_I$. 

\begin{proposition}\label{Anti-Generic comparison}
    There is a canonical $\mathfrak{J}_I$-equivariant morphism of $\check{\mathfrak{c}}_{I}$-schemes
    $$
    \eta^-_{I-\gen}: \check{M}\tripslash(\check{U},\psi) \times_{\check{\mathfrak{c}}} \check{\mathfrak{c}}_I \rightarrow \check{M}\tripslash(\check{U},\psi_I) 
    $$
    which pulls back to an isomorphism of $\check{\mathfrak{c}}^{-}_{I-\gen}$-schemes
    $$
    \eta^-_{I-\gen} \times_{\check{\mathfrak{c}}_I} \check{\mathfrak{c}}^-_{I-\gen}: \left( \check{M}\tripslash(\check{U},\psi) \times_{\check{\mathfrak{c}}} \check{\mathfrak{c}}_I \right) \times_{\check{\mathfrak{c}}_I} \check{\mathfrak{c}}_{I-\gen}^- \simeq \check{M}\tripslash(\check{U},\psi_I) \times_{\check{\mathfrak{c}}_I} \check{\mathfrak{c}}^-_{I-\gen} 
    $$
\end{proposition}
\begin{proof}
    The construction of $\eta_{I-\gen}^-$ is easily adapted from that of $\eta_{I-\gen}$ in the proof of Proposition~\ref{Generic Comparison}, using Lemma~\ref{Partial cotangent bundle comparison} in place of Lemma~\ref{Cotangent bundle comparison} and Lemma~\ref{Partial Kostant-Whittaker reduction via the cotangent bundle} in place of Lemma~\ref{Hamiltonian reduction via the cotangent bundle}.
\end{proof}

We omit the proof of the following straightforward lemma. 

\begin{lemma}\label{Anti-Generic comparison, equivariant}
    Suppose that $\check{M}$ is equipped with an action of an algebraic group $\check{H}$ commuting with the action of $\check{G}$. Then, the morphism $\eta_{I-\gen}^-$ of Proposition~\ref{Anti-Generic comparison} is $\check{H}$-equivariant.
\end{lemma}

\subsection{Affine closure}\label{affineclosure}

We will now apply Proposition~\ref{Generic Comparison} and Proposition~\ref{Anti-Generic comparison} to obtain a description of the ring $\mathcal{O}(\check{M}\tripslash(\check{U}, \psi_I))$ in terms of the Hamiltonian reduction $\mathcal{O}(\check{M} \tripslash \check{U})$ and the non-degenerate Kostant-Whittaker reduction $\mathcal{O}(\check{M} \tripslash (\check{U},\psi))$. Consider the union $\check{\mathfrak{c}}_I^\circ := \check{\mathfrak{c}}_{I-\gen} \cup \check{\mathfrak{c}}_{I-\gen}^-$ of the $I$-generic and $I$-anti-generic loci in $\check{\mathfrak{c}}_I$. Let 
$$
\pi_I: \check{M}\tripslash (\check{U},\psi_I) \rightarrow \check{\mathfrak{c}}_I
$$
denote the natural projection induced by the $\check{V}_I$-reduced moment map $\check{M}\tripslash \check{V}_I \rightarrow \check{\mathfrak{l}}^*_I\doubslash\check{L} = \check{\mathfrak{c}}_I$. 

\begin{lemma}\label{Spectral Hartogs}
    The restriction map 
    \begin{equation}\label{Spectral Hartogs eq1}
        \mathcal{O}(\check{M}\tripslash(\check{U},\psi_I)) \rightarrow \mathcal{O}(\pi_I^{-1}(\check{\mathfrak{c}}_I^\circ))
    \end{equation}
    restricting regular functions on $\check{M}\tripslash (\check{U},\psi_I)$ to the open subscheme $\pi_I^{-1}(\check{\mathfrak{c}}_I^\circ) \subseteq \check{M} \tripslash (\check{U}, \psi_I)$ is an isomorphism. 
\end{lemma}
\begin{proof}
    We have assumed that $\check{M}$ is a free $\check{G}$-variety; thus, the reduction $\check{M}\tripslash\check{V}_I$ is a free $\check{L}_I$-variety. Therefore, the moment map $\mu_I: \check{M}\tripslash \check{V}_I \rightarrow \check{\mathfrak{l}}_I^*$ is smooth. By Lemma~\ref{Reduction by Restriction}, the projection $\pi_I$ is a base change of $\mu_I$, hence is smooth. In particular, $\pi_I$ is flat.
    
    The closed subscheme $\check{\mathfrak{c}}_I \setminus \check{\mathfrak{c}}_I^\circ$ has codimension 2 in $\check{\mathfrak{c}}_I$. Indeed, 
    $$
    \check{\mathfrak{c}}_I \setminus \check{\mathfrak{c}}_I^\circ = \bigcup_{\substack{\alpha \in I \\ \beta \not\in I}} \{\alpha = 0\} \cap \{\beta = 0\}. 
    $$
    Since $\alpha \neq \beta$ in any of the terms of this union, we conclude that the linear subspace $\{\alpha = 0\} \cap \{\beta = 0\} = \{\alpha,\beta = 0\}$ has codimension $2$ in $\check{\mathfrak{c}}_I$. 
    
    It then follows from the flatness of $\pi_I$ that the preimage $\pi_I^{-1}(\check{\mathfrak{c}}_I^\circ)$ of $\check{\mathfrak{c}}_I \setminus \check{\mathfrak{c}}_I^\circ$ has codimension $\geq 2$ in $\check{M}\tripslash(\check{U}, \psi_I)$. Moreover, $\check{M}\tripslash(\check{U},\psi_I)$ is smooth over the affine space $\check{\mathfrak{c}}_I$ and is therefore regular (in particular, normal). By Hartog's principle (say, in the form of \cite[\href{https://stacks.math.columbia.edu/tag/0BCS}{Tag 0BCS}]{stacks-project}), we conclude that the restriction map \eqref{Spectral Hartogs eq1} is an isomorphism. 
\end{proof}

\begin{construction}\label{descriptionofaffineclosure}
    Since $\check{\mathfrak{c}}_I^\circ = \check{\mathfrak{c}}_{I-\gen} \cup \check{\mathfrak{c}}_{I-\gen}^-$, Lemma~\ref{Spectral Hartogs} yields an isomorphism 
    $$
    \mathcal{O}(\check{M}\tripslash(\check{U},\psi_I)) \xrightarrow{\sim} \mathcal{O}(\pi_I^{-1}(\check{\mathfrak{c}}_{I-\gen})) \times_{\mathcal{O}(\pi_I^{-1}(\check{\mathfrak{c}}_{I-\gen}^- \cap \check{\mathfrak{c}}_{I-\gen}))} \mathcal{O}(\pi_I^{-1}(\check{\mathfrak{c}}_{I-\gen}^-)).
    $$
    Recall that $f_I \in \mathcal{O}(\check{\mathfrak{t}}^*)$ (respectively, $g_I \in \mathcal{O}(\check{\mathfrak{t}}^*)$) denotes the product of the roots in $\check{\Phi}_I$ (respectively, in $\check{\Phi} \setminus \check{\Phi}_I$). Since $f_I$ and $g_I$ are $W_I$-invariant, they descend to functions on $\check{\mathfrak{c}}_I$. Moreover, the non-vanishing locus $D(f_I)$ (resp. $D(g_I)$) in $\check{\mathfrak{c}}_I$ is exactly $\check{\mathfrak{c}}_{I-\gen}$ (resp. $\check{\mathfrak{c}}_{I-\gen}^-$). Since $\check{\mathfrak{c}}_I$ is affine, restriction therefore induces isomorphisms
    $$
    \mathcal{O}(\check{M}\tripslash(\check{U},\psi_I))_{f_I} \xrightarrow{\sim} \mathcal{O}(\pi_I^{-1}(\check{\mathfrak{c}}_{I-\gen}))
    $$
    $$
    \mathcal{O}(\check{M}\tripslash(\check{U},\psi_I))_{g_I} \xrightarrow{\sim} \mathcal{O}(\pi_I^{-1}(\check{\mathfrak{c}}^-_{I-\gen}))
    $$
    $$
    \mathcal{O}(\check{M}\tripslash(\check{U},\psi_I))_{f_I g_I} \xrightarrow{\sim} \mathcal{O}(\pi_I^{-1}(\check{\mathfrak{c}}_{I-\gen} \cap \check{\mathfrak{c}}_{I-\gen}^-)). 
    $$
    Now we can bring in Corollary~\ref{Generic Comparison Intermediate} and Proposition~\ref{Anti-Generic comparison} to obtain isomorphisms
    $$
    \eta_{I-\gen}^*: \mathcal{O}(\pi_I^{-1}(\check{\mathfrak{c}}_{I-\gen})) \simeq \mathcal{O}(\check{M}\tripslash \check{U} \times_{\check{\mathfrak{t}}^*} \check{\mathfrak{t}}^*_{I-\gen}) \simeq  \mathcal{O}(\check{M}\tripslash \check{U})_{f_I}
    $$
    $$
    (\eta_{I-\gen}^-)^*: \mathcal{O}(\pi_I^{-1}(\check{\mathfrak{c}}^-_{I-\gen})) \simeq \mathcal{O}(\check{M}\tripslash \check{U} \times_{\check{\mathfrak{t}}^*} \check{\mathfrak{t}}^{-*}_{I-\gen}) \simeq  \mathcal{O}(\check{M}\tripslash (\check{U},\psi))_{g_I}
    $$
    In summary, we have a canonical isomorphism
    $$
    \eta_I: \mathcal{O}(\check{M}\tripslash (\check{U},\psi_I)) \simeq \mathcal{O}(\check{M}\tripslash \check{U})_{f_I} \times_{\mathcal{O}(\check{M}\tripslash (\check{U},\psi))_{f_Ig_I}} \mathcal{O}(\check{M}\tripslash (\check{U},\psi))_{g_I}.
    $$
    Here, the map $\mathcal{O}(\check{M}\tripslash \check{U})_{f_I} \rightarrow \mathcal{O}(\check{M}\tripslash (\check{U},\psi))_{f_Ig_I}$ is given by the natural inclusion 
    $$
    \mathcal{O}(\check{M}\tripslash \check{U})_{f_I} \hookrightarrow \mathcal{O}(\check{M}\tripslash \check{U})_{f_I g_I} 
    $$
    followed by the identification
    $$
    \mathcal{O}(\check{M}\tripslash \check{U})_{f_I g_I} \simeq \mathcal{O}(\check{M}\tripslash (\check{U},\psi))_{f_I g_I}
    $$
    of Proposition~\ref{Generic Comparison} (the ``absolute'' version of Corollary~\ref{Generic Comparison Intermediate} in which \textit{all} roots are inverted). 
\end{construction}

\section{Comparison}\label{comparison}

\subsection{Parabolic restriction} We will now complete our discussion of the parabolic restriction functors $\mathrm{Res}_I^\natural: D_{G(\mathcal{O})}(\Gr_G) \rightarrow D_{L_I(\mathcal{O})}(\Gr_I)$ that was initiated in §§\ref{parabolic restriction}. As we recalled in Remark~\ref{normalized parabolic restriction}, the functor $\mathrm{Res}_I^\natural$ is $t$-exact and therefore induces an exact functor 
$$
\mathrm{Res}_I^\natural: \mathcal{P}_{G(\mathcal{O})}(\Gr_G) \rightarrow \mathcal{P}_{L_I(\mathcal{O})}(\Gr_I)
$$
on abelian categories of perverse sheaves. Moreover, we recalled in Construction~\ref{parabolic restriction monoidal structure} the construction of a monoidal structure on $\mathrm{Res}_I^\natural$. 

On the other hand, the inclusion $\check{L}_I \subseteq \check{G}$ of the canonical dual Levi subgroup induces an exact restriction functor 
$$
\operatorname{Res}^{\check{G}}_{\check{L}_I}: \operatorname{Rep}(\check{G}) \rightarrow \operatorname{Rep}(\check{L}_I). 
$$
Beilinson and Drinfeld \cite[§5.3.29]{beilinsondrinfeld} have established the following compatibility between $\mathrm{Res}_I^\natural$ and $\operatorname{Res}^{\check{G}}_{\check{L}_I}$ under the geometric Satake equivalence. 
\begin{theorem}[Beilinson-Drinfeld]\label{parabolic restriction and Satake}
    Let 
    \begin{align*}
    &\mathbb{S}_G \, : \mathcal{P}_{G(\mathcal{O})}(\Gr_G) \rightarrow \operatorname{Rep}(\check{G}) \\
    &\mathbb{S}_{L_I}: \mathcal{P}_{L_I(\mathcal{O})}(\Gr_I) \rightarrow \operatorname{Rep}(\check{L}_I) 
    \end{align*}
    denote the respective geometric Satake transforms of Mirkovi\'{c}-Vilonen \cite{MV04}. Then, the following diagram of abelian categories and exact tensor functors commutes up to a (canonical) natural isomorphism of tensor functors: 
    $$
    \begin{tikzcd}
        \mathcal{P}_{G(\mathcal{O})}(\Gr_G) \ar[r, "\mathbb{S}_G"] \ar[d, "\mathrm{Res}_I^\natural"] & \operatorname{Rep}(\check{G}) \ar[d, "\operatorname{Res}^{\check{G}}_{\check{L}_I}"]\\
        \mathcal{P}_{L_I(\mathcal{O})}(\Gr_I) \ar[r, "\mathbb{S}_{L_I}"] & \operatorname{Rep}(\check{L}_I)
    \end{tikzcd}
    $$
\end{theorem}
We refer the reader to \cite[Proposition 15.3]{baumann2018notes} for a proof. 

\begin{remark}
    It is our point of view that Theorem~\ref{parabolic restriction and Satake} is the (Tannakian) \textit{definition} of the Langlands dual Levi subgroup $\check{L}_I \subseteq \check{G}$.
\end{remark}

\subsection{The Theorem of Ginzburg-Riche}\label{ginzburgrichesubsection} It is now time to bring in our main tool, the results of Ginzburg and Riche in \cite{GR13}. Recall from \eqref{ungraded anti-generic comparison automorphic side} that we have produced an injective graded $R_\emptyset$-module homomorphism  
$$
\xi_\emptyset^\natural: H^*_{T(\mathcal{O})}(\Gr_T, i_\emptyset^{!,\natural} \mathcal{F}_\reg) \rightarrow H_{T(\mathcal{O})}^*(\Gr_T, \operatorname{Res}^\natural_\emptyset(\mathcal{F}_\reg)). 
$$
It is a $\check{G}$-equivariant (obvious) $\check{T}$-equivariant (Remark~\ref{anti-generic comparison automorphic side and the regular centrlizer}) ring homomorphism (Remark~\ref{xiIisaringhomomorphism}). Moreover, we can invoke Theorem~\ref{parabolic restriction and Satake} (in the case $L = T$, which is really a direct consequence of the Mirkovi\'{c}-Vilonen proof \cite{MV04} of geometric Satake) in combination with \cite[Lemma 2.2]{yun2009integral} to obtain a canonical $\check{G}$-equivariant $\check{T}$-equivariant $R_\emptyset$-algebra isomorphism 
\begin{equation}\label{ginzburgricheparabolicrestriction}
H_{T(\mathcal{O})}^*(\Gr_T, \operatorname{Res}^\natural_\emptyset(\mathcal{F}_\reg)) \simeq H_{T(\mathcal{O})}^* \left(\Gr_T, \mathbb{S}_{\check{T}}^{-1}\left(\mathrm{Res}^{\check{G}}_{\check{T}}\left( \mathcal{O}(\check{G}) \right)\right)\right)\simeq \mathcal{O}(\check{G}) \otimes R_\emptyset. 
\end{equation}
Via this natural identification, we may regard $\xi_\emptyset^\natural$ as a homomorphism 
$$
\xi_\emptyset^\natural: H^*_{T(\mathcal{O})}(\Gr_T, i_\emptyset^{!,\natural} \mathcal{F}_\reg) \rightarrow \mathcal{O}(\check{G}) \otimes R_\emptyset. 
$$
On the spectral side, we consider the algebra of functions $\mathcal{O}(T^*(\check{G}/\check{U}))$ on the cotangent bundle of the basic affine space. It is $\check{G}$-module via the natural left action of $\check{G}$ on $T^*(\check{G}/\check{U})$ and a $\check{T}$-module via the natural right action of $\check{T} \simeq \check{B}/\check{U}$ on $T^*(\check{G}/\check{U})$. We have a $\check{G} \times \check{T}$-equivariant morphism 
$$
\check{G} \times \check{\mathfrak{t}}^* \hookrightarrow \check{G} \times \left( \check{\mathfrak{g}}/\check{\mathfrak{u}} \right)^* \twoheadrightarrow \left(\check{G} \times \left(\check{\mathfrak{g}}/\check{\mathfrak{u}}\right)^* \right)/\check{U} = T^*(\check{G}/\check{U})
$$
Here, the embedding $\check{\mathfrak{t}}^* \hookrightarrow \left( \check{\mathfrak{g}}/\check{\mathfrak{u}} \right)^* = \check{\mathfrak{u}}^\perp \subseteq \check{\mathfrak{g}}^*$ takes a linear form $\omega \in \check{\mathfrak{t}}^*$ on $\check{\mathfrak{t}}$ to the unique linear form $\tilde{\omega} \in \check{\mathfrak{g}}^*$ on $\check{\mathfrak{g}}$ extending $\omega$ (along the embedding $\check{\mathfrak{t}} \hookrightarrow \check{\mathfrak{g}}$) and vanishing on both $\check{\mathfrak{u}}$ and $\check{\mathfrak{u}}^-$. Note that this embedding appeared in Construction~\ref{partialkostantslice} as the example of our (coadjoint) partial Kostant section $\sigma_I: \check{\mathfrak{c}}_I \hookrightarrow \check{\mathfrak{g}}^*$ in the case $I = \emptyset$. We obtain an induced $\check{G}\times\check{T}$-equivariant algebra homomorphism on the coordinate rings 
\begin{equation}\label{ginzburgrichespectralmap}
\mathcal{O}(T^*(\check{G}/\check{U})) \rightarrow \mathcal{O}(\check{G}) \otimes \mathcal{O}(\check{\mathfrak{t}}^*) = \mathcal{O}(\check{G}) \otimes R_\emptyset. 
\end{equation}
With this notation in place, we can recall the ``semi-classical limit'' of the main theorem of Ginzburg and Riche as stated in \cite[Theorem 2.3.1]{GR13} in a form convenient for our purposes. The notation $\Upsilon_\emptyset$ that we use in the statement is intended to emphasize that their isomorphism will be recovered in the case $I = \emptyset$ of our Theorem~\ref{Main Theorem}. 

\begin{theorem}[{Ginzburg, Riche \cite[Theorem 2.3.1]{GR13}}] \label{grmaintheoremdetailed}
    There exists a unique $\check{G} \times \check{T}$-equivariant isomorphism of graded $R_\emptyset$-algebras 
    $$
    \Upsilon_\emptyset: H_{T(\mathcal{O})}^*(\Gr_T,  i_\emptyset^{!,\natural} \mathcal{F}_\reg) \simeq \mathcal{O}(T^*(\check{G}/\check{U}))
    $$
    such that the following diagram of $R_L$-algebras commutes:
    \begin{equation}\label{ginzburgrichecommutativediagram}
    \begin{tikzcd}
         H_{T(\mathcal{O})}^*(\Gr_T,  i_\emptyset^{!,\natural} \mathcal{F}_\reg) \ar[r, "\xi_\emptyset^\natural"] \ar[d, "\Upsilon_\emptyset"] & H_{T(\mathcal{O})}^*(\Gr_T, \operatorname{Res}^\natural_\emptyset(\mathcal{F}_\reg)) \ar[d, "\eqref{ginzburgricheparabolicrestriction}"]\\
         \mathcal{O}(T^*(\check{G}/\check{U})) \ar[r, "\eqref{ginzburgrichespectralmap}"] & \mathcal{O}(\check{G}) \otimes R_\emptyset
    \end{tikzcd}
    \end{equation}
\end{theorem}

\begin{remark}
    The fact that $\Upsilon_\emptyset$ is an an algebra homomorphism is an immediate consequence of the commutativity of \eqref{ginzburgrichecommutativediagram}. Indeed, the other three maps $\xi^\natural_\emptyset$, \eqref{ginzburgricheparabolicrestriction}, and \eqref{ginzburgrichespectralmap} appearing in the diagram are algebra homomorphisms, and the horizontal maps $\xi_\emptyset^\natural$, \eqref{ginzburgrichespectralmap} are injective (if it is not clear that \eqref{ginzburgrichespectralmap} is injective, simply note that \eqref{ginzburgricheparabolicrestriction} is an isomorphism and that $\xi_\emptyset^\natural$ is injective by Proposition~\ref{corestriction vs parabolic restriction localization}). 
\end{remark}

\subsection{The Theorem of Yun-Zhu}\label{yun zhu comparison} We will now conclude our discussion of the results of Yun and Zhu from \cite{yun2009integral} initiated in §§\ref{equivarianthomologysubsection}. Recall that in Remark~\ref{homomorphism on regular centralizers, group schemes} we defined
$$
\mathfrak{A}_I :=  \mathfrak{A}_{L_I} := \operatorname{Spec}{H_*^{L_I(\mathcal{O})}}(\Gr_I, \CC) \rightarrow \operatorname{Spec}{R_I} = \mathfrak{c}_I 
$$
to be the $\mathfrak{c}_I$-group scheme obtained from the natural Hopf $R_I$-algebra structure of Construction~\ref{hopfalgebrastructure} on the $R_I$-module $H_*^{L_I(\mathcal{O})}(\Gr_I, \CC)$. In what follows, we will use the canonical identification 
$$
\mathfrak{c}_I = \mathfrak{t}\doubslash W_I \simeq \check{\mathfrak{t}}^* \doubslash W_I = \check{\mathfrak{c}}_I
$$
without further comment. For any $\mathcal{F} \in \mathcal{P}_{L_I(\mathcal{O})}(\Gr_I)$, we recalled in Construction~\ref{comodulestructure} the natural $\mathfrak{A}_I$-module structure on the $R_I$-module $H^*_{L_I(\mathcal{O})}(\Gr_I, \mathcal{F})$. We can take $\mathcal{F}$ to be the regular sheaf on $\Gr_I$, that is, the Hopf algebra object $\mathcal{F}_{I,\reg} \in \mathcal{P}_{L_I(\mathcal{O})}(\Gr_I)$ corresponding to the regular representation $\mathcal{O}(\check{L}_I) \in \operatorname{Rep}(\check{L}_I)$ under the geometric Satake equivalence. Moreover, \cite[Lemma 2.4]{yun2009integral} yields a natural isomorphism 
$$
H^*_{L_I(\mathcal{O})}(\Gr_I, \mathcal{F}_{I,\reg}) \simeq H_{T(\mathcal{O})}^*(\Gr_I, \mathcal{F}_{I,\reg})^{W_I} \simeq (\mathcal{O}(\check{L}_I) \otimes R_\emptyset)^{W_I} \simeq \mathcal{O}(\check{L}_I) \otimes R_I
$$
of Hopf $R_I$-algebras. Hence, we obtain a natural action of $\mathfrak{A}_I$ on $\mathcal{O}(\check{L}_I) \otimes R_I$ through Hopf $R_I = \mathcal{O}(\check{\mathfrak{c}}_I)$-algebra homomorphisms. In other words, we obtain an action of $\mathfrak{A}_I$ on the $\check{\mathfrak{c}}_I$-scheme $\check{L}_I\times \check{\mathfrak{c}}_I$ through group homomorphisms. Acting on the identity section of $\check{L}_I \times \check{\mathfrak{c}}_I$ defines a homomorphism of $\check{\mathfrak{c}}_I$-group schemes 
$$
\mathfrak{A}_I \rightarrow \check{L}_I \times \check{\mathfrak{c}}_I. 
$$
On the other hand, we have the $\check{\mathfrak{c}}_I$-group scheme of regular centralizers $\mathfrak{J}_I \rightarrow \check{\mathfrak{c}}_I$, recalled in §§\ref{spectralactionofthecentralizer}. By its very construction, it is equipped with a canonical group scheme embedding $\mathfrak{J}_I \hookrightarrow \check{{L}}_I \times \check{\mathfrak{c}}_I$.  

\begin{theorem}[{Yun, Zhu \cite{yun2009integral}}]\label{yunzhumaintheorem}
    Assume that $G$ has almost simple derived subgroup. Then, the map $\mathfrak{A}_I \rightarrow \check{L}_I \times \check{\mathfrak{c}}_I$ induces a canonical isomorphism of $\check{\mathfrak{c}}_I$-group schemes $\mathfrak{A}_I \simeq \mathfrak{J}_I$. 
\end{theorem}

\begin{remark}
    Note that the hypothesis that $G$ has almost simple derived subgroup passes to all Levi subgroups of $G$, which justifies our application of \cite{yun2009integral} to the connected reductive group $L_I$. 
\end{remark}

\begin{remark}
    Assuming that $G$ has almost simple derived subgroup, we can combine the natural isomorphism of Remark~\ref{change of group} with Theorem~\ref{yunzhumaintheorem} (in the case $L = G$) to obtain an isomorphism of $\check{\mathfrak{c}}_I$-group schemes
    $$
    \operatorname{Spec}{H_*^{L(\mathcal{O})}}(\Gr_G, \CC) \simeq \mathfrak{A} \times_{\check{\mathfrak{c}}} \check{\mathfrak{c}}_I \simeq \mathfrak{J} \times_{\check{\mathfrak{c}}} \check{\mathfrak{c}}_I. 
    $$
    We can now apply Remark~\ref{homomorphism on regular centralizers, group schemes} to construct a homomorphism of $\check{\mathfrak{c}}_I$-group schemes
    \begin{equation}\label{geometricconstructionofrho}
    \rho_I^{\mathrm{geom}}: \mathfrak{J} \times_{\check{\mathfrak{c}}} \check{\mathfrak{c}}_I \simeq \mathfrak{A} \times_{{\mathfrak{c}}} {\mathfrak{c}}_I \xrightarrow{\rho^G_{L_I}} \mathfrak{A}_I \simeq \mathfrak{J}_I. 
    \end{equation}
    On the other hand, \cite[Theorem 3.4.2]{Ri17} provides canonical identifications of schemes of (commutative) Lie algebras over the bases $\check{\mathfrak{c}}$ and $\check{\mathfrak{c}}_I$, respectively:
    \begin{eqnarray}\label{liealgebraofregularcentralizer}
    \mathbb{L}\mathrm{ie}(\mathfrak{J}/\check{\mathfrak{c}}) \simeq T^*(\check{\mathfrak{c}}) & \mathbb{L}\mathrm{ie}(\mathfrak{J}_I/\check{\mathfrak{c}_I}) \simeq T^*(\check{\mathfrak{c}}_I). 
    \end{eqnarray}
    Here, $T^*(\check{\mathfrak{c}})$ and $T^*(\check{\mathfrak{c}}_I)$ denote the respective cotangent bundles. On the other hand, we may differentiate the map $\pi_I: \check{\mathfrak{c}}_I \rightarrow \check{\mathfrak{c}}$ to obtain a morphism of schemes of (commutative) Lie algebras over $\check{\mathfrak{c}}_I$
    $$
    \pi_I^*: T^*(\check{\mathfrak{c}}) \times_{\check{\mathfrak{c}}} \check{\mathfrak{c}}_I \rightarrow T^*(\check{\mathfrak{c}}_I). 
    $$
\end{remark}

Since the following proposition will not be used in what follows, we omit its (straightforward) proof. 

\begin{proposition}
    The following diagram of $\check{\mathfrak{c}}_I$-schemes commutes.
    $$
    \begin{tikzcd}
        T^*(\check{\mathfrak{c}}) \times_{\check{\mathfrak{c}}} \check{\mathfrak{c}}_I \ar[rr, "\pi_I^*"] \ar[d, "\eqref{liealgebraofregularcentralizer}"] && T^*(\check{\mathfrak{c}}_I) \ar[d, "\eqref{liealgebraofregularcentralizer}"]\\
        \mathbb{L}\mathrm{ie}(\mathfrak{J}/\check{\mathfrak{c}}) \times_{\check{\mathfrak{c}}} \check{\mathfrak{c}}_I \ar[rr,"\mathbb{L}\mathrm{ie}(\rho_I^{\mathrm{geom}})"] && \mathbb{L}\mathrm{ie}(\mathfrak{J}_I/\check{\mathfrak{c}_I})
    \end{tikzcd}
    $$
\end{proposition}

\subsection{Proof of Theorem~\ref{Main Theorem}}\label{theproof} We are now ready to tie in the results of §\ref{automorphic} and §\ref{Hamiltonian Reduction} to give our proof of the main theorem, Theorem~\ref{Main Theorem}. We will state and prove a refined version of it, Theorem~\ref{Main Theorem 1} below. The following elementary lemma will be used in the course of the proof.

\begin{lemma}\label{automorphichartogs}
    Let $M$ be a flat $R_L$-module. Let $f_I, g_I \in R_I$ denote the elements of \eqref{fIdefinition} and \eqref{gIdefinition}, respectively. Then, the restriction map 
    $$
    M \rightarrow M_{f_I} \times_{M_{f_I g_I}} M_{g_I}
    $$
    is an isomorphism. 
\end{lemma}
\begin{proof}
    By Lazard's theorem \cite[\href{https://stacks.math.columbia.edu/tag/058G}{Tag 058G}]{stacks-project}, we may express $M$ as a filtered colimit of free modules of finite rank. The exactness of localization then allows us to assume that $M$ itself is free of finite rank. Obviously, it is then sufficient to treat the case in which $M$ is free of rank $1$. The claim is then that the restriction map 
    $$
    R_L \rightarrow (R_L)_{f_I} \times_{(R_L)_{f_I g_I}} (R_L)_{g_I}
    $$
    is a bijection. Since $R_L$ is a regular domain (in particular an integrally closed domain), the claim follows from the observation (see the proof of Lemma~\ref{Spectral Hartogs}) that 
    $$
    \operatorname{codim}_{\operatorname{Spec}{R_L}}(\operatorname{Spec}{R_L/(f_I)} \cap \operatorname{Spec}{R_L/(g_I)}) \geq 2
    $$
    in conjunction with Hartog's principle \cite[\href{https://stacks.math.columbia.edu/tag/0BCS}{Tag 0BCS}]{stacks-project}.
\end{proof}

Before proceeding to the proof of Theorem~\ref{Main Theorem 1}, we make some preliminary definitions and observations.

\begin{construction}\label{spectralgrading}
    We define the $\ZZ$-grading on the algebra $\mathcal{O}(T^*(\check{G}/(\check{U},\psi_I)))$ which will correspond under the isomorphism of Theorem~\ref{Main Theorem} to the cohomological grading on the geometric side. Note that we have an injection 
    $$
    \mathcal{O}(T^*(\check{G}/(\check{U},\psi_I))) \simeq (\mathcal{O}(\check{G}) \otimes \mathcal{O}{(\check{\mathfrak{u}}^\perp + \psi_I)})^{\check{U}} \subseteq \mathcal{O}(\check{G}) \otimes \mathcal{O}(\check{\mathfrak{u}}^\perp + \psi_I).
    $$
    We take the $\ZZ$-grading on the $\check{L}_I$-module $\mathcal{O}(\check{G})$ given by the eigenvalues of the distinguished torus $2\check{\rho}_I: \mathbb{G}_m \rightarrow \check{L}_I$ (that which induces the cohomological grading under geometric Satake for the group $\check{L}_I$). Observe that if $\mathbb{G}_m$ acts on $\check{\mathfrak{g}}^*$ through $2\check{\rho}_I$, then the element $\psi_I$ acquires the weight $+2$ (recall that there is an $\check{G}$-equivariant isomorphism $\check{\mathfrak{g}}^* \simeq \check{\mathfrak{g}}$ taking $\psi_I$ to the irregular nilpotent element $e_I = \sum_{\alpha \in I} X_\alpha$, which has $2\check{\rho}_I$-weight $+2$). Instead, we consider the $\mathbb{G}_m$-action on $\check{\mathfrak{g}}^*$ given by $t \mapsto t^{-2} 2\check{\rho}_I(t)$; by construction, this action leaves $\psi_I$ fixed and hence preserves the affine slice $\check{\mathfrak{u}}^\perp + \psi_I$. We grade $\mathcal{O}(\check{\mathfrak{u}}^\perp + \psi_I)$ by the eigenvalues of \textit{this} $\mathbb{G}_m$-action. 

    We need to grade the algebra $\mathcal{O}(T^*(\check{G}/(\check{U}, \psi)) \times_{\check{\mathfrak{c}}} \check{\mathfrak{c}}_I)$ in such a way that the canonical map 
    $$
    (\eta_{I-\gen}^-)^*: \mathcal{O}(T^*(\check{G}/(\check{U},\psi_I))) \rightarrow \mathcal{O}(T^*(\check{G}/(\check{U}, \psi)) \times_{\check{\mathfrak{c}}} \check{\mathfrak{c}}_I) \simeq \mathcal{O}(\check{G}) \otimes \mathcal{O}(\check{\mathfrak{c}}_I). 
    $$
    of Proposition~\ref{Anti-Generic comparison} is graded. Of course, we once again grade $\mathcal{O}(\check{G})$ by the coweight $2\check{\rho}_I: \mathbb{G}_m \rightarrow \check{L}_I$. We must grade $\mathcal{O}(\check{\mathfrak{c}}_I)$ in such a way that the algebra map $\sigma_I^*: \mathcal{O}(\check{\mathfrak{u}}^\perp + \psi_I) \rightarrow \mathcal{O}(\check{\mathfrak{c}}_I)$ of Construction~\ref{partialkostantslice} is graded. Dually, we must define an action of $\mathbb{G}_m$ on $\check{\mathfrak{c}}_I$ such that the partial Kostant section $\sigma_I: \check{\mathfrak{c}}_I \rightarrow \check{\mathfrak{u}}^\perp + \psi_I$ is $\mathbb{G}_m$-equivariant. Since $\sigma_I$ is the composition of the embedding $\check{\mathfrak{l}}_I^* \hookrightarrow \check{\mathfrak{g}}^*$ with the standard Kostant section $\kappa_I: \check{\mathfrak{c}}_I \rightarrow \check{\mathfrak{l}}_I^*$, the problem is to give an action of $\mathbb{G}_m$ on $\check{\mathfrak{c}}_I$ for which the Kostant section $\kappa_I: \check{\mathfrak{c}}_I \rightarrow \check{\mathfrak{l}}_I^*$ is $\mathbb{G}_m$-equivariant, where $\mathbb{G}_m$ acts on $\check{\mathfrak{l}}_I^*$ (as above) by $t \mapsto t^{-2} 2\check{\rho}_I(t)$. But this action of $\mathbb{G}_m$ on $\check{\mathfrak{l}}_I^*$ clearly preserves the image of the closed immersion $\kappa_I$, so we obtain the desired (and in fact uniquely determined) action of $\mathbb{G}_m$ on $\check{\mathfrak{c}}_I$. 

    An alternative construction of this action can be obtained as follows. Let $\mathbb{G}_m$ act on $\check{\mathfrak{l}}_I^*$ by $t \mapsto t^{-2}$. This action commutes with the coadjoint action of $\check{L}_I$, hence induces an action of $\mathbb{G}_m$ on $\check{\mathfrak{l}}_I^*\doubslash \check{L}_I \simeq \check{\mathfrak{c}}_I$; this is the same action. The latter description makes it clear that the natural map $\check{\mathfrak{t}}^* \rightarrow \check{\mathfrak{c}}_I$ is $\mathbb{G}_m$-equivariant, where $\mathbb{G}_m$ acts on $\check{\mathfrak{t}}^*$ by $t \mapsto t^{-2}$. Hence, we obtain an inclusion of graded algebras $\mathcal{O}(\check{\mathfrak{c}}_I) \hookrightarrow \mathcal{O}(\check{\mathfrak{t}}^*) \simeq \operatorname{Sym}{\mathfrak{t}^*}$, where $\mathfrak{t}^*$ is placed in degree $+2$. This grading on $\mathcal{O}(\mathfrak{t}^*)$ agrees with the cohomological grading on $R_\emptyset = H_T^*(\mathrm{pt},\CC)$ under the standard isomorphism $R_\emptyset \simeq \operatorname{Sym}{\mathfrak{t}^*}$. It immediately follows that the standard isomorphism $R_I \simeq \mathcal{O}(\check{\mathfrak{c}}_I)$ identifies the cohomological grading on $R_I$ with our grading on $\mathcal{O}(\check{\mathfrak{c}}_I)$.   

    This action of $\mathbb{G}_m$ on $\check{\mathfrak{c}}_I$ appears frequently in representation theory; for example, in Ng\^{o}'s work \cite{Ngo10}.
\end{construction}
 
\begin{construction}\label{Upsilondefinition}
    We need an $L_I(\mathcal{O})$-equivariant version of the isomorphism \eqref{ginzburgricheparabolicrestriction}. More precisely, we define an isomorphism of $R_I$-algebras
    \begin{equation}
    \Upsilon: H^*_{L(\mathcal{O})}(\Gr_I, \operatorname{Res}^\natural_I(\mathcal{F}_\reg)) \simeq \mathcal{O}(T^*(\check{G}/(\check{U},\psi)) \times_{\check{\mathfrak{c}}} \check{\mathfrak{c}}_I). 
    \end{equation}
    By \cite[Lemma 2.2]{yun2009integral} in conjunction with Theorem~\ref{parabolic restriction and Satake}, we have a natural isomorphism of graded $R_\emptyset$-algebras
    $$
    H^*_{T(\mathcal{O})}(\Gr_I, \operatorname{Res}^\natural_I(\mathcal{F}_\reg)) \operatornamewithlimits{\simeq}^{(1)} H^*(\Gr_I, \operatorname{Res}^\natural_I(\mathcal{F}_\reg)) \otimes R_\emptyset \operatornamewithlimits{\simeq}^{(2)} \mathcal{O}(\check{G}) \otimes R_\emptyset \operatornamewithlimits{\simeq}^{(3)} \mathcal{O}(\check{G} \times \check{\mathfrak{c}}_I) \otimes_{R_I} R_\emptyset. 
    $$
    We can pass to $W_I$-invariants (using the equivariant formality of $\mathrm{Res}^\natural_I(\mathcal{F}_\reg)$) to obtain the desired isomorphism 
    $$
    \Upsilon: H^*_{L(\mathcal{O})}(\Gr_I, \operatorname{Res}^\natural_I(\mathcal{F}_\reg)) \simeq \mathcal{O}(\check{G} \times \check{\mathfrak{c}}_I) \otimes_{R_I} R_\emptyset  ^{W_I} \simeq \mathcal{O}(T^*(\check{G}/(\check{U},\psi)) \times_{\check{\mathfrak{c}}} \check{\mathfrak{c}}_I).
    $$
    The last isomorphism is Lemma~\ref{Reduction by Restriction}. By construction, $\Upsilon$ is an isomorphism of $R_I$-algebras. It is $\check{G}$-equivariant: indeed, $(1)$ is $\check{G}$-equivariant by the very definition of the $\check{G}$-action on $H^*_{T(\mathcal{O})}(\Gr_I, \operatorname{Res}^\natural_I(\mathcal{F}_\reg))$, $(2)$ and $(3)$ are obviously $\check{G}$-equivariant, and the isomorphism of Lemma~\ref{Reduction by Restriction} is $\check{G}$-equivariant by Lemma~\ref{Reduction by restriction, equivariant}. When $G$ has almost simple derived subgroup, we can use the Yun-Zhu isomorphism $\mathfrak{A}_I \simeq \mathfrak{J}_I$ of Theorem~\ref{yunzhumaintheorem} and ask if $\Upsilon$ is ${\mathfrak{J}}_I$-equivariant. That $\Upsilon$ is indeed ${\mathfrak{J}}_I$-equivariant is a direct consequence of the Bezrukavnikov-Finkelberg proof of the derived Satake equivalence \cite{BF07}, which identifies the equivariant cohomology functor with the Kostant-Whittaker reduction of $\operatorname{Sym}{\check{\mathfrak{g}}[-2]}$-modules; we refer the reader to Riche's proof \cite[Theorem 5.5.1]{Ri17} of the mixed modular derived Satake equivalence for the details.

    Finally, we assert that $\Upsilon$ is compatible with the $\ZZ$-gradings. The isomorphism (1) respects the cohomological grading by construction (see the proof of \cite[Lemma 2.2]{yun2009integral}). The isomorphism (2) identifies the cohomological grading on $H^*(\Gr_I, \operatorname{Res}^\natural_I(\mathcal{F}_\reg))$ with the the grading of the $\check{L}_I$-module $\mathcal{O}(\check{G})$ by the eigenvalues of $2\check{\rho}_I: \mathbb{G}_m \rightarrow \check{L}_I$ (this is a standard property of the geometric Satake equivalence; see for example \cite[Theorem 3.6]{MV04}). Therefore, $\Upsilon$ identifies the cohomological grading on $H^*_{L(\mathcal{O})}(\Gr_I, \operatorname{Res}_I^\natural(\mathcal{F}_\reg))$ with the grading on $\mathcal{O}(T^*(\check{G}/(\check{U},\psi)) \times_{\check{\mathfrak{c}}} \check{\mathfrak{c}}_I) \simeq \mathcal{O}(\check{G}) \otimes \mathcal{O}(\check{\mathfrak{c}}_I)$ induced by the grading on $\mathcal{O}(\check{G})$ by the eigenvalues of $2\check{\rho}_I$ and the cohomological grading on $\mathcal{O}(\mathfrak{\check{c}}_I) \simeq R_I$; that is exactly the grading defined in Construction~\ref{spectralgrading}. 
\end{construction}

\begin{theorem}\label{Main Theorem 1}
    Let $L = L_I \subseteq G$ denote a Levi subgroup of the connected reductive group $G$ with simple roots $I \subseteq \Delta$. Let $i_I: \Gr_I := \Gr_L \hookrightarrow \Gr_G$ denote the induced closed immersion on affine Grassmannians. Let $\mathcal{F}_\reg \in D_{G(\mathcal{O})}(\Gr_G)$ denote the regular sheaf of Remark~\ref{ring structure on the regular sheaf} equipped with its natural structure of $\check{G}$-equivariant ring object of $D_{G(\mathcal{O})}(\Gr_G)$. Let $H_{L(\mathcal{O})}^*(\Gr_I, i_I^{!,\natural} \mathcal{F}_\reg)$ denote the $L(\mathcal{O})$-equivariant cohomology of its normalized (see Remark~\ref{normalized corestriction}) corestriction, equipped with its natural structure of graded $R_I := H^*_{L(\mathcal{O})}(\mathrm{pt})$-algebra (arising from the lax monoidal structures of Construction~\ref{lax definition}, Construction~\ref{colax definition}, and Remark~\ref{lax monoidal structure on cohomology}). Moreover, in the case that $G$ has almost simple derived subgroup, we view $H_{L(\mathcal{O})}^*(\Gr_L, i_I^{!,\natural} \mathcal{F}_\reg)$ as a module over the $R_L$-group scheme $\mathfrak{J}_I$ via Remark~\ref{comodulestructure} and the Yun-Zhu isomorphism $\mathfrak{J}_I \simeq \mathfrak{A}_I = \operatorname{Spec}{H^{L(\mathcal{O})}_*(\Gr_L, \CC)}$ of Theorem~\ref{yunzhumaintheorem}. 

    On the other hand, consider the ring of regular functions $\mathcal{O}(T^*(\check{G}/(\check{U}, \psi_I)))$. Here, $\psi_I \in \check{\mathfrak{u}}^*$ is the additive character of Notation~\ref{additivecharacterdefinition} and $T^*(\check{G}/(\check{U}, \psi_I)) := T^*\check{G} \tripslash (\check{U}, \psi_I)$ is the Hamiltonian reduction of the Hamiltonian $\check{U}$-variety $T^*\check{G}$ at the level $\psi_I \in \check{\mathfrak{u}}^*$. View $\mathcal{O}(T^*(\check{G}/(\check{U}, \psi_I)))$ as an $R_I$-algebra via the canonical map $T^*(\check{G}/(\check{U}, \psi_I)) \rightarrow \check{\mathfrak{c}}_I$ of \eqref{projectiontochevalley}, where $\check{\mathfrak{c}}_I = \operatorname{Spec}{R_I}$ is the Chevalley scheme of $\check{L}_I$. Equip $\mathcal{O}(T^*(\check{G}/(\check{U}, \psi_I)))$ with its natural structure of module over the $R_I$-group scheme $\mathfrak{J}_I$ (see §§\ref{spectralactionofthecentralizer}). Moreover, regard $\mathcal{O}(T^*(\check{G}/(\check{U}, \psi_I)))$ as a $\check{G}$-module via the natural left translation action of $\check{G}$ on $T^*(\check{G}/(\check{U},\psi_I))$. We equip $T^*(\check{G}/(\check{U},\psi_I))$ with the grading induced by the $\mathbb{G}_m$-action of Construction~\ref{spectralgrading}.

    Then, there exists a unique $\check{G}$-equivariant isomorphism of graded $R_I$-algebras 
    $$
    \Upsilon_I: H_{L(\mathcal{O})}^*(\Gr_L, i_I^{!,\natural} \mathcal{F}_\reg) \simeq \mathcal{O}(T^*(\check{G}/(\check{U}, \psi_I)))
    $$
    such that the following diagram of $R_L$-algebras commutes:
    \begin{equation}\label{Main Theorem Main commutative diagram}
    \begin{tikzcd}
        H^*_{L(\mathcal{O})}(\Gr_L, i^{!,\natural}\mathcal{F}_\reg) \ar[rr, "\xi_I^\natural"] \ar[d, "\Upsilon_I"] && H^*_{L(\mathcal{O})}(\Gr_L, \operatorname{Res}^\natural_I(\mathcal{F}_\reg)) \ar[d, "\Upsilon"] \\
        \mathcal{O}(T^*(\check{G}/(\check{U}, \psi_I))) \ar[rr, "(\eta_{I-\gen}^-)^*"] && \mathcal{O}(T^*(\check{G}/(\check{U}, \psi)) \times_{\check{\mathfrak{c}}} \check{\mathfrak{c}}_I) 
    \end{tikzcd}
    \end{equation}
    In the case that $G$ has almost simple derived subgroup, we also have that $\Upsilon_I$ is $\mathfrak{J}_I$-equivariant. 
\end{theorem}
\begin{proof}
    Recall the injective graded $R_I$-module homomorphism of \eqref{ungraded anti-generic comparison automorphic side}
    $$
    \xi_I^\natural: H_{L(\mathcal{O})}^*(\Gr_L, i^{!,\natural}_I \mathcal{F}_\reg) \rightarrow H_{L(\mathcal{O})}^*(\Gr_L, \operatorname{Res}_I(\mathcal{F}_\reg)). 
    $$
    It is a $\check{G}$-equivariant (obvious) $\mathfrak{A}_I$-equivariant (Remark~\ref{anti-generic comparison automorphic side and the regular centrlizer}) ring homomorphism (Remark~\ref{xiIisaringhomomorphism}). Moreover, its localization $(\xi^\natural_I)_{g_I}$ at the element $g_I \in R_I$ of \eqref{gIdefinition} is an isomorphism of $(R_I)_{g_I}$-modules (see \eqref{ungraded anti-generic comparison automorphic side} and Proposition~\ref{corestriction vs parabolic restriction localization}). Note that the commutativity of \eqref{Main Theorem Main commutative diagram} would imply, upon localization at $g_I \in R_I$, that the isomorphism $(\Upsilon_I)_{g_I}$ fits into the following commutative diagram of isomorphisms of $(R_I)_{g_I}$-modules:
    \begin{equation}\label{Main Theorem Eq1}
    \begin{tikzcd}
            H_{L(\mathcal{O})}^*(\Gr_L, i^{!,\natural}_I \mathcal{F}_\reg)_{g_I} \ar[r, "(\xi_I^\natural)_{g_I}"] \ar[d,"(\Upsilon_I)_{g_I}",dashed] & H_{L(\mathcal{O})}^*(\Gr_G, \operatorname{Res}^\natural_I(\mathcal{F}_\reg))_{g_I} \ar[d, "\Upsilon_{g_I}"] \\
            \mathcal{O}(T^*(\check{G}/(\check{U}, \psi_I)))_{g_I} \ar[r, "(\eta_{I-\gen}^-)^*"] & \mathcal{O}(T^*(\check{G}/(\check{U}, \psi)) \times_{\check{\mathfrak{c}}} \check{\mathfrak{c}}_I)_{g_I}. 
    \end{tikzcd}
    \end{equation}
    The isomorphisms $(\xi_I^\natural)_{g_I}$ (explained above), $\Upsilon_{g_I}$ (see Construction~\ref{Upsilondefinition}), and $(\eta_{I-\gen}^-)^*$ (see §§\ref{affineclosure}, and Construction~\ref{spectralgrading} for the grading compatibility) are $\check{G}$-equivariant $\mathfrak{J}_I$-equivariant (in the case that $G$ has simple derived subgroup) graded $R_I$-algebra isomorphisms. Therefore, it will follow that $(\Upsilon_I)_{g_I}$ is a $\check{G}$-equivariant $\mathfrak{J}_I$-equivariant graded $(R_I)_{g_I}$-algebra isomorphism. Since $\Upsilon_I$ is an isomorphism of \textit{free} modules (by Proposition~\ref{equivariantformality}) over the domain $R_I$, it follows that $\Upsilon_I$ is in fact a $\check{G}$-equivariant $\mathfrak{J}_I$-equivariant (in the case that $G$ has simple derived subgroup) isomorphism of graded $R_I$-algebras and that \eqref{Main Theorem Main commutative diagram} commutes. Hence, it is sufficient to construct $\Upsilon_I$ as a mere $R_I$-\textit{module} isomorphism, subject to the further constraint that the diagram \eqref{Main Theorem Eq1} commutes.

    If we succeed in defining $\Upsilon_I$, then we will have a commutative diagram of $R_I$-modules:
    \begin{equation}
    \begin{tikzcd}
        H_{L(\mathcal{O})}^*(\Gr_L, i^{!,\natural}_I \mathcal{F}_\reg) \ar[r] \ar[d, dashed,"\Upsilon_I"] & H_{L(\mathcal{O})}^*(\Gr_L, i^{!,\natural}_I \mathcal{F}_\reg)_{f_I} \times_{H_{L(\mathcal{O})}^*(\Gr_L, i^{!,\natural}_I \mathcal{F}_\reg)_{f_I g_I}} H_{L(\mathcal{O})}^*(\Gr_L, i^{!,\natural}_I \mathcal{F}_\reg)_{g_I} \ar[d, "(\Upsilon_I)_{f_I} \times (\Upsilon_I)_{g_I}", dashed]\\
         \mathcal{O}(T^*(\check{G}/(\check{U}, \psi_I))) \ar[r,"\eta_I"] &  \mathcal{O}(T^*(\check{G}/(\check{U},\psi_I)))_{f_I} \times_{ \mathcal{O}(T^*(\check{G}/(\check{U}, \psi_I)))_{f_I g_I}}  \mathcal{O}(T^*(\check{G}/(\check{U}, \psi_I)))_{g_I}
    \end{tikzcd}
    \end{equation}
    The horizontal maps are the obvious restriction morphisms. In Construction~\ref{descriptionofaffineclosure} (using Lemma~\ref{Spectral Hartogs}), we invoked Hartog's principle to show that the bottom horizontal map $\eta_I$ is an isomorphism. Moreover, by the above observation Lemma~\ref{automorphichartogs}, the top horizontal map is also an isomorphism of $R_I$-modules. Therefore, to define $\Upsilon_I$, it is in fact sufficient to directly define the isomorphisms 
    \begin{align*}
        \Upsilon_{I-\gen} := (\Upsilon_I)_{f_I}:  H_{L(\mathcal{O})}^*(\Gr_L, i^{!,\natural}_I \mathcal{F}_\reg)_{f_I} &\simeq  \mathcal{O}(T^*(\check{G}/(\check{U}, \psi_I)))_{f_I}\\
        \Upsilon_{I-\gen}^- := (\Upsilon_I)_{g_I}: H_{L(\mathcal{O})}^*(\Gr_L, i^{!,\natural}_I \mathcal{F}_\reg)_{g_I} &\simeq  \mathcal{O}(T^*(\check{G}/(\check{U}, \psi_I)))_{g_I}
    \end{align*}
    subject to the constraint that the following diagram of $R_I$-modules commutes. 
    \begin{equation}\label{Main Theorem Eq2}
    \begin{tikzcd}
         H_{L(\mathcal{O})}^*(\Gr_L, i^{!,\natural}_I \mathcal{F}_\reg)_{f_Ig_I} \ar[rr, "(\Upsilon_{I-\gen})_{g_I}"] \ar[d, equals] &&  \mathcal{O}(T^*(\check{G}/(\check{U}, \psi_I)))_{f_Ig_I} \ar[d, equals]\\
          H_{L(\mathcal{O})}^*(\Gr_L, i^{!,\natural}_I \mathcal{F}_\reg)_{g_If_I} \ar[rr, "(\Upsilon_{I-\gen}^-)_{f_I}"] &&  \mathcal{O}(T^*(\check{G}/(\check{U}, \psi_I)))_{g_If_I}
    \end{tikzcd}
    \end{equation}
    Here, the vertical equalities denote the obvious canonical isomorphisms. The condition that the diagram \eqref{Main Theorem Eq1} commutes now becomes the \textit{definition} of the isomorphism $\Upsilon_{I-\gen}^-$. With this choice of $\Upsilon_{I-\gen}^-$, our goal is reduced to producing the isomorphism $\Upsilon_{I-\gen}$ subject to the constraint that the following expanded version of the diagram \eqref{Main Theorem Eq2} commutes (where the equalities are now suppressed for brevity):
    \begin{equation}
    \begin{tikzcd}\label{Main Theorem compatiblity diagram}
        H_{L(\mathcal{O})}^*(\Gr_L, i^{!,\natural}_I \mathcal{F}_\reg)_{f_Ig_I} \ar[rr, "(\Upsilon_{I-\gen})_{g_I}",dashed] \ar[d, "(\xi_I^\natural)_{f_I g_I}"] &&  \mathcal{O}(T^*(\check{G}/(\check{U}, \psi_I)))_{f_Ig_I} \ar[d, "(\eta_{I-\gen}^-)^*_{f_I}"] \\
        H_{L(\mathcal{O})}^*(\Gr_G, \operatorname{Res}_I^\natural(\mathcal{F}_\reg))_{f_Ig_I} \ar[rr, "\Upsilon_{f_I g_I}"] && \mathcal{O}(T^*(\check{G}/(\check{U},\psi)) \times_{\check{\mathfrak{c}}} \check{\mathfrak{c}}_I)_{f_I g_I}
    \end{tikzcd}    
    \end{equation}
    Note, moreover, that $H_{L(\mathcal{O})}^*(\Gr_L, i^{!,\natural}_I \mathcal{F}_\reg)$ differs from $H_{L(\mathcal{O})}^*(\Gr_L, i^{!}_I \mathcal{F}_\reg)$ only by a change in grading. Since we have already established that the compatibility of $\Upsilon_I$ with the gradings is an automatic consequence of the commutativity of \eqref{Main Theorem Eq1}, we will now ignore the distinction between these $R_I$-modules. 

    We have an isomorphism of free $(R_\emptyset)_{f_I}$-modules (see Remark~\ref{generic comparison automorphic side})
    \begin{align}\label{Main Theorem Eq3}
        H_{T(\mathcal{O})}^*(\Gr_L, i_I^! \mathcal{F}_\reg)_{f_I} & \simeq H_{T(\mathcal{O})}^*(\Gr_T, i_\emptyset^! \mathcal{F}_\reg)_{f_I} 
    \end{align}
    where $i_\emptyset: \Gr_T \hookrightarrow \Gr_G$ continues to denote the closed inclusion of $\Gr_T$ into $\Gr_G$ and $R_\emptyset = H_T^*(\mathrm{pt},\CC)$. Next, we bring in the isomorphism of the Theorem~\ref{grmaintheoremdetailed} of Ginzburg-Riche
    \begin{equation}\label{Main Theorem Eq4}
        (\Upsilon_{\emptyset})_{f_I}: H_{T(\mathcal{O})}^*(\Gr_T, i_\emptyset^! \mathcal{F}_\reg)_{f_I}  \simeq \mathcal{O}(T^*(\check{G}/\check{U}))_{f_I}. 
    \end{equation}
    Moreover, we have the isomorphism of Corollary~\ref{Generic Comparison Intermediate}
    \begin{equation}\label{Main Theorem Eq5}
        \eta_{I-\gen}: \mathcal{O}(T^*(\check{G}/\check{U}))_{f_I} \simeq \mathcal{O}(T^*(\check{G}/(\check{U}, \psi_I)))_{f_I} \otimes_{(R_I)_{f_I}} (R_\emptyset)_{f_I}. 
    \end{equation}
    We can chain together \eqref{Main Theorem Eq3}, \eqref{Main Theorem Eq4}, and \eqref{Main Theorem Eq5} to obtain an isomorphism of $(R_\emptyset)_{f_I}$-modules
    \begin{equation}\label{Main Theorem Eq6}
         H_{T(\mathcal{O})}^*(\Gr_L, i_I^! \mathcal{F}_\reg)_{f_I} \simeq \mathcal{O}(T^*(\check{G}/(\check{U}, \psi_I)))_{f_I} \otimes_{(R_I)_{f_I}} (R_\emptyset)_{f_I}. 
    \end{equation}
    Recall that $R_I = (R_\emptyset)^{W_I}$, where $W_I$ is the Weyl group of $L$. By the equivariant formality of $i_I^!\mathcal{F}_\reg$ (Proposition~\ref{equivariantformality}), taking $W_I$-invariants in \eqref{Main Theorem Eq6} yields an isomorphism of $(R_I)_{f_I}$-modules
    \begin{equation}
        \Upsilon_{I-\gen}: H_{L(\mathcal{O})}^*(\Gr_L, i_I^! \mathcal{F}_\reg)_{f_I} \simeq \mathcal{O}(T^*(\check{G}/(\check{U}, \psi_I)))_{f_I}. 
    \end{equation}
    It remains to verify the commutativity of \eqref{Main Theorem compatiblity diagram}. We start by observing the commutativity of the following diagram
    \begin{equation}\label{Main Theorem Eq7}
    \begin{tikzcd}
        H_{T(\mathcal{O})}^*(\Gr_T, i_\emptyset^! \mathcal{F}_\reg)_{f_Ig_I} \ar[r, "(\Upsilon_{\emptyset})_{f_I g_I}"] \ar[d] & \mathcal{O}(T^*(\check{G}/\check{U}))_{f_Ig_I} \ar[d] \\
        H^*_{T(\mathcal{O})}(\Gr_G, \mathcal{F}_\reg)_{f_I g_I} \ar[r, "\eqref{ginzburgricheparabolicrestriction}"] & \mathcal{O}(T^*(\check{G}/(\check{U}, \psi)))_{f_I g_I} \otimes_{R_{f_I g_I}} (R_{\emptyset})_{f_I g_I}. 
    \end{tikzcd}
    \end{equation}
    Indeed, this diagram is nothing other than the commutative diagram \eqref{ginzburgrichecommutativediagram} of Theorem~\ref{grmaintheoremdetailed} characterizing the Ginzburg-Riche isomorphism $\Upsilon_\emptyset$, localized at $f_I g_I \in R_\emptyset$ (note the evident identification of $\mathcal{O}(T^*(\check{G}/(\check{U},\psi)))$ with $\mathcal{O}(\check{G}) \otimes R$ coming from Lemma~\ref{Reduction by Restriction}). On the geometric side, we have the commutative diagram
    \begin{equation}\label{Main Theorem Eq8}
    \begin{tikzcd}
        H_{T(\mathcal{O})}^*(\Gr_T, i_\emptyset^! \mathcal{F}_\reg)_{f_Ig_I} \ar[d] \ar[r] & H_{T(\mathcal{O})}^*(\Gr_L, i_I^! \mathcal{F}_\reg)_{f_I g_I} \ar[ld]\\
        H^*_{T(\mathcal{O})}(\Gr_G, \mathcal{F}_\reg)_{f_I g_I} .
    \end{tikzcd}
    \end{equation}
    obtained by localizing at $f_I g_I \in R_I$ the evidently commutative diagram 
    $$
    \begin{tikzcd}
        H_{T(\mathcal{O})}^*(\Gr_T, i_\emptyset^! \mathcal{F}_\reg) \ar[d] \ar[r] & H_{T(\mathcal{O})}^*(\Gr_L, i_I^! \mathcal{F}_\reg) \ar[ld]\\
        H^*_{T(\mathcal{O})}(\Gr_G, \mathcal{F}_\reg).
    \end{tikzcd}
    $$
    On the spectral side, we have the commutative diagram 
    \begin{equation}\label{Main Theorem Eq9}
    \begin{tikzcd}
        \mathcal{O}(T^*(\check{G}/\check{U}))_{f_Ig_I} \ar[d,"(\eta_{\emptyset-\gen}^*)_{f_I g_I}"] \ar[r,"(\eta_{I-\gen}^*)_{f_I g_I}"] & \mathcal{O}(T^*(\check{G}/(\check{U},\psi_I)))_{f_I g_I} \otimes_{(R_I)_{f_I g_I}} (R_\emptyset)_{f_I g_I} \ar[ld, "(\eta_{I-\gen}^- \times_{\check{\mathfrak{c}}_I} \check{\mathfrak{t}}^*)^*_{f_I g_I}"] \\
        \mathcal{O}(T^*(\check{G}/(\check{U},\psi)))_{f_I g_I} \otimes_{R_{f_I g_I}} (R_\emptyset)_{f_I g_I} 
    \end{tikzcd}
    \end{equation}
    obtained from the diagram of varieties 
    $$
    \begin{tikzcd}
        T^*(\check{G}/\check{U}) & T^*(\check{G}/(\check{U}, \psi_I)) \times_{\check{\mathfrak{c}}_I} \check{\mathfrak{t}}^* \ar[l, "\eta_{I-\gen}"]\\
        \left(T^*(\check{G}/(\check{U}, \psi)) \times_{\check{\mathfrak{c}}} \check{\mathfrak{c}}_I\right) \times_{\check{\mathfrak{c}}_I} \check{\mathfrak{t}}^* \ar[ur, "\eta_{I-\gen}^- \times_{\check{\mathfrak{c}}_I} \check{\mathfrak{t}}^*",swap] \ar[u, "\eta_{\emptyset-\gen}"]
    \end{tikzcd}
    $$
    by passing to rings of regular functions and localizing at $f_I g_I \in R_I$. Combining the diagrams \eqref{Main Theorem Eq7}, \eqref{Main Theorem Eq8}, and \eqref{Main Theorem Eq9} and then passing to $W_I$-invariants yields the diagram \eqref{Main Theorem compatiblity diagram}, which concludes the proof. 
\end{proof}
\emergencystretch=1em
\printbibliography

\end{document}